\documentclass[a4paper,USenglish,cleveref, autoref]{lipics-v2019}

\title{Factoring Polynomials over Finite Fields with Linear Galois Groups: An Additive Combinatorics Approach} 

\titlerunning{Factoring Polynomials over Finite Fields with Linear Galois Groups}

\author{Zeyu Guo\footnote{This work was done while the author was at the CSE department, IIT Kanpur.}}{Department of Computer Science, University of Haifa, Israel}{zguotcs@gmail.com}{https://orcid.org/0000-0001-7893-4346}{}
 
\authorrunning{Z. Guo}

\Copyright{Zeyu Guo}

\ccsdesc[300]{Mathematics of computing~Computations in finite fields}
\ccsdesc[300]{Mathematics of computing~Computations on polynomials}
\ccsdesc[100]{Mathematics of computing~Combinatoric problems}
\ccsdesc[100]{Computing methodologies~Algebraic algorithms}

\keywords{polynomial factoring, permutation group, finite field, algebraic combinatorics, additive combinatorics, derandomization}

\supplement{}

\funding{}

\acknowledgements{The author is grateful to Nitin Saxena for helpful discussions.}

\nolinenumbers 

\hideLIPIcs  

\EventEditors{Javier Esparza and Daniel Kr{\'a}l'}
\EventNoEds{2}
\EventLongTitle{45th International Symposium on Mathematical Foundations of Computer Science (MFCS 2020)}
\EventShortTitle{MFCS 2020}
\EventAcronym{MFCS}
\EventYear{2020}
\EventDate{August 24--28, 2020}
\EventLocation{Prague, Czech Republic}
\EventLogo{}
\SeriesVolume{170}
\ArticleNo{34}
 
\usepackage{tikz-cd}
\usepackage{mathtools}
\usepackage{enumerate}

\newcommand{\F}{\mathbb{F}}
\newcommand{\Z}{\mathbb{Z}}
\newcommand{\Q}{\mathbb{Q}}
\newcommand{\C}{\mathbb{C}}
\newcommand{\N}{\mathbb{N}}

\newcommand{\gl}{\mathrm{GL}}
\newcommand{\sym}{\mathrm{Sym}}

\newcommand{\gal}{\mathrm{Gal}}

\newcommand{\lin}{\mathrm{Lin}}

\newcommand{\M}{\mathcal{M}}
\newcommand{\B}{\mathcal{B}}
\newcommand{\nospacepunct}[1]{\makebox[0pt][l]{\,#1}}
\newcommand{\sg}[1]{\langle #1\rangle}
\newcommand{\E}{\mathop{\mathbb{E}}}

\makeatletter
\newcommand*{\rom}[1]{\expandafter\@slowromancap\romannumeral #1@}
\makeatother

\begin{document}

\maketitle
\begin{abstract}
Let $\tilde{f}(X)\in\Z[X]$ be a degree-$n$ polynomial such that $f(X):=\tilde{f}(X)\bmod p$ factorizes into $n$ distinct linear factors over $\F_p$.  We study the problem  of \emph{deterministically} factoring $f(X)$ over $\F_p$ given $\tilde{f}(X)$.
Under  the generalized Riemann hypothesis (GRH), we give an improved deterministic algorithm that computes the complete factorization of $f(X)$ in the case that the Galois group of $\tilde{f}(X)$ is (permutation isomorphic to) a \emph{linear group} $G\leq \gl(V)$ on the set $S$ of roots of $\tilde{f}(X)$, where $V$ is a finite-dimensional vector space over a finite field $\F$ and $S$ is identified with a subset of $V$. 
 In particular, when $|S|=|V|^{\Omega(1)}$, the algorithm runs in time polynomial in $n^{\log n/(\log\log\log\log n)^{1/3}}$ and the size of the input, improving Evdokimov's algorithm.
Our result also applies to a general Galois group $G$ when combined with a recent algorithm of the author. 

To prove our main result, we introduce a family of objects called \emph{linear $m$-schemes} and reduce the problem of factoring $f(X)$ to a combinatorial problem about these objects.
We then apply techniques from additive combinatorics to   obtain an improved bound. Our techniques may be of independent interest.
\end{abstract}

\section{Introduction}
Univariate polynomial factoring over finite fields is a fundamental problem in computer algebra, which has been extensively studied over the years. 
A longstanding open problem in this area is finding a \emph{deterministic} algorithm that factors a degree-$n$ polynomial $f(X)$ over a finite field $\F_q$  in time polynomial in $n$ and $\log q$.
There is a long list of work on this problem \cite{AMM77, Ber67, Ber70, Sch85, von87, Ron88, Ron89, Pil90, Sho90, Sho96, Hua91-2, Hua91, Evd92, Ron92, Evd94, CH00, Gao01, IKS09, Gua09, IKRS12, Aro13, AIKS14, BKS15}.
In particular, Berlekamp \cite{Ber70} gave a deterministic factoring algorithm that runs in time $\mathrm{poly}(n, \log q, \mathrm{char}(\F_q))$.
Building the work of R{\'o}nyai \cite{Ron88},  Evdokimov \cite{Evd94} gave a deterministic $\mathrm{poly}(n^{\log n},\log q)$-time algorithm under the generalized Riemann hypothesis (GRH).

Efforts were made to understand the combinatorics behind Evdokimov's algorithm \cite{CH00, Gao01}, culminating in the work \cite{IKS09} that proposed the notion of \emph{$m$-schemes} together with an  algorithm that subsumes those in \cite{Ron88, Evd94}. See also the follow-up work \cite{Aro13, AIKS14}. 
  An $m$-scheme, parameterized by $m\in\N^+$, can be seen as an extension of the notion of \emph{association schemes} in algebraic combinatorics. It was shown in \cite{IKS09} that whenever the  algorithm fails to produce a proper factorization of $f(X)$  in time $\mathrm{poly}(n^m, \log q)$, there always exists an $m$-scheme on $[n]$ satisfying strict combinatorial properties. Evdokimov's result can then be interpreted as the fact that such an $m$-scheme exists only for $m=O(\log n)$. Thus, one natural way of beating Evdokimov's $\mathrm{poly}(n^{\log n}, \log q)$-time algorithm is improving this $O(\log n)$ upper bound for $m$.  
  However, attempts of establishing an $o(\log n)$ upper bound for $m$ have be unsuccessful so far. Currently, the best  known general upper bound  is $m\leq c \log n+O(1)$, where $c=2/\log_2 12=0.557\dots$, proved in  \cite{Gua09} and independently in \cite{Aro13}. 
 
In another line of research \cite{Hua91-2,Hua91,Evd92,Ron92},  the finite field over which $f(X)$ is defined is assumed to be a prime field $\F_p$,  and   a \emph{lifted polynomial} of $f(X)$ is  assumed to be given, i.e., a degree-$n$ polynomial $\tilde{f}(X)\in\Z[X]$ satisfying $\tilde{f}(X)\bmod p=f(X)$.
In particular, Huang  \cite{Hua91-2,Hua91} proved that $f(X)\in\F_p[X]$ can be deterministically factorized  in polynomial time under GRH if the Galois group $G$ of   $\tilde{f}(X)$ is abelian. This was generalized in \cite{Evd92} to the case that $G$ is solvable.
 For general  $G$, R{\'o}nyai \cite{Ron92} gave a deterministic algorithm  under GRH that runs in time  polynomial in $|G|$ and the size of the input.
    
Recently, the author  \cite{Guo19, Guo17, Guo19-2} proposed a unifying approach for deterministic polynomial factoring over finite fields  based on the notion of \emph{$\mathcal{P}$-schemes}, where $\mathcal{P}$ is a collection of subgroups of the Galois group $G$ of $\tilde{f}(X)$. It was shown that above results  \cite{Ron88, Evd94, IKS09, Hua91-2,Hua91,Evd92,Ron92}  can be derived from this approach in a uniform way.  In particular, the results based on $m$-schemes \cite{IKS09} may be obtained using $\mathcal{P}$-schemes by assuming  $G$ to be the full symmetric group $\sym(n)$ (which is the most difficult case). When $G$ is less complex than a full symmetric group, the  approach based on $\mathcal{P}$-schemes may lead to better factoring algorithms by employing the structure of $G$. For example, a deterministic factoring algorithm  was given  in \cite{Guo19} (under GRH) whose running time is bounded in terms of the nonabelian composition factors of $G$. It runs in polynomial time when these nonabelian composition factors are all subquotients of $\sym(k)$ for $k=2^{O(\sqrt{\log n})}$. 

\subsection{Our Results.}

This paper is a continuation of the work in \cite{Guo19, Guo17, Guo19-2}.
We consider the problem of deterministically factoring $f(X)\in\F_p[X]$ given a lifted polynomial $\tilde{f}(X)\in\Z[X]$ of $f(X)$  whose Galois group is denoted by $G$.
 We want to apply the main result of \cite{Guo19} to  families of Galois groups that are less complex than full symmetric groups.
Natural candidates of such kinds of groups come from \emph{linear groups}, which are the main focus of this paper.
 
For example, suppose  the action of $G$ on the set of $n$ roots of $\tilde{f}(X)$ is permutation isomorphic to the action of $\gl(V)$ on $V\setminus\{0\}$, where  $V$ is a  finite-dimensional  vector space over a finite field $\F$.
We know Evdokimov's algorithm \cite{Evd94} factorizes $f(X)$ in time  $\mathrm{poly}(n^{\log n},\log p)$, whereas R{\'o}nyai's algorithm \cite{Ron92} runs in time polynomial in $|\gl(V)|=n^{\Theta(\dim V)}=n^{\Theta(\log n/\log |\F|)}$ and the size of the input.
When $|\F|=O(1)$, the latter time bound is still at least $\mathrm{poly}(n^{\log n},\log p)$. Can we factorize $f(X)$ in time polynomial in $n^{o(\log n)}$ and the size of the input? We answer this question affirmatively in this paper.

Let $S$ be a subset of a vector space $V$, and let $G$ be a permutation group on $S$. We say $G$ \emph{acts linearly} on $S$ if we can identify $G$ with a subgroup of $\gl(V)$ such that the  action of $G$ on $S$ is induced by the natural action of $\gl(V)$ on $V$.
Our main result states as follows:

\begin{theorem}\label{thm_factormain}
Under GRH, there exists a deterministic algorithm that, given $f(X)\in\F_p[X]$ that factorizes into $n$ distinct linear factors over $\F_p$, and a lifted polynomial $\tilde{f}(X)\in\Z[X]$ whose Galois group $G$ acts linearly on the set $S$ of roots of $\tilde{f}(X)$, where $S$ is identified with a subset of a vector space $V$ over a finite field $\F$, completely factorizes $f(X)$ over $\F_p$ in time polynomial in $n^m$ and the size of the input, where $m$ is an integer satisfying:
\begin{enumerate}[(1)]
\item $m=O(\log n)$ and $m\leq \dim \sg{S}_{\F}$, where $\sg{S}_{\F}\subseteq V$ is the subspace spanned by $S$ over $\F$.
\item $m=O\left(\frac{\log n}{(\rho^2 \log\log\log\log n)^{1/3}}\right)$, where $\rho:=\log |S|/\log |\sg{S}|$ and $\sg{S}\subseteq V$ is the abelian subgroup generated by $S$.
\end{enumerate}
\end{theorem}

Note $\dim \sg{S}_{\F}=\frac{\log |S|}{ \rho\log |\F|}=\frac{\log n}{\rho \log |\F|}$. Thus, the bound~(2) slightly improves (1) when both $\rho^{-1}$ and $|\F|$ are small enough.

 \begin{remark*}
The assumption that $f(X)$ factorizes into distinct linear factors over $\F_p$ is not essential. It can be removed if we replace the $\mathcal{P}$-scheme algorithm  \cite{Guo19} used in our proof by  the generalized  $\mathcal{P}$-scheme  algorithm in \cite[Chapter~5]{Guo17} which  works for arbitrary $f$.
 We also note that there exists a standard reduction in literature that reduces the problem of factoring a univariate polynomial  over a finite field  to the special case of factoring a  polynomial defined over a prime field $\F_p$  that factorizes into distinct linear factors over $\F_p$ \cite{Ber70, Yun76}.
\end{remark*}

\subparagraph*{General Galois groups.} Combining our techniques with \cite{Guo19-2},  we also obtain an improved algorithm that applies to \emph{any} finite Galois group $G$, whose running time is bounded in terms of the nonabelian composition factors of $G$. 

Specifically, two functions $d_\sym(m)$ and $d_\lin(m,q)$ are  introduced in \cite{Guo19-2}. These functions are further used to define quantities $N_\mathcal{A}(G)\in\N^+$ and  $N_\mathcal{C}(G)\in\N^+$ respectively for every finite group $G$. The following theorem is then proved in \cite{Guo19-2}.

\begin{theorem}[{\cite[Theorem~1.2]{Guo19-2}}]\label{thm_compfac}
Under GRH, there exists a  deterministic  algorithm that, given $f(X)\in\F_p[X]$ that factorizes into $n$ distinct linear factors over $\F_p$, and a lifted polynomial $\tilde{f}(X)\in\Z[X]$  with Galois group $G$, completely factorizes $f$ over $\F_p$ in time polynomial in $N_\mathcal{A}(G)$, $N_\mathcal{C}(G)$, and the size of the input.
\end{theorem}

Here  $N_\mathcal{A}(G)$ (resp. $N_\mathcal{C}(G)$)  measures the contribution from the alternating groups (resp. classical groups) among the nonabelian composition factors of $G$ to the running time. Using the bounds $d_\sym(m)=O(\log m)$ and $d_\lin(m,q)\leq m$, it is shown in \cite{Guo19-2} that $N_\mathcal{A}(G), N_\mathcal{C}(G)=k^{O(\log k)}$ if these alternating groups and classical groups are all (isomorphic to) subquotients of a symmetric group $\sym(k)$. In particular, choosing $k=n$ yields an $n^{O(\log n)}$-time deterministic algorithm under GRH, matching the state-of-the-art results \cite{Evd94, IKS09}.

In this paper, we obtain the following new bound for $d_\lin(m,q)$.

\begin{theorem}\label{thm_lin}
$d_\lin(m,q)=O\left(\frac{m\log q}{( \log\log\log(m\log q))^{1/3}}\right)$.
\end{theorem}

This bound is derived from Theorem~\ref{thm_mainbound} stated  below.
Its proof is deferred to Appendix~\ref{sec_pscheme}, where the definition of $d_\lin(m,q)$ is also given.

When $q$ is small, the bound in Theorem~\ref{thm_lin} is better than the bound $d_\lin(m,q)\leq m$.
It has the following implication, which states that the contribution $N_\mathcal{C}(G)$ from classical groups to the time complexity of the algorithm in Theorem~\ref{thm_compfac} is always subpolynomial in $n^{\log n}$. Thus, the contribution $N_\mathcal{A}(G)$ from alternating groups is the only bottleneck for obtaining an $n^{o(\log n)}$-time deterministic algorithm under GRH. 

\begin{corollary}\label{cor_bottleneck}
We have $N_\mathcal{C}(G)=n^{o(\log n)}$ in Theorem~\ref{thm_compfac}. Furthermore, if every alternating group among the composition factors of $G$ has degree $n^{o(1)}$, then the algorithm in Theorem~\ref{thm_compfac} runs in time polynomial in $n^{o(\log n)}$ and the size of the input.
\end{corollary}

 The proof of Corollary~\ref{cor_bottleneck} is deferred to  Appendix~\ref{sec_pscheme}.
 
 \subparagraph*{Realizing a Galois group over $\Q$.} Given the results above, it is a natural question to ask if a finite classical group $G$  (or a finite group $G$ that has large classical groups as composition factors) can indeed be realized as a Galois group over $\Q$. The problem of realizing a given group $G$ as a Galois group over $\Q$ is known as the \emph{inverse Galois problem} \cite{MM99}. While this problem is unsolved in general, many partial results are known. In particular, there are infinite families of finite classical groups that are  realizable over $\Q$. For example,  $\mathrm{PSL}_n(p)$ is realizable over $\Q$ for odd prime $p$ when $\gcd(n, p-1)=1$, $p>3$ and $p\not\equiv-1\pmod{12}$ \cite[Theorem~\rom{3}.6.8]{MM99}. See \cite[Section~\rom{3}.10.2]{MM99} for a summary about realizing finite simple  groups over $\Q$. These groups may also be used to build larger Galois groups via semidirect products or wreath products \cite{MM99}.
 
Furthermore, given a Galois extension $L/\Q$ with $\gal(L/\Q)=G$, we could realize any permutation representation  $G\to \sym(S)$ as follows: Let $H=G_x$ be a stabilizer for some $x\in S$, and let $K=L^H$, the fixed subfield of $H$.
Choose $\tilde{f}(X)\in\Z[X]$ to be the minimal polynomial of an integral primitive element  of $K$. Then the action of $G$ on the set of roots of $\tilde{f}$ in $L$ is permutation isomorphism to its action on $S$. 

Finally, by Chebotarev's density theorem \cite{Neu99}, there exist infinitely many primes $p$ such that $\tilde{f}(X)\bmod p$ factorizes into distinct linear factors, so that Theorem~\ref{thm_factormain} and Theorem~\ref{thm_compfac}  may apply.

\subsection{Proof Overview}

We give a high-level overview of the proof of Theorem~\ref{thm_factormain} in this subsection.

\subparagraph*{Linear $m$-schemes.} To prove Theorem~\ref{thm_factormain}, we  introduce a family of combinatorial objects called \emph{linear $m$-schemes}, which can be seen as the linear analogue of $m$-schemes studied in \cite{IKS09}.
For $m\in\N^+$ and a subset $S\subseteq V$, a linear $m$-scheme on $S$ is a collection $\Pi=\left\{\Pi^{(1)},\dots, \Pi^{(m)}\right\}$ of partitions satisfying a list of axioms, where   $\Pi^{(i)}$ is a partition of $S^i$  for $i\in [m]$ (see Definition~\ref{def_linm} for the formal definition).
We are  interested in a special kind of linear $m$-schemes called \emph{strongly antisymmetric linear $m$-schemes}.
In particular, we will prove the following  statement about these objects.

\begin{theorem}\label{thm_mainbound}
Let $V$ be a vector space over a finite field $\F$, $S\subseteq V$, $n=|S|$, and $\rho=\log |S|/\log |\sg{S}|$.
Suppose  $\Pi$  is a strongly antisymmetric linear $m$-scheme on $S$, and $\Pi^{(1)}$ is \emph{not} the finest partition of $S$. 
Then $m=O\left(\frac{\log n}{(\rho^2 \log\log\log\log n)^{1/3}}\right)$.
\end{theorem}

Moreover, we relate linear $m$-schemes to the notion of $\mathcal{P}$-schemes in \cite{Guo19}, which allows us to translate Theorem~\ref{thm_mainbound} into a statement about $\mathcal{P}$-schemes. Theorem~\ref{thm_factormain} then follows from the machinery developed in \cite{Guo19}.
As the general theory of $\mathcal{P}$-schemes is not the focus of this paper, we defer the derivation of  Theorem~\ref{thm_factormain} from Theorem~\ref{thm_mainbound} to the appendix (see Appendix~\ref{sec_pscheme}). The main text of this paper then focuses on Theorem~\ref{thm_mainbound}, which is a purely combinatorial statement.

\subparagraph*{Reducing the cardinality of sets by restricting to a fiber.} For  $B\subseteq S$, $B'\subseteq B\times B$ and $x\in B$, call $B'_x:=\{y\in S: (x,y)\in B'\}\subseteq B$ the \emph{$x$-fiber} of $B'$.
The combinatorics behind Evdokimov's algorithm \cite{Evd94} can be very roughly summarized as follows: The algorithm produces a partition $P$ of the set $S$, such that if $B\in P$ is not a singleton, we can find $B'\subseteq B\times B$ and $x\in B$ such that $1<|B'_x|<|B|/2$.
The algorithm then replaces $B$ by $B'_x$ and repeats. At each step, $|B|$ is reduced by at least a factor of two. So this process has at most $\log |B|\leq \log n$ steps, which gives the $O(\log n)$ upper bound for $m$. 
To prove the inequality $|B'_x|<|B|/2$,   Evdokimov crucially used the permutation $(\alpha,\beta)\mapsto (\beta,\alpha)$ of $S^2$, which can be seen as an element of the symmetric group $\sym(2)$. The algorithm in \cite{IKS09} based on $m$-schemes then upgraded this method by using permutations in $\sym(k)$ for $k\in [m]$.

Our analysis uses similar ideas. The main difference is that here the structure of linear Galois groups allows us to employ not only the permutations in $\sym(k)$ but also \emph{linear automorphisms}.
For example, when $k=2$, we will use not only the map $(\alpha,\beta)\mapsto (\beta,\alpha)$ but also maps of the form $(\alpha,\beta)\mapsto (a\alpha+b\beta, c\alpha+d\beta)$, where $a,b,c,d\in\F$.
This set of permutations forms a permutation group larger than $\sym(k)$. Because of the richer set of permutations, we are able to prove that on average, restricting to a fiber at each step reduces the cardinality of a set by a \emph{superconstant} factor.
This is summarized by Lemma~\ref{lem_fastshrink} (the Key Lemma) from which the $o(\log n)$ bound in Theorem~\ref{thm_mainbound} follows.
  
\subparagraph*{Additive combinatorics.} 
Our proof of  Lemma~\ref{lem_fastshrink} heavily uses tools from additive combinatorics.  These tools seem very useful for studying linear $m$-schemes as they apply to ``soft'' combinatorial objects like subsets and partitions while also capturing the rigid abelian group structure of $V$.
Specifically, our analysis for a subset $B\subseteq S$ is divided into the following three cases, depending on how large   $B+B$ is compared with $B$ and $B\times B$: 
\begin{enumerate}
\item $|B|\ll |B+B|\ll |B|^2$. In this case, we show that if $K|B|\leq |B+B|\leq  |B|^2/K$ for some factor $K$, then restricting to a fiber at each step reduces $|B|$ by a factor of $K^{\Omega(1)}$. 
\item $|B+B|/|B|$ is small. This is the most difficult case and the proof becomes rather technical. In particular, we will prove a ``decomposition theorem'' using Fourier analysis. Due to the page limit, we defer the analysis for this case to the appendix. 
\item $|B|^2/|B+B|$ is small. This happens only when the ``entropy rate'' $\rho(B):=\log |B|/\log|\sg{B}|$ is low ($\lessapprox 1/2$). We reduce this case to the previous two cases by replacing $B$ with a partial sumset $B'\subseteq kB$ for some integer $k>1$, which increases the entropy rate.
\end{enumerate}

\section{Notations and Preliminaries}\label{sec_pre}

 Let $\N:=\{0,1,2,\dots \}$ and $\N^+:=\{1,2,\dots\}$.
Let $[k]:=\{1,2,\dots,k\}$.    Write $\log$ for base 2 logarithms.
Denote by $A \setminus B$ the set difference $\{x: x\in A ~\text{and}~ x\not\in B\}$.
The cardinality of a set $S$ is  $|S|$. Alternatively, we write $\#\{\,\cdots\}$ for the cardinality of a set $\{\,\cdots\}$.  
The restriction of a map $f: S\to S'$ to a subset $T\subseteq S$ is denoted by $f|_T$.

A \emph{partition}  of a set $S$ is a set $P$ of nonempty subsets of $S$ satisfying $S=\coprod_{B\in P} B$, where $\coprod$ denotes the disjoint union.   
Each $B\in P$ is called a \emph{block}  of $P$.  
For $T\subseteq S$ and a partition $P$ of $S$, the set $P|_T:=\{B\cap T: B\in P\} \setminus \{\emptyset\}$ is a partition of $T$, called the \emph{restriction} of $P$ to $T$. 
Denote by $\infty_S$ the finest partition of $S$, i.e.,  $\infty_S=\{\{x\}: x\in S\}$.  
For a set $P$ of subsets of $S$, define $\B(P)$ to be the set of subsets of $S$ that are unions of sets in $P$. 


\subparagraph*{Additive combinatorics.} Suppose $V$ is a vector space over a field $\F$. For $A,B\subseteq V$, define $A+B:=\{a+b:a\in A, b\in B\}$ and $A-B:=\{a-b:a\in A, b\in B\}$. For $k\in \N^+$, write $kA$ for $\underbrace{A+A+\dots+A}_{k~\text{times}}$.
Write $\sg{A}$ for the abelian subgroup of $V$ generated by $A$. 
For $A, B\subseteq V$, define $\mu_B(A)$ to be the density of $A$ in $B$, i.e., $\mu_B(A):=|A\cap B|/|B|$.
Write $\mu(A)$ for $\mu_{\sg{A}}(A)$.
Clearly, if $|\sg{A}|/|A|$ is small, so is $|A+A|/|A|$.
The inverse of this fact is the content of the \emph{Freiman--Ruzsa Theorem} \cite{Ruz99}.
We need the following version of this theorem.

\begin{theorem}[Freiman--Ruzsa Theorem  \cite{Eve12, EL14}]\label{thm_fr}
Let $V$ be a vector space over a prime finite field $\F_\ell$. Suppose $A\subseteq V$ satisfies  $|A+A|\leq K |A|$ for some $K>0$.
Then $|\sg{A}|\leq \ell^{2K}|A|$.
\end{theorem}

We also need  \emph{Pl\" unnecke's inequality}:

\begin{theorem}[{Pl\" unnecke's inequality \cite[Corollary~6.28]{TV06}}]\label{thm_plunnecke}
Suppose $A, B\subseteq V$ satisfies $|A+B|\leq K |A|$ for some $K>0$. Then $|kB|\leq K^k|A|$ for $k\in\N^+$.
\end{theorem}
 
\section{Introducing Linear $m$-schemes}\label{sec_lsch}

Let $V$ be a finite-dimensional vector space  over a finite field $\F$.  For $k,k'\in\N^+$, denote by $\M_{k,k'}(\F)$ the set of linear maps $\tau:V^k\to V^{k'}$ of the form
\[
\mathbf{x}=(x_1,\dots,x_k)\mapsto \left(\sum_{i=1}^k c_{i,1} x_i, \dots, \sum_{i=1}^k c_{i, k'} x_i\right), \quad\text{where}~c_{i,j}\in\F,
\]
i.e., each coordinate of $\tau(\mathbf{x})\in V^{k'}$ is a linear combination of the coordinates of $\mathbf{x}\in V^k$ over $\F$. In most cases, the base field $\F$ is clear from the context  and we simply write $\M_{k,k'}$ for $\M_{k,k'}(\F)$.

The following  special  maps in $\M_{k, 1}$ will be used in the paper.
 \begin{definition}[projection and summation]\label{defi_proj}
 For $k\in \N^+$ and $i\in [k]$, write $\pi_{k,i}:V^k\to V$ for the projection of $V^k$ to its $i$th coordinate, and write $\sigma_k :V^k\to V$ for the map sending $(x_1,\dots,x_k)\in V^k$ to $x_1+x_2+\dots+x_k$.  We have $\pi_{k,i},\sigma_k\in \M_{k,1}$ for $k\in \N^+$ and $i\in [k]$.
 \end{definition}
 
Now we are ready to define the notion of \emph{linear $m$-schemes}.

\begin{definition}[linear $m$-scheme]\label{def_linm}
Let $m\in\N^+$ and $S\subseteq V$. Let 
$\Pi=\left\{\Pi^{(1)},\dots, \Pi^{(m)}\right\}$,
 where $\Pi^{(k)}$ is a partition of $S^k$ for $k\in [m]$.
We say $\Pi$ is a \emph{linear $m$-scheme} on $S$ if for $k, k'\in [m]$, $B\in \Pi^{(k)}$, $B'\in \Pi^{(k')}$, and $\tau\in \M_{k,k'}$, we have
\begin{description}
\item[(P1):] Either $\tau(B)=B'$ or $\tau(B)\cap B'=\emptyset$.
\item[(P2):] $\#\{x\in B: \tau(x)=y\}$ is constant when $y$ ranges over $B'$.
\end{description}
\end{definition}

Definition~\ref{def_linm} can be viewed as a linear analogue of $m$-schemes in \cite{IKS09}. In fact, it is not hard to show that a linear $m$-scheme on a set $S$ always induces an $m$-scheme on $S$.

The following lemma states that the coordinates of elements in the same block  of a linear $m$-scheme  always satisfy the same   linear relations. Its proof can be found in the appendix.

\begin{lemma}\label{lem_linrel}
 Let  $\Pi$ be a  linear $m$-scheme on $S$. For $k\in [m]$, $B\in \Pi^{(k)}$ and $\mathbf{x}=(x_1,\dots,x_k), \mathbf{y}=(y_1,\dots,y_k)\in B$, the coordinates $x_i$ satisfy a linear relation $\sum_{i=1}^k c_i x_i=0$ iff  the coordinates $y_i$ satisfy the same relation, i.e., $\sum_{i=1}^k c_i y_i=0$.
 \end{lemma}

\subparagraph*{Strong antisymmetry.} We are interested in a special kind of linear $m$-schemes called \emph{strongly antisymmetric linear $m$-schemes}.
 
\begin{definition}[strong antisymmetry]\label{defi_sas}
Let $\Pi$ be a linear $m$-scheme. 
 Define
\[
\M_\Pi:=\left\{\tau|_B: B\to B'\,\middle\vert\, 
\begin{array}{l} 
k,k'\in [m], B\in \Pi^{(k)}, B'\in \Pi^{(k')},\\ \tau\in \M_{k,k'}, \tau \text{ maps $B$ bijectively to $B'$}
\end{array}\right\}.
\]
Define $\widetilde{\M}_\Pi$ to be the set of all possible compositions of the maps $\tau\in \M_\Pi$ as well as their inverses $\tau^{-1}$.
We say $\Pi$ is \emph{strongly antisymmetric} if for $k\in [m]$ and $B\in\Pi^{(k)}$, $\widetilde{\M}_\Pi$ does not contain a nontrivial permutation of $B$.
 \end{definition}

 \subsection{Basic Facts about Linear $m$-schemes}
 
 In this subsection, we list some basic facts about  linear $m$-schemes. Most proofs are omitted due to the page limit and can be found in the appendix.
 
\subparagraph*{Closedness of sets $\B(\Pi^{(k)})$.}
 
Recall that for a set $P$ of subsets of $S$, we define $\B(P)$ to be the set of subsets of $S$ that are unions of sets in $P$.  The following lemma states that for a linear $m$-scheme $\Pi$, the sets $\B(\Pi^{(k)})$ are closed under various operations.
 
 \begin{lemma}\label{lem_closedness}
Let $\Pi$ be a linear $m$-scheme on $S\subseteq V$. We have:
 \begin{enumerate}
 \item For $k\in [m]$,  $\B(\Pi^{(k)})$ is  closed under union, intersection, and complement in $S^k$.
 \item Let $k,k'\in [m]$ such that $k+k'\leq m$. Let $B\in \B(\Pi^{(k+k')})$.  
 Let $Q$ be a quantifier of the form $\exists$, $\forall$, or $\exists_{=t}$ (which means ``there exist exactly $t$ elements'').
 Let $B_Q$ be the set of $x\in S^k$ satisfying the condition ``$Q\,y\in S^{k'}: (x,y)\in B$''.
 Then $B_Q\in \B(\Pi^{(k)})$.
  \item Let  $k,k'\in [m]$, $B\in \B(\Pi^{(k)})$, and $\tau\in \M_{k,k'}$. Then $\tau(B)\cap S^{k'}\in \B(\Pi^{(k')})$.
  \item Let  $k,k'\in [m]$, $B\in \B(\Pi^{(k')})$, and $\tau\in \M_{k,k'}$. Then $\tau^{-1}(B)\cap S^{k}\in \B(\Pi^{(k)})$.
 \end{enumerate}
 \end{lemma}
 
 \subparagraph*{Recursive structure of linear $m$-schemes.}

Next, we show that linear $m$-schemes have a recursive structure. Namely, for $t\in [m-1]$, each ``fiber'' of a linear $m$-scheme with respect to the projection to the first $t$ coordinates is a linear $(m-t)$-scheme. 

\begin{definition}\label{defi_fixing}
Let $\Pi$ be a linear $m$-scheme on $S\subseteq V$. Let $t\in [m-1]$ and $x=(x_1,\dots,x_t)\in S^t$.  Define 
$
\Pi_x=\left\{\Pi_x^{(1)},\dots, \Pi_x^{(m-t)}\right\}
$, 
where for $k\in [m-t]$, $\Pi_x^{(k)}$ is the partition of $S^k$ such that two elements $y, z\in S^k$ are in the same block of $\Pi_x^{(k)}$ iff $(x,y), (x,z)\in S^{t+k}$ are in the same block of $\Pi^{(t+k)}$.
Also write $\Pi_{x_1,\dots,x_t}$ for $\Pi_x$.
\end{definition}


\begin{lemma}\label{lem_fixing}
$\Pi_x$   in Definition~\ref{defi_fixing} is a linear $(m-t)$-scheme on $S$.
Moreover, if $\Pi$ is strongly antisymmetric, so is $\Pi_x$.
\end{lemma}

We also have the following easy observation.

\begin{lemma}\label{lem_refine}
Let $\Pi$ and $\Pi_x$ be as in Definition~\ref{defi_fixing}.
Then $\B(\Pi^{(1)})\subseteq \B(\Pi_x^{(1)})$, i.e., the partition $\Pi_x^{(1)}$  refines $\Pi^{(1)}$.
\end{lemma}
\begin{proof}
As $\Pi_{x_1,\dots,x_t}=(\Pi_{x_1,\dots,x_{t-1}})_{x_t}$, it suffices to prove the claim for $t=1$. The claim follows by noting that $B\times B\in \B(\Pi^{(2)})$ for $B\in \B(\Pi^{(1)})$.
\end{proof}
 
\subparagraph*{Basic upper bounds for $m$.} Next, we give the following basic upper bounds for $m$ when $\Pi$ is a strongly antisymmetric linear $m$-scheme satisfying $\Pi^{(1)}\neq\infty_S$.
 
 \begin{lemma}\label{lemma_bounddim}
Suppose  $\Pi$  is a strongly antisymmetric linear $m$-scheme on $S\subseteq V$, where $|S|=n$, and $B\in \Pi^{(1)}$ is not a singleton.
Denote by $\sg{S}_{\F}$ the subspace of $V$ spanned by $S$ over $\F$. 
Then (1) $m<\dim\sg{S}_{\F}$ and (2) $m\leq\log |B|\leq \log n$.
\end{lemma}

 \section{Proof of Theorem~\ref{thm_mainbound}}\label{sec_main}
 
In the rest of the paper,  $\Pi$ is assumed to be a strongly antisymmetric linear $m$-scheme on $S\subseteq V$, where $V$ is a finite-dimensional vector space over a finite field $\F$. Let  $n:=|S|$, $\rho:=\log |S|/\log |\sg{S}|$, and $\ell:=\mathrm{char}(\F)$.

\subparagraph*{Assumptions.} Throughout the analysis, we make the following assumptions: Assume $n\geq C$ for some sufficiently large constant $C$.
Also assume $\rho^2 \log\log\log\log n>1$,  since otherwise Theorem~\ref{thm_mainbound} holds by  Lemma~\ref{lemma_bounddim}\,(2).

 In addition, we assume $\F$ is a \emph{prime field}, which  can be justified as follows:
Note that $V$, as a vector space over $\F$, may be identified with a vector space over $\F_\ell$. Under this identification, we have $\M_{k,k'}(\F_\ell)\subseteq \M_{k,k'}(\F)$ for $k,k'\in [m]$, because linear combinations of the $k$ coordinates of $\mathbf{x}\in V^k$ over $\F_\ell$ are also linear combinations of these coordinates over $\F$. This means if $\Pi$ is a strongly antisymmetric linear $m$-scheme over $\F$, then it remains so over $\F_\ell$.
Therefore, it suffices to prove Theorem~\ref{thm_mainbound} for the case $\F=\F_\ell$.

Because of the assumption that $\F$ is a prime field, the abelian group $\sg{S}$ and the $\F$-subspace $\sg{S}_\F$ spanned by $S$  coincide. They are used interchangeably from now on.

Finally, assume $\log \ell\leq (\rho^{-1} \log\log\log\log n)^{1/3}\leq  (\log\log\log\log n)^{1/2}$, since otherwise 
$\dim\sg{S}_\F=\log_\ell n^{1/\rho}\leq \frac{\log n}{(\rho^2 \log\log\log\log n)^{1/3}}$ and  Theorem~\ref{thm_mainbound} holds by  Lemma~\ref{lemma_bounddim}\,(1).

\subparagraph*{Reduction to the Key Lemma.}

The following lemma is the key in the proof of Theorem~\ref{thm_mainbound}.

\begin{lemma}[Key Lemma]\label{lem_fastshrink}
Suppose  $B\in\Pi^{(1)}$ has cardinality at least $n^{1/(\rho^2 \log\log\log\log n)^{1/3}}$, and $m\geq (\log\log n)^2$.
Then there exist $k\in [m-2]$, $x_1,\dots,x_k\in B$, and $B'\in \B(\Pi_{x_1,\dots,x_k}^{(1)})$
such that $B'\subseteq B$ and $\min\{|B'|, |B|/|B'|\}\geq 2^{C k (\rho^2 \log\log\log\log n)^{1/3}}$ for some constant $C>0$.
\end{lemma}

We first use Lemma~\ref{lem_fastshrink} to prove a very similar lemma below, which shows that on average, replacing $\Pi$ by $\Pi_x$ at each step reduces the cardinality of blocks by a superconstant factor.

\begin{lemma}\label{lem_fastshrink2}
Suppose  $B\in\Pi^{(1)}$, $|B|\geq n^{1/(\rho^2 \log\log\log\log n)^{1/3}}>1$, and  $m\geq (\log\log n)^2$.
Then there exist $k\in [m-2]$, $x_1,\dots,x_k\in B$, and $B'\in \Pi_{x_1,\dots,x_k}^{(1)}$
such that $B'\subsetneq B$, $|B'|>1$, and  $|B|/|B'|\geq 2^{ C k (\rho^2 \log\log\log\log n)^{1/3}}$  for some constant $C>0$.
\end{lemma}

Due to the page limit, we omit the derivation of Lemma~\ref{lem_fastshrink2} from Lemma~\ref{lem_fastshrink} and defer it to the appendix.
Theorem~\ref{thm_mainbound} now follows from Lemma~\ref{lem_fastshrink2} and a simple induction:

\begin{proof}[Proof of Theorem~\ref{thm_mainbound}]
If $m<(\log\log n)^2$ then we are done. So assume $m\geq(\log\log n)^2$.
Let $C>0$ be the constant in Lemma~\ref{lem_fastshrink2}.
Let $t:=1/(\rho^2 \log\log\log\log n)^{1/3}$.
Choose $B\in \Pi^{(1)}$ such that $|B|>1$. 
We claim 
\[
m\leq C^{-1}t \log |B| +t \log n =O\left(\frac{\log n}{(\rho^2 \log\log\log\log n)^{1/3}}\right).
\]
Induct on $|B|$.
If $|B|<n^{t}$, we have $m\leq \log |B|\leq t \log n$ by Lemma~\ref{lemma_bounddim}\,(2). So the claim holds in this case.
Now assume $|B|\geq n^{t}$.
Then we can choose $k\in [m-2]$, $x_1,\dots,x_k\in B$, and  $B'\in \Pi_{x_1,\dots,x_k}^{(1)}$ as in Lemma~\ref{lem_fastshrink2}.
By Lemma~\ref{lem_fixing}, $ \Pi_{x_1,\dots,x_k}$ is a strongly antisymmetric $(m-k)$-scheme.
By the induction hypothesis, we have $m-k\leq C^{-1} t \log |B'|+t \log n$. The claim then follows from the inequality $|B'|\leq 2^{-C k/t} |B|$.
\end{proof}

So it remains to prove Lemma~\ref{lem_fastshrink}. We divide its proof into three cases: (1) $|B|\ll |B+B|\ll |B|^2$, (2) $|B+B|/|B|$ is small, and (3) $|B|^2/|B+B|$ is small.

\subsection{The Case $|B|\ll |B+B|\ll |B|^2$}\label{subsec_case1}

We first prove Lemma~\ref{lem_fastshrink} for the case $|B|\ll |B+B|\ll |B|^2$. 
To see the intuition, consider $x,y\in B$. The set $B\cap (x+y-B)=\{z\in B: x+y-z\in B\}$ is in $\B(\Pi_{x,y}^{(1)})$, since $\{(x,y,z)\in S^3: z\in B, x+y-z\in B\}\in \B(\Pi^{(3)})$ by Lemma~\ref{lem_closedness}\,(1) and (4).
Moreover, $B\cap (x+y-B)$ maps bijectively to $\{(z,w)\in B\times B: z+w=x+y\}$ via $z\mapsto (z, x+y-z)$. Therefore, 
\[
|B\cap (x+y-B)|=\#\{(z,w)\in B\times B: z+w=x+y\}.
\]
On the other hand, we have 
\[
\sum_{t\in B+B} \#\{(z,w)\in B\times B: z+w=t\}=|B\times B|=|B|^2.
\]
Let us   pretend that the sets $\{(z,w)\in B\times B: z+w=t\}$ have equal size for all $t\in B+B$.
Then we may choose $B'=B\cap (x+y-B)$ for arbitrary $x,y\in B$, whose cardinality is $|B'|=|B|^2/|B+B|$.
As $|B|\ll |B+B|\ll |B|^2$, both $|B'|$ and $|B|/|B'|$ are large, as required by Lemma~\ref{lem_fastshrink}.

In general, the sets  $\{(z,w)\in B\times B: z+w=t\}$ may have very different sizes. Still, we manage to  prove that if $K|B|\leq |B+B|\leq |B|^2/K$ holds for some $K\geq 4$, then there exist $x,y\in B$ and a subset $B'$ of $B$ in $\B(\Pi_{x,y}^{(1)})$ such that $|B'|, |B|/|B'|\geq K^{1/2}$. 
 In fact, in order to later  extend the analysis to the case that $|B|^2/|B+B|$ is small, we prove the result in the following more  general  form.

\begin{lemma}\label{lem_shrinkweak}
Let $B\in \Pi^{(1)}$ and $k\in\N^+$. Suppose $m\geq 2k+2$.
Let $A$ be a block in $\Pi^{(k)}$ contained in $B^k$, and let $A'=\sigma_k(A)$ (see Definition~\ref{defi_proj}).
Suppose  $K |A'|\leq |A'+B|\leq |A'| |B|/K$ for some $K\geq 4$. 
Then there exist $x_1,\dots,x_{k+1}\in B$ and $B'\in \B(\Pi_{x_1,\dots,x_{k+1}}^{(1)})$  such that $B'\subseteq B$ and $\min\{|B'|, |B|/|B'|\}\geq K^{1/2}$.
\end{lemma}

In particular, by choosing $k=1$ and $A=A'=B$, we see that Lemma~\ref{lem_fastshrink} holds when $K|B| \leq |B+B|\leq |B|^2/K$ for some $K=2^{\Omega((\rho^2 \log\log\log\log n)^{1/3})}\geq 4$.

\begin{proof}[Proof of Lemma~\ref{lem_shrinkweak}]
For $z\in A'+B$, define $\nu^+(z):=\#\{(x,y)\in A'\times B: x+y=z\}$.
First assume there exists an element $z\in A'+B$ such that $\nu^+(z)\in [K^{1/2},  |B|/K^{1/2}]$.
Fix such $z$.
 Choose $(x_1,\dots,x_k)\in A$ and $x_{k+1}\in B$ such that $x_1+\dots+x_{k+1}=z$.
Let  $T=\{y\in B: z-y\in A'\}$. Note  $y\mapsto (z-y,y)$ is a one-to-one correspondence between $T$ and $\{(x,y)\in A'\times B:x+y=z\}$.
So $|T|=\nu^+(z)\in [K^{1/2}, |B|/K^{1/2}]$.
We also have
\[
T=\{y\in B: \exists\,(x_1',\dots,x_k')\in A: x_1'+\dots+x_k'+y=x_1+\dots+x_{k+1}\}
\]
which is in $\B(\Pi_{x_1,\dots,x_{k+1}}^{(1)})$ by Lemma~\ref{lem_closedness}.
Choosing $B'=T$ proves the lemma.

So we may assume $\nu^+(z)\not\in [K^{1/2},  |B|/K^{1/2}]$ for $z\in A'+B$.
Define 
\[
Z:=\{z\in A'+B: \nu^+(z)\leq |B|/K^{1/2}\}=\{z\in A'+B: \nu^+(z)<K^{1/2}\}.
\]
As $\sum_{z\in A'+B} \nu^+(z)=|A'||B|$, the number of $z\in A'+B$ satisfying $\nu^+(z)>|B|/K^{1/2}$ is less than $K^{1/2} |A'|$. So we have 
\[
|Z|>|A'+B|-K^{1/2}|A'|\geq K|A'|-K^{1/2}|A'|\geq K^{1/2}|A'|
\]
where the last inequality holds since $K\geq 4$.

For $x\in A'$, define $Z_x:=\{y\in B: x+y\in Z\}$. Then $Z=\bigcup_{x\in A'} (x+Z_x)$.
Therefore,
\begin{equation}\label{eq_lb}
\sum_{x\in A'} |Z_x|\geq |Z|\geq K^{1/2}|A'|.
\end{equation}

On the other hand, we have
\begin{equation}\label{eq_ub}
\begin{aligned}
\sum_{x\in A'} |Z_x|&=\#\{(x,y)\in A'\times B: x+y\in Z\}=\sum_{z\in Z}\#\{(x,y)\in A'\times B: x+y=z\}\\
&=\sum_{z\in Z} \nu^+(z)\leq K^{1/2} |A'+B|\leq |A'||B|/K^{1/2}.
\end{aligned}
\end{equation}

 
\begin{claim}\label{claim_constantz}
$|Z_x|$ is constant when $x$ ranges over $A'$.
\end{claim}

\begin{claimproof}[Proof of Claim~\ref{claim_constantz}]
For $t\in\N$, let $A_t$ be the set of $y \in A$ such that there exist precisely $t$ elements $x\in  A$ satisfying $\sigma_k(x)=\sigma_k(y)$. Then $A_t\in\B(\Pi^{(k)})$ for $t\in\N$ by Lemma~\ref{lem_closedness}. 
Also note $ A=\bigcup_{t\in\N} A_t$. As $ A\in \Pi^{(k)}$, we have $A=A_{t_0}$ for some $t_0\in \N$. This means for all $z\in A'$, there exist precisely $t_0$ elements $x\in A$ satisfying $\sigma_k(x)=z$.

For $t\in\N$, denote by $X_t$ the set of $(z,w)\in A\times B$ such that there exist precisely $t$ elements $(x,y)\in A'\times B$ satisfying $x+y=\sigma_k(z)+w$, or equivalently,
there exist precisely $t t_0$ elements $(x,y)\in A\times B$ satisfying $\sigma_k(x)+y=\sigma_k(z)+w$.
The latter characterization shows $X_t\in\B(\Pi^{(k+1)})$ for $t\in\N$ by Lemma~\ref{lem_closedness}. 
Let $Z'=\bigcup_{t\in\N: t<K^{1/2}} X_t\in \B(\Pi^{(k+1)})$.
Then $(x,y)\in A\times B$ is in $Z'$ iff $\sigma_k(x)+y$ is in $Z$.

For $t\in\N$, denote by $Y_t$ the set of $x\in A$ such that there exist precisely $t$ elements $y\in B$ satisfying $(x,y)\in Z'$, or equivalently, $\sigma_k(x)+y\in Z$.
 We have $Y_t\in\B(\Pi^{(k)})$ for $t\in\N$ by Lemma~\ref{lem_closedness}. 
Also note $A=\bigcup_{t\in\N} Y_t$. As $A\in \Pi^{(k)}$, we have $A=Y_{t_1}$ for some $t_1\in\N$. So for all $x\in A$, there exist precisely $t_1$ elements $y\in B$ such that $\sigma_k(x)+y\in Z$, i.e., $|Z_{\sigma_k(x)}|=t_1$. As $A'=\sigma_k(A)$, this proves the claim.
\end{claimproof}

By \eqref{eq_lb}, \eqref{eq_ub} and  Claim~\ref{claim_constantz}, we have $K^{1/2}\leq|Z_x|\leq |B|/K^{1/2}$ for all $x\in A'$.
Choose arbitrary $x=(x_1,\dots,x_k)\in A$ and $x_{k+1}\in B$, and let $x'=\sigma_k(x)\in A'$. Note $Z_{x'}=\{y\in B:x'+y\in Z\}=\{y\in B: (x,y)\in Z'\}$ is in $\B(\Pi_{x_1,\dots,x_k}^{(1)})\subseteq \B(\Pi_{x_1,\dots,x_k,x_{k+1}}^{(1)})$.
The lemma follows by choosing $B'=Z_{x'}$.
\end{proof}

\subsection{The Case $|B+B|/|B|$ is small}

Next, we address the case that $|B+B|/|B|$ is small. This   is equivalent to  $\mu(B)^{-1}=|\sg{B}|/|B|$ being small by the Freiman-Ruzsa Theorem (Theorem~\ref{thm_fr}).
Our main result in this case is the following lemma.

\begin{lemma}\label{lem_densecase}
Let $N\geq c$ such that $\log \ell\leq  (\log\log\log\log N)^{1/2}$, where $c>0$ is a sufficiently large constant.
Suppose $B\in\Pi^{(1)}$, $|B+B|/|B|\leq (\log\log\log N)^{1/2}$, and $m, |B|\geq \log\log N$.
Then there exist $k=O(\log\log\log N)$, $x_1,\dots,x_k\in B$, and $B'\in \B(\Pi_{x_1,\dots,x_k}^{(1)})$ such that $B'\subseteq B$ and  
$2^{C k\log\log\log\log N  }\leq |B|/|B'|\leq 2^{(\log N)^{1/2}}$
for some constant $C>0$.
\end{lemma}

Due to the page limit, we defer the proof of Lemma~\ref{lem_densecase} to the appendix (see Appendix~\ref{sec_dense}). Here we only sketch the main ideas: 
Observe that the argument in Subsection~\ref{subsec_case1} does not directly apply  since  $B$ is dense in $\sg{B}$ and therefore $|B+B|/|B|$ is small. So our first step is reducing the density of $B$.
Roughly speaking, we show that restricting to a fiber of $\Pi$ (i.e., replacing $\Pi$ by $\Pi_x$ for some $x\in B$) each time reduces not only the cardinality of a block but also its \emph{density} by at least a constant factor. By repeatedly restricting to fibers $k$ times for some $k=\omega(1)$, we reduce the density of a block to $\exp(-k)$. Then we manage to prove Lemma~\ref{lem_densecase} by repeatedly applying an argument similar to that in Subsection~\ref{subsec_case1} to blocks that are already sparse enough. 

The actual proof is much more complicated than the above sketch due to many technical issues that we need to solve, and we refer the reader to the appendix for the full details.
For example, one issue is that replacing $B\in\Pi^{(1)}$ by a subset $B'\in\Pi_x^{(1)}$, while always reducing the cardinality, may actually increase the density (i.e., $\mu(B')>\mu(B)$). We observe that this happens only when $B$ is ``overrepresented'' in the subspace $\sg{B'}$. To solve this problem, we find a small collection of subspaces $W_i\subseteq \sg{B}$ such that $B$ becomes ``pseudorandom'' within each $W_i$, which ensures that  overrepresentation never occurs within  $W_i$. We state this as the \emph{decomposition theorem} (Theorem~\ref{thm_struct}). The actual proof of Lemma~\ref{lem_densecase} then uses the density $\mu_{W_i}(B)$ of $B$ in some $W_i$ instead of  $\mu(B)$.

\subsection{The Case $|B|^2/|B+B|$ is small}\label{subsec_case3}

Finally, we address the case that $|B|^2/|B+B|$ is small  and finish the proof of Lemma~\ref{lem_fastshrink}.
When $|B|^2/|B+B|$ is small, the argument in Subsection~\ref{subsec_case1} does not directly apply since there are not enough linear dependencies of the form $a+b=c+d$ with $a,b,c,d\in B$. To solve this problem, we first find a partial sumset $A'=\sigma_k(A)$ for some $A\subseteq B^k$, where $k\in\N^+$ is small, such that either $|A'+A'|/|A'|$ is small or Lemma~\ref{lem_fastshrink} already holds. 

For $B\subseteq V$, define $\rho(B):=\log|B|/\log|\sg{B}|$. Then we have

 \begin{lemma}\label{lem_red1}
  Let $K\geq 4$.
 Suppose   $B\in\Pi^{(1)}$ has cardinality at least $2K^{2}$ and  $m> 4/\rho(B)+\log K+1$.
 Then one of the following is true:
\begin{enumerate}
\item There exist $1\leq k\leq 2/\rho(B)+1$, $x_1,\dots,x_k\in B$, and $B'\in \B(\Pi_{x_1,\dots,x_k}^{(1)})$ such that $B'\subseteq B$ and  $\min\{|B'|, |B|/|B'|\}\geq K^{1/2}$.
\item There exist $1\leq k\leq 2/\rho(B)$ and $A\in \Pi^{(k)}$ such that  $A\subseteq B^k$, $|A|\geq  |B|^k/K^{(k-1)/2}$, $|A'+A'|\leq K^{2k}|A'|$, and $\sigma_k|_A: A\to A'$ is bijective, where $A':=\sigma_k(A)$. 
\end{enumerate}
 \end{lemma}
 
The proof of Lemma~\ref{lem_red1} uses the following lemma, whose proof is deferred to the appendix.

\begin{lemma}\label{lem_bij}
Let $k\in [m]$, $A\in \Pi^{(k)}$ and $A'=\sigma_k(A)$. 
Suppose $m\geq 2k$ and $m> k+\log(|A|/|A'|)$. Then $\sigma_k$ maps $A$ bijectively to $A'$. In particular, $|A'|=|A|$.
\end{lemma}

\begin{proof}[Proof of Lemma~\ref{lem_red1}]
Let $k=1$ and  $A=A'=B$. We will gradually increase $k$ and update $A, A'=\sigma_k(A)$ until we find the desired data.
Throughout the process, we maintain the invariants that $A\subseteq B^k$, $|A|\geq |B|^k/K^{(k-1)/2}$ and $\sigma_k|_A: A\to A'$ is bijective, which obviously hold when $k=1$.
Note these invariants imply $k\leq 2/\rho(B)$ since $|A|=|A'|\leq |\sigma_k(B^k)| \leq |\sg{B}|$ and $|A|\geq |B|^k/K^{(k-1)/2}\geq |B|^{k/2}$.

Consider the following cases.

\subparagraph*{Case 1:  $K|A'| \leq |A'+B|\leq |A'||B|/K$.} In this case, (1) of the lemma holds  by Lemma~\ref{lem_shrinkweak}.

\subparagraph*{Case 2: $|A'+B|<K|A'|$.} It follows from Pl\"unnecke's inequality (Theorem~\ref{thm_plunnecke}) that $|2kB|\leq K^{2k} |A'|$.
As $A'+A'\subseteq 2kB$, we have $|A'+A'|\leq K^{2k} |A'|$. 
So (2) holds.

\subparagraph*{Case 3:  $|A'+B|>|A'||B|/K$.} In this case, let $T$ be the union of the blocks $B'\in\Pi^{(k+1)}$ satisfying $B'\subseteq A\times B$ and $|B'|\leq K^{1/2}|A|$.
First assume $|T|\geq K^{1/2}|A|$.  By removing a subset of blocks in $\Pi^{(k+1)}$ from $T$ if necessary, we can find a subset $T'\subseteq T$ such that $T'\in\B(\Pi^{(k+1)})$ and $K^{1/2}|A|\leq |T'|\leq 2K^{1/2}|A|\leq  |A||B|/K^{1/2}$.
Choose $x\in A$ and let $B'=\{y\in B: (x,y)\in T'\}\in \B(\Pi_x^{(1)})$. Then $|B'|=|T'|/|A|\in [K^{1/2}, |B|/K^{1/2}]$ by Property (P2). So (1) holds.

So we may assume $|T|<K^{1/2}|A|$. 
For $x\in A$, the number of $y\in B$ satisfying $(x,y)\in T$ is bounded by $K^{1/2}$ by Property (P2).
So $|\sigma_{k+1}(T)|\leq K^{1/2}|A'|$.
Let $U=(A\times B)\setminus T$.
As $A'+B=\sigma_{k+1}(A\times B)=\sigma_{k+1}(T)\cup\sigma_{k+1}(U)$, we have 
\[
|\sigma_{k+1}(U)| \geq |A'+B|-|\sigma_{k+1}(T)|\geq |A'||B|/K- K^{1/2}|A'|\geq  |A'||B|/(2K).
\]
So $|U|\leq |A\times B|=|A'||B|\leq 2K|\sigma_{k+1}(U)|$. By an averaging argument,  there exists $A^*\in \Pi^{(k+1)}$ such that $A^*\subseteq U$ and $|A^*|\leq 2K |\sigma_{k+1}(A^*)|$.
By Lemma~\ref{lem_bij} and the fact $m\geq 4/\rho(B)+\log K+1$, the map $\sigma_{k+1}|_{A^*}:A^*\to \sigma_{k+1}(A^*)$ is bijective.
Pick $x\in A$ and let $B'=\{y\in B: (x,y)\in A^*\}$. Then $B'\in \Pi_x^{(1)}$. As $A^*\subseteq U$, we have $|B'|=|A^*|/|A|\geq K^{1/2}$. If $|B'|\leq |B|/K^{1/2}$ then (1) holds.
So assume $|B'|>|B|/K^{1/2}$. Then $|A^*|=|A||B'|\geq |A||B|/K^{1/2}\geq |B|^{k+1}/K^{k/2}$, where the last inequality holds since $|A|\geq |B|^k/K^{(k-1)/2}$. Then we replace $k$,  $A$, and $A'$ by $k+1$, $A^*$, and $\sigma_{k+1}(A^*)$ respectively. Note all the invariants are preserved.
 
Continue the above process and note  $k$  never exceeds $2/\rho(B)$. This proves the lemma.
\end{proof}

In Case~(2) of  Lemma~\ref{lem_red1}, we obtain a set $A'=\sigma_k(A)$ such that $|A'+A'|/|A'|$ is small. Our strategy in this case consists of the following steps:
\begin{enumerate}
\item  Using $\Pi$  to construct  a new linear $m'$-scheme $\Pi'$ on $A'$ such that $A'\in \Pi'^{(1)}$, where $m'\leq m$.
\item  Applying Lemma~\ref{lem_densecase} to $\Pi'$ and $A'$, and obtain an improved bound with respect to $\Pi'$.
\item  Turning the bound obtained in Step (2) into an improved bound with respect to $\Pi$. 
\end{enumerate}
Step~(1) is achieved by the following lemma, whose proof is deferred to the appendix.

\begin{lemma}\label{lem_schemepower}
Let $k,m'\in [m]$, $A\in\Pi^{(k)}$ and $A'=\sigma_k(A)$ such that $m\geq 2km'$ and $\sigma_k|_A: A\to A'$ is bijective.
For $k\in [m']$, write $\sigma_k^{(i)}: V^{ki}\to V^i$ for the map sending $(x_1,\dots,x_i)$  to $(\sigma_k(x_1),\dots,\sigma_k(x_i))$, where $x_1,\dots,x_i\in V^k$.
Define $\Pi'=\{\Pi'^{(1)},\dots,\Pi'^{(m')}\}$ such that for $i\in [m']$,   $\Pi'^{(i)}:=\{\sigma_k^{(i)}(B): B\in \Pi^{(ki)}, B\subseteq A^i\}$.
Then $\Pi'$ is a well defined strongly antisymmetric linear $m'$-scheme on $A'$.
 Moreover, for $i\in [m']$ and $B\in \Pi^{(ki)}$ satisfying $B\subseteq A^{i}$, the map $\sigma_k^{(i)}|_{B}: B\to \sigma_k^{(i)}(B)$ is bijective.
\end{lemma}

Now we are ready to prove Lemma~\ref{lem_fastshrink}.
 
\begin{proof}[Proof of Lemma~\ref{lem_fastshrink}] 
As  $|\sg{B}|\leq |\sg{S}|=|S|^{1/\rho}=n^{1/\rho}$ and $|B|\geq n^{1/(\rho^2 \log\log\log\log n)^{1/3}}$, we have 
 $\rho(B)=\log |B|/\log |\sg{B}|\geq (\rho/\log\log\log\log n)^{1/3}$.
Let $K=(\log\log\log n)^{\rho(B)/8}$.
By Lemma~\ref{lem_red1}, one of the following is true:
\begin{enumerate}
\item There exist $1\leq k\leq 2/\rho(B)+1$, $x_1,\dots,x_k\in B$, and $B'\in \B(\Pi_{x_1,\dots,x_k}^{(1)})$ such that $B'\subseteq B$ and  $\min\{|B'|, |B|/|B'|\}\geq K^{1/2}$.
\item There exist $1\leq k\leq 2/\rho(B)$ and $A\in \Pi^{(k)}$ such that  $A\subseteq B^k$, $|A|\geq  |B|^k/K^{(k-1)/2}$, $|A'+A'|\leq K^{2k}|A'|$, and $\sigma_k|_A: A\to A'$ is bijective, where $A':=\sigma_k(A)$. 
\end{enumerate}
If (1) holds then we are done since $K^{1/2}=2^{\Omega(\rho(B)\log\log\log\log n)}=2^{\Omega(k(\rho^2\log\log\log\log n)^{1/3})}$.
 So assume (2) holds.
Choose $m'=\lfloor m/(2k)\rfloor$.
Let $\Pi'$ be the strongly antisymmetric linear $m'$-scheme on $A'$ constructed from $\Pi$ as in Lemma~\ref{lem_schemepower}.

 As $|A'+A'|\leq K^{2k}|A'|$ and $K^{2k}\leq K^{4/\rho(B)} \leq (\log\log\log n)^{1/2}$,
we know by Lemma~\ref{lem_densecase} (applied to $\Pi'$, $A'$, and $N=n$) that there exist $r=O(\log\log\log n)$, $x_1,\dots,x_r\in A'$, and $A''\in \B(\Pi'^{(1)}_{x_1,\dots,x_r})$
such that $A''\subseteq A'$ and
\[
 2^{C r\log\log\log\log n  }\leq |A'|/|A''|\leq 2^{(\log n)^{1/2}}
\]
for some constant $C>0$.

For $i\in [r]$, choose $y_i\in A$ such that $\sigma_k(y_i)=x_i$. Let $y=(y_1,\dots,y_r)\in A^{r}\subseteq B^{kr}$.
We then have the following claim, whose proof is deferred to the appendix.

\begin{claim}\label{claim_final}
There exists a set $T\in \B(\Pi_y^{(k)})$ such that $T\subseteq A$ and $|T|=|A''|$.
\end{claim}

Let $T$ be as in Claim~\ref{claim_final}.
Let $K'=(\log\log\log n)^{r\rho(B)}$.
 For $0\leq i\leq  k$, let $\pi_i:V^{k}\to V^{i}$ be the projection to the first $i$ coordinates.
For $i\in [k]$, we say a block $U\in \Pi_y^{(k)}$ is \emph{$i$-small} if $|\pi_i(U)|/|\pi_{i-1}(U)|\leq K'$.
For $i\in [k]$, let $T_i$ be the union of the $i$-small blocks $U\in \B(\Pi_y^{(k)})$ satisfying $U\subseteq T$.
We address the following two cases separately:

\subparagraph*{Case 1:  $|T_i|\geq K'|B|^{k-1}$ for some $i\in [k]$.} Fix such $i\in [k]$. 
As $T_i\subseteq T\subseteq A\subseteq B^k$, we have $|\pi_{i-1}(T_i)|\leq |B|^{i-1}$ and $|\pi_{i}(T_i)|\geq |T_i|/|B|^{n-i}\geq K'|B|^{i-1}$. By the pigeonhole principle, there exists $z\in \pi_{i-1}(T_i)$ such that the cardinality of $Z:=\{w\in B: (z,w)\in \pi_i(T_i)\}$ is at least $K'$. Fix such $z$. Then $Z\in \B(\Pi_{y,z}^{(1)})$. 
As $T_i$ only contains $i$-small blocks, every block in $\Pi_{y,z}^{(1)}$ contained in $Z$ has cardinality at most $K'$.
By removing some of these blocks if necessary, we obtain a subset $Z'\subseteq Z$ such that $Z'\in \B(\Pi_{y,z}^{(1)})$ and $K'\leq |Z'|\leq 2K'=O(|B|/K')$.

Choose $B'=Z'$. 
Note $(y,z)\in B^{k'}$ where $k':=kr+i-1$.
To see Lemma~\ref{lem_fastshrink} is satisfied by $B'$, it suffices to show $K'=2^{\Omega(k'(\rho^2 \log\log\log\log n)^{1/3})}$.
This holds since $k'=O(r/\rho(B))$,  $K'=(\log\log\log n)^{r\rho(B)}$, and $\rho(B)\geq (\rho/\log\log\log\log n)^{1/3}$.

\subparagraph*{Case 2: $|T_i|<K'|B|^{k-1}$ for all $i\in [k]$.} So  
$\sum_{i=1}^k |T_i|<kK'|B|^{k-1}=|B|^{k-1}2^{(\log n)^{o(1)}}$.
As $\sigma_k|_A: A\to A'$ is bijective, we also have  
\[
|A'|=|A|\geq |B|^k/K^{(k-1)/2}=|B|^k/2^{(\log n)^{o(1)}}.
\]
 Using the facts $|B|\geq n^{1/(\rho^2 \log\log\log\log n)^{1/3}}$ and $|T|=|A''|\geq |A'|/2^{(\log n)^{1/2}}$, we see $|T|>\sum_{i=1}^k |T_i|$.
Therefore, there exists  a block $U\in \Pi_y^{(k)}$ such that $U\subseteq T$ and $U$ is not $i$-small for $i\in [k]$.
Note $|U|=\prod_{i=1}^k |\pi_i(U)|/|\pi_{i-1}(U)|$. Fix $i\in[k]$ that minimizes $ |\pi_i(U)|/|\pi_{i-1}(U)|$.
Then  
\[
|\pi_i(U)|/|\pi_{i-1}(U)| \leq |U|^{1/k} \leq  |T|^{1/k} =  |A''|^{1/k}\leq \left( |A|/ 2^{-C r\log\log\log\log n}\right)^{1/k}\leq |B|/K'^{\Omega(1)}.
\]
As $U$ is not $i$-small, we also have $|\pi_i(U)|/|\pi_{i-1}(U)|\geq K'$.
Pick $z\in \pi_{i-1}(U) $. 
Let $B'=\{w\in B: (z,w)\in \pi_i(T_i)\}$. Then $B'\in  \B(\Pi_{y,z}^{(1)})$, $|B'|=|\pi_i(U)|/|\pi_{i-1}(U)|$, and $(y,z)\in B^{k'}$, where $k'=kr+i-1$.
 As in the previous case, we have $K'=2^{\Omega(k'(\rho^2 \log\log\log\log n)^{1/3})}$.  So Lemma~\ref{lem_fastshrink} is satisfied by $B'$.
\end{proof} 

\section{Conclusion}\label{sec_conclusion}
It is natural to ask how to simplify our proof and/or improve our bounds. The bottleneck is Lemma~\ref{lem_densecase}, whose proof suffer exponential loss in several places, resulting in the weak $(\rho^2 \log\log\log\log n)^{1/3}$ improvement. One place is the Freiman--Ruzsa Theorem (Theorem~\ref{thm_fr}), where the bound $|\sg{A}|/|A|\leq \ell^{2K}$ is exponential in $K$. One natural idea is replacing it by the quasi-polynomial Freiman-Ruzsa Theorem \cite{San12}. However, it is not clear to us if the latter can be made constructive enough to be compatible with our notion of linear $m$-schemes.
Another place is Theorem~\ref{thm_smallgen}, which gives the bound $h(A)=O(1/\mu(A))$. We suspect it may be improved to $h(A)=\tilde{O}(\log (1/\mu(A)))$ when the ambient space $V$ is defined over a small prime field $\F$. Indeed, this was achieved in \cite{Lev03} for the special case $\F=\F_2$.

%

\bibliography{ref}

\begin{thebibliography}{10}

\bibitem{AMM77}
L.~Adleman, K.~Manders, and G.~Miller.
\newblock On taking roots in finite fields.
\newblock In {\em Proceedings of the 18th Annual Symposium on Foundations of
  Computer Science}, pages 175--178, 1977.

\bibitem{Aro13}
M.~Arora.
\newblock {\em Extensibility of association schemes and {GRH-based}
  deterministic polynomial factoring}.
\newblock PhD thesis, Universit{\"a}ts-und Landesbibliothek Bonn, 2013.

\bibitem{AIKS14}
M.~Arora, G.~Ivanyos, M.~Karpinski, and N.~Saxena.
\newblock Deterministic polynomial factoring and association schemes.
\newblock {\em LMS Journal of Computation and Mathematics}, 17(01):123--140,
  2014.

\bibitem{BS94}
A.~Balog and E.~Szemer{\'e}di.
\newblock A statistical theorem of set addition.
\newblock {\em Combinatorica}, 14(3):263--268, 1994.

\bibitem{Ber67}
E.~R. Berlekamp.
\newblock Factoring polynomials over finite fields.
\newblock {\em Bell System Technical Journal}, 46(8):1853--1859, 1967.

\bibitem{Ber70}
E.~R. Berlekamp.
\newblock Factoring polynomials over large finite fields.
\newblock {\em Mathematics of Computation}, 24(111):713--735, 1970.

\bibitem{BKS15}
J.~Bourgain, S.~Konyagin, and I.~Shparlinski.
\newblock Character sums and deterministic polynomial root finding in finite
  fields.
\newblock {\em Mathematics of Computation}, 84(296):2969--2977, 2015.

\bibitem{CH00}
Q.~Cheng and M.~A. Huang.
\newblock Factoring polynomials over finite fields and stable colorings of
  tournaments.
\newblock In {\em Proceedings of the 4th Algorithmic Number Theory Symposium},
  pages 233--245, 2000.

\bibitem{Evd92}
S.~A. Evdokimov.
\newblock Factorization of solvable polynomials over finite fields and the
  generalized {Riemann} hypothesis.
\newblock {\em Journal of Soviet Mathematics}, 59(3):842--849, 1992.

\bibitem{Evd94}
S.~A. Evdokimov.
\newblock Factorization of polynomials over finite fields in subexponential
  time under {GRH}.
\newblock In {\em Proceedings of the 1st Algorithmic Number Theory Symposium},
  pages 209--219, 1994.

\bibitem{Eve12}
C.~Even-Zohar.
\newblock On sums of generating sets in $\mathbb{Z}_2^n$.
\newblock {\em Combinatorics, probability and computing}, 21(6):916--941, 2012.

\bibitem{EL14}
C.~Even-Zohar and S.~Lovett.
\newblock The {Freiman--Ruzsa} theorem over finite fields.
\newblock {\em Journal of Combinatorial Theory, Series A}, 125:333--341, 2014.

\bibitem{Gao01}
S.~Gao.
\newblock On the deterministic complexity of factoring polynomials.
\newblock {\em Journal of Symbolic Computation}, 31(1):19--36, 2001.

\bibitem{Gow98}
W.~T. Gowers.
\newblock A new proof of {Szemer{\'e}di's} theorem for arithmetic progressions
  of length four.
\newblock {\em Geometric and Functional Analysis}, 8(3):529--551, 1998.

\bibitem{Gua09}
Y.~Guan.
\newblock {\em Factoring polynomials and Gr\"obner bases}.
\newblock PhD thesis, Clemson University, 2009.

\bibitem{Guo17}
Z.~Guo.
\newblock {\em $\mathcal{P}$-schemes and deterministic polynomial factoring
  over finite fields}.
\newblock PhD thesis, Caltech, 2017.

\bibitem{Guo19-2}
Z.~Guo.
\newblock Deterministic polynomial factoring over finite fields with restricted
  {Galois} groups, 2019.
\newblock Manuscript. \url{https://zeyuguo.bitbucket.io/papers/galois.pdf}.

\bibitem{Guo19}
Z.~Guo.
\newblock Deterministic polynomial factoring over finite fields: a uniform
  approach via $\mathcal{P}$-schemes.
\newblock {\em Journal of Symbolic Computation}, 96:22 -- 61, 2020.

\bibitem{HR96}
Y.~O. Hamidoune and {\O}.~R{\o}dseth.
\newblock On bases for $\sigma$-finite groups.
\newblock {\em Mathematica Scandinavica}, pages 246--254, 1996.

\bibitem{Hua91-2}
M.~A. Huang.
\newblock Factorization of polynomials over finite fields and decomposition of
  primes in algebraic number fields.
\newblock {\em Journal of Algorithms}, 12(3):482 -- 489, 1991.

\bibitem{Hua91}
M.~A. Huang.
\newblock Generalized {Riemann} hypothesis and factoring polynomials over
  finite fields.
\newblock {\em Journal of Algorithms}, 12(3):464 -- 481, 1991.

\bibitem{IKRS12}
G.~Ivanyos, M.~Karpinski, L.~R{\'o}nyai, and N.~Saxena.
\newblock Trading {GRH} for algebra: algorithms for factoring polynomials and
  related structures.
\newblock {\em Mathematics of Computation}, 81(277):493--531, 2012.

\bibitem{IKS09}
G.~Ivanyos, M.~Karpinski, and N.~Saxena.
\newblock Schemes for deterministic polynomial factoring.
\newblock In {\em Proceedings of the International Symposium on Symbolic and
  Algebraic Computation}, pages 191--198, 2009.

\bibitem{KL03}
B.~Klopsch and V.~F. Lev.
\newblock How long does it take to generate a group?
\newblock {\em Journal of Algebra}, 261(1):145--171, 2003.

\bibitem{Lev03}
V.~F. Lev.
\newblock Generating binary spaces.
\newblock {\em Journal of Combinatorial Theory, Series A}, 102(1):94--109,
  2003.

\bibitem{MM99}
G.~Malle and B.~H. Matzat.
\newblock {\em Inverse Galois Theory}.
\newblock Springer, 1999.

\bibitem{Neu99}
J.~Neukirch.
\newblock {\em Algebraic Number Theory}.
\newblock Springer-Verlag, 1999.

\bibitem{Pil90}
J.~Pila.
\newblock Frobenius maps of abelian varieties and finding roots of unity in
  finite fields.
\newblock {\em Mathematics of Computation}, 55(192):745--763, 1990.

\bibitem{Ron88}
L.~R{\'o}nyai.
\newblock Factoring polynomials over finite fields.
\newblock {\em Journal of Algorithms}, 9(3):391--400, 1988.

\bibitem{Ron89}
L.~R{\'o}nyai.
\newblock Factoring polynomials modulo special primes.
\newblock {\em Combinatorica}, 9(2):199--206, 1989.

\bibitem{Ron92}
L.~R{\'o}nyai.
\newblock {Galois} groups and factoring polynomials over finite fields.
\newblock {\em SIAM Journal on Discrete Mathematics}, 5(3):345--365, 1992.

\bibitem{Ruz99}
I.~Z. Ruzsa.
\newblock An analog of {Freiman's} theorem in groups.
\newblock {\em Asterisque}, 258:323--326, 1999.

\bibitem{San12}
T.~Sanders.
\newblock On the {Bogolyubov--Ruzsa} lemma.
\newblock {\em Analysis \& PDE}, 5(3):627--655, 2012.

\bibitem{Sch85}
R.~Schoof.
\newblock Elliptic curves over finite fields and the computation of square
  roots mod $p$.
\newblock {\em Mathematics of Computation}, 44(170):483--494, 1985.

\bibitem{Sho90}
V.~Shoup.
\newblock On the deterministic complexity of factoring polynomials over finite
  fields.
\newblock {\em Information Processing Letters}, 33(5):261--267, 1990.

\bibitem{Sho96}
V.~Shoup.
\newblock Smoothness and factoring polynomials over finite fields.
\newblock {\em Information Processing Letters}, 38(1):39--42, 1991.

\bibitem{TV06}
T.~Tao and V.~H. Vu.
\newblock {\em Additive Combinatorics}, volume 105.
\newblock Cambridge University Press, 2006.

\bibitem{von87}
J.~von~zur Gathen.
\newblock Factoring polynomials and primitive elements for special primes.
\newblock {\em Theoretical Computer Science}, 52(1):77--89, 1987.

\bibitem{Yun76}
D.~Y.Y. Yun.
\newblock On square-free decomposition algorithms.
\newblock In {\em Proceedings of the 3rd ACM Symposium on Symbolic and
  Algebraic Computation}, pages 26--35, 1976.

\end{thebibliography}
 
\appendix
 
 \section{$\mathcal{P}$-schemes}\label{sec_pscheme}
 
 \subparagraph*{Notations.} A \emph{(left)  action}   of a group $G$ on a set $S$ is a function  $\varphi:G\times S\to S$ satisfying (1) $\varphi(e,x)=x$ for all $x\in S$ and (2) $\varphi(g,\varphi(h,x))=\varphi(gh,x)$ for $x\in S$ and $g,h\in G$. We also write $\prescript{g}{}{x}$ for $\varphi(g,x)$. The \emph{orbit} or \emph{$G$-orbit} of   $x\in S$ is $Gx:=\{\prescript{g}{}{x}:g\in G\}$. 
 The \emph{stabilizer} of $x\in S$ is $G_x:=\{g\in G:\prescript{g}{}{x}=x\}$. 
For $T\subseteq S$, the \emph{pointwise stabilizer} of $T$ is  $G_T:=\{g\in G: \prescript{g}{}{x}=x \text{ for } x\in T\}$.
When $T=\{x_1,\dots,x_k\}$, we also write $G_{x_1,\dots,x_k}$ for $G_T$. 

 \subparagraph*{$\mathcal{P}$-schemes.} We  recall the definition of  $\mathcal{P}$-schemes in \cite{Guo19}:
Let $G$ be a finite group. A set $\mathcal{P}$ of subgroups of $G$ is called a \emph{subgroup system over $G$} if $\mathcal{P}$ is closed under conjugation, i.e., $gHg^{-1}\in\mathcal{P}$ for $H\in\mathcal{P}$ and $g\in G$.
Define the following two kinds of maps between right coset spaces $H\backslash G$ for various subgroups $H\leq G$:
\begin{itemize}
\item For $H\leq H'\leq G$, define the \emph{projection} $\pi_{H, H'}: H\backslash G\to H'\backslash G$ to be the map sending $Hg\in H\backslash G$ to $H'g\in H'\backslash G$.
\item For $H\leq G$ and $g\in G$, define the \emph{conjugation} $c_{H,g}: H\backslash G\to gHg^{-1}\backslash G$ to be the map sending $Hh\in H\backslash G$ to $(gHg^{-1})gh \in gHg^{-1}\backslash G$.
\end{itemize}

\begin{definition}[$\mathcal{P}$-scheme \cite{Guo19}]\label{defi_pscheme}
Let $\mathcal{P}$ be a subgroup system over $G$. A \emph{$\mathcal{P}$-collection} is a set $\mathcal{C}=\{C_H: H\in \mathcal{P}\}$ indexed by $\mathcal{P}$, where   $C_H$ is a partition of $H\backslash G$ for $H\in\mathcal{P}$.
Moreover, we say $\mathcal{C}$ is a \emph{$\mathcal{P}$-scheme} if it has the following properties:
\begin{description}
\item[(Compatibility):]  For $H,H'\in \mathcal{P}$ with $H\leq H'$ and $x,x'\in H\backslash G$  in the same block of $C_H$,  $\pi_{H,H'}(x)$ and $\pi_{H,H'}(x')$ are in the same block of $C_{H'}$.
\item[(Invariance):]  For $H\in\mathcal{P}$ and $g\in G$, the map $c_{H,g}: H\backslash G\to gHg^{-1}\backslash G$ maps any block of $C_H$ bijectively to a block of $C_{gHg^{-1}}$.
\item[(Regularity):] For $H,H'\in\mathcal{P}$ with $H\leq H'$, any block $B\in C_H$, $B'\in C_{H'}$, the number of $x\in B$ satisfying $\pi_{H,H'}(x)=y$ is  constant when $y$ ranges over the set $B'$.
\end{description}
In addition, $\mathcal{C}$ is  \emph{discrete} on $H\in\mathcal{P}$ if $C_H$ is the finest partition of $H\backslash G$.  
It is \emph{strongly antisymmetric}  if no nontrivial permutation of any block in any partition $C_H$ can be obtained by composing maps of the form  $c_{H_{i-1},g}|_{B_{i-1}}$, $\pi_{H_{i-1},H_i}|_{B_{i-1}}$, or  $(\pi_{H_i,H_{i-1}}|_{B_i})^{-1}$.
\end{definition}

Now suppose $G$ is a permutation group on a finite set $S$.
For $m\in\N^+$, define   $\mathcal{P}_m:=\{G_{x_1,\dots,x_k}: k\in [m],~x_1,\dots,x_k\in S\}$, which is a subgroup system over $G$, called the \emph{system of stabilizers of depth $m$}. Denote by $d(G)$ the smallest $m\in\N^+$ such that every strongly antisymmetric $\mathcal{P}_m$-scheme is discrete on $G_x$ for all $x\in S$. Then $d(G)$ is well defined and we have 

 \begin{theorem}[{\cite[Theorem~1.3]{Guo19}}]\label{thm_algspecial}
Under GRH, there exists a deterministic algorithm that, given $f(X)\in\F_p[X]$ of degree $n$ that factorizes into $n$ distinct linear factors over $\F_p$ and a lifted polynomial $\tilde{f}(X)\in\Z[X]$ with the Galois group $G$ acting on the set of roots of $\tilde{f}(X)$, completely factorizes $f(X)$ over $\F_p$ in time polynomial in $n^{d(G)}$  and the size of the input.
\end{theorem}

\subparagraph*{From a $\mathcal{P}_m$-scheme to a linear $m$-scheme.}

Now suppose $G\leq\gl(V)$ acts linearly on $S\subseteq V$, where $V$ is a vector space over a finite field $\F$.
Let $m\in\N^+$, and let  $\mathcal{C}=\{C_H: H\in \mathcal{P}_m\}$ be a $\mathcal{P}_m$-scheme.
We will use  $\mathcal{C}$ to construct a linear $m$-scheme $\Pi(\mathcal{C})$ on $S$.

For $k\in [m]$, equip $S^k$  with the diagonal action of $G$.
Then $\mathcal{P}_m=\{G_x: x\in S^k, k\in [m]\}$.
For a subgroup $H\leq G$, equip $H\backslash G$ with the \emph{inverse right action} of $G$, i.e., $\prescript{g}{}{Hh}=Hhg^{-1}$ for $Hh\in H\backslash G$ and $g\in G$.
For $k\in [m]$ and $x\in S^k$, let $\lambda_x$ be the map  $G x\to G_x\backslash G$ sending  $\prescript{g}{}{x}\in G x$ to $G_x g^{-1}$ for $g\in G$.
Then $\lambda_x$ is a well defined bijection and is \emph{$G$-equivariant}, i.e., $\lambda_x(\prescript{g}{}{\cdot})=\prescript{g}{}{\lambda_x(\cdot)}$ for $g\in G$.

\begin{definition}
For a $\mathcal{P}_m$-scheme $\mathcal{C}=\{C_H: H\in \mathcal{P}_m\}$, let
\[
 \Pi(\mathcal{C}):=\{\Pi(\mathcal{C})^{(1)},\dots,\Pi(\mathcal{C})^{(m)}\}
 \] where for $k\in [m]$,  $\Pi(\mathcal{C})^{(k)}$ is a partition of $S^k$ defined as follows:  elements of $S^k$ are in different blocks of $\Pi(\mathcal{C})^{(k)}$ if they are in different $G$-orbits. For each $G$-orbit $O$, choose $x\in O$ so that $O=G x$ and  we have a bijection $\lambda_x: O\to G_x\backslash G$. Then define the partition $\Pi(\mathcal{C})^{(k)}|_O$ of $O$ to be $\lambda_x^{-1}(C_{G_x})$.
\end{definition}

\begin{lemma}\label{lem_ptom}
$\Pi(\mathcal{C})$ is well defined linear $m$-scheme on $S$ independent of the choices of the elements $x$. Moreover, if $\mathcal{C}$ is strongly antisymmetric, so is $\Pi(\mathcal{C})$. 
\end{lemma}

\begin{proof}
For $x'=\prescript{g}{}{x}$, $g\in G$, we have $\lambda_{x'}= c_{G_x,g} \circ \lambda_x$. It follows from invariance of $\mathcal{C}$ that $\Pi(\mathcal{C})^{(k)}|_O$ does not depend on the choice of $x\in O$. So $\Pi(\mathcal{C})$ is well defined.

Next we prove that $\Pi(\mathcal{C})$ is a linear $m$-scheme on $S$.
Let $\tau\in\M_{k,k'}$ where  $k,k'\in [m]$. Then $\tau$ has the form
\[
\mathbf{x}=(x_1,\dots,x_k)\mapsto \left(\sum_{i=1}^k c_{i,1} x_i, \dots, \sum_{i=1}^k c_{i, k'} x_i\right), \quad\text{where}~c_{i,j}\in\F.
\]
As $G\leq \gl(V)$ acts linearly on $S$, and  diagonally on $S^k$ and $S^{k'}$, we have
\[
\prescript{g}{}{\tau(\mathbf{x})} = \prescript{g}{}{\left(\sum_{i=1}^k c_{i,1} x_i, \dots, \sum_{i=1}^k c_{i, k'} x_i\right)}=\left(\sum_{i=1}^k c_{i,1} \prescript{g}{}{x_i}, \dots, \sum_{i=1}^k c_{i, k'} \prescript{g}{}{x_i}\right)
=\tau(\prescript{g}{}{\mathbf{x}}),
\]
 i.e., $\tau$ is $G$-equivariant. Then for $x\in S^k$ such that $\tau(x)\in S^{k'}$, we have $G_x \leq G_{\tau(x)}$ and 
the following diagram commutes.
\begin{equation}\label{eq_ptom}
\begin{tikzcd}[column sep=large]
G x\arrow[r, "\tau|_{G x}"] \arrow[d, "\lambda_x"']
& G \tau(x) \arrow[d, "\lambda_{ \tau(x)}"] \\
G_x\backslash G  \arrow[r,  "  \pi_{G_x,G_{ \tau(x)}}"]
& G_{ \tau(x)}\backslash G  \nospacepunct{.}
\end{tikzcd}
\end{equation}
Also note that the maps $\lambda_x$ and $\lambda_{\tau(x)}$ are bijections, sending blocks to blocks. 
The properties (P1) and (P2) of $\Pi(\mathcal{C})$ then follow from compatibility and regularity of $\mathcal{C}$. 
 So $\Pi(\mathcal{C})$ is a linear $m$-scheme.

Now assume $\Pi(\mathcal{C})$ is not strongly antisymmetric and we prove $\mathcal{C}$ is not either. By definition, for some $k\in [m]$, there exists a nontrivial permutation $\tau\in\widetilde{\M}_{\Pi(\mathcal{C})}$ of a block $B\in \Pi(\mathcal{C})^{(k)}$. Using Diagram~\eqref{eq_ptom}, we see that there exist $x,x'\in B$, and a map $\widetilde{\tau}:\widetilde{B}\to \widetilde{B}'$ that is a composition of maps of the form $\pi_{H,H'}|_{B'}$ or $(\pi_{H,H'}|_{B'})^{-1}$  (where $H,H'\in\mathcal{P}_m$ and $B'$ is a block)  such that   the following diagram commutes.
\[
\begin{tikzcd}[column sep=large]
B\arrow[r, "\tau"] \arrow[d, "\lambda_x"']
& B\arrow[d, "\lambda_{x'}"] \\
\widetilde{B}  \arrow[r,  " \widetilde{\tau}"]
& \widetilde{B}'  \nospacepunct{.}
\end{tikzcd}
\]
By construction, $B$ is a subset of a $G$-orbit of $S^k$. So $x'=\prescript{g}{}{x}$ for some $g\in G$.
Note $\lambda_{x'}= c_{G_x,g} \circ \lambda_x$. By replacing  $\widetilde{\tau}$ with  $\widetilde{\tau}\circ c_{G_x,g}^{-1}$, $x$ with $x'$, and $\widetilde{B}$ with $\widetilde{B}'$ respectively, we may assume $x=x'$ and $\widetilde{B}=\widetilde{B}'$.
Then as $\tau$ is a nontrivial permutation, so is $\widetilde{\tau}$. So $\mathcal{C}$ is not strongly antisymmetric.
\end{proof}

Now we are ready to  derive  Theorem~\ref{thm_factormain} from Theorem~\ref{thm_mainbound} (together with the more elementary Lemma~\ref{lemma_bounddim}).

\begin{proof}[Proof of  Theorem~\ref{thm_factormain} ]
Let $G=\gl(V)$ act linearly on $S\subseteq V$.
We will prove  $d(G)=O(\log n)$, $d(G)\leq\dim\sg{S}_\F$, and  $d(G)=O\left(\frac{\log n}{(\rho^2 \log\log\log\log n)^{1/3}}\right)$.
Theorem~\ref{thm_factormain} then follows directly from Theorem~\ref{thm_algspecial}.

If $d(G)=1$ then we are done. So assume $d(G)>1$.
Let $m=d(G)-1<d(G)$. Then there exists a strongly antisymmetric $\mathcal{P}_m$-scheme $\mathcal{C}$ on $S$ that is not discrete on $G_x$ for some $x\in S$.
By Lemma~\ref{lem_ptom}, $\Pi(\mathcal{C})$ is a strongly antisymmetric linear $m$-scheme on $S$. 
As $\mathcal{C}$ is not discrete on $G_x$, $\Pi(\mathcal{C})^{(1)}$ is not the finest partition of $S$.
We then have $m\leq \log n$ and $m<\dim\sg{S}_\F$ by Lemma~\ref{lemma_bounddim}. Then $d(G)=m+1$ satisfies $d(G)=O(\log n)$ and $d(G)\leq\dim\sg{S}_\F$.
Similarly,   $d(G)=m+1=O\left(\frac{\log n}{(\rho^2 \log\log\log\log n)^{1/3}}\right)$ holds by Theorem~\ref{thm_mainbound}.
\end{proof}

Next, we define the function $d_\lin(m,q)$ and prove  Theorem~\ref{thm_lin}.

\begin{definition}[\cite{Guo19-2}]
 For $m\in \N^+$ and a prime power $q$, define $d_{\lin}(m,q)$ to be the maximum possible value of $d(G)$, where  $G\leq \gl_{m'}(q')$ acts linearly on a set $S\subseteq \F_{q'}^{m'}$,
  and $(m',q')$ ranges over the set of pairs satisfying $m'\leq m$, $q'^{m'}\leq q^m$  and $\gcd(q,q')\neq 1$.
\end{definition}

\begin{proof}[Proof of Theorem~\ref{thm_lin}]
Let $m\in \N^+$ and $q$ be a prime power.
Let $G$ act linearly on a subset $S\subseteq \F_{q'}^{m'}$  where $m'\leq m$, $q'^{m'}\leq q^m$  and $\gcd(q,q')\neq 1$.
Suppose $m_0\in\N^+$ and $\mathcal{C}$ is a strongly antisymmetric $\mathcal{P}_{m_0}$-scheme on $S$ that is not discrete on $G_x$ for some $x\in S$.
We want to prove  $m_0=O\left(\frac{m\log q}{( \log\log\log(m\log q))^{1/3}}\right)$.

By Lemma~\ref{lem_ptom}, $\Pi(\mathcal{C})$ is a strongly antisymmetric linear $m_0$-scheme on $S$. 
As $\mathcal{C}$ is not discrete on $G_x$, $\Pi(\mathcal{C})^{(1)}$ is not the finest partition of $S$.
By Theorem~\ref{thm_mainbound},  we have  $m_0=O\left(\frac{\log n}{(\rho^2 \log\log\log\log n)^{1/3}}\right)$
where $\rho=\log |S|/\log |\sg{S}|$ and $n=|S|=|\sg{S}|^\rho\leq q^{\rho m}$. Combining this bound with the facts $\rho\leq 1$  and $\log n\leq \rho m\log q$ gives the desired bound.
\end{proof}

Finally, we prove Corollary~\ref{cor_bottleneck}.

\begin{proof}[Proof of  Corollary~\ref{cor_bottleneck}]
By the definition of $N_\mathcal{C}(G)$  \cite[Theorem~1.2]{Guo19-2}, we have $N_\mathcal{C}(G)=q_0^{m_0 d_\lin(cm_0,q_0)}$, 
where $G$ has a nonabelian composition factor  that is a classical group of rank $m_0$ over $\F_{q_0}$  and $c\in\N^+$ is an absolute constant. 
Here $q_0^{m_0}=n^{O(1)}$ by \cite[Corollary~4.10]{Guo19-2}. So $d_\lin(cm_0,q_0)=O\left(\frac{n}{( \log\log\log\log n)^{1/3}}\right)=o(\log n)$ by Theorem~\ref{thm_lin}. It follows that $N_\mathcal{C}(G)=n^{o(\log n)}$.

To prove the second statement, it suffices to prove  $N_\mathcal{A}(G)=n^{o(\log n)}$.
By the definition of $N_\mathcal{A}(G)$ \cite[Theorem~1.2]{Guo19-2}, we have $N_\mathcal{A}(G)=m_1^{d_\sym(m_1^c)}$, 
where $G$ has a nonabelian composition factor  that is an alternating group of degree $m_1$  and $c\in\N^+$ is an absolute constant.
By assumption, we have $m_1=n^{o(1)}$. Now $N_\mathcal{A}(G)=n^{o(\log n)}$ follows from $d_\sym(m)=O(\log m)$  \cite[Lemma~3.18]{Guo19-2}.
\end{proof}

\section{Omitted Proofs in Section~\ref{sec_lsch}}\label{sec_pf1}


 \begin{proof}[Proof of Lemma~\ref{lem_linrel}]
 Assume  $\sum_{i=1}^k c_i x_i=0$ and we prove that $\sum_{i=1}^k c_i y_i=0$.
 The claim is trivial if all $c_i$ are zero. By symmetry, we may assume $c_k\neq 0$. By scaling, we may further assume $c_k=-1$, i.e., $x_k=\sum_{i=1}^{k-1} c_i x_i$.
 Let $\tau\in \M_{k,k}$ be the map  $(a_1,\dots,a_{k-1},a_k)\mapsto (a_1,\dots,a_{k-1},\sum_{i=1}^{k-1} c_i a_i)$.
We have $\mathbf{x}=\tau(\mathbf{x})\in\tau(B)$. Then $\tau(B)=B$ by Property (P1). So $\mathbf{y}\in\tau(B)$. But $\sum_{i=1}^{k} c_i a_i=0$ holds for all $(a_1,\dots,a_k)\in\tau(B)$. So $\sum_{i=1}^k c_i y_i=0$.
 \end{proof}

 \begin{proof}[Proof of Lemma~\ref{lem_closedness}]
 Claim (1) holds since $\Pi^{(k)}$ is a partition of $S^k$.
 To prove (2), first consider the case $Q=\exists_{=t}$ where $t\in\N$. Suppose $B$ is the disjoint union of  $B_1,\dots,B_s\in \Pi^{(k+k')}$.
 Then 
 \[
 B_{\exists_{=t}}=\bigcup_{\substack{t_1,\dots,t_s\in \N\\ t_1+\dots+t_s=t}} \bigcap_{i=1}^s (B_i)_{\exists_{=t_i}}.
 \]
 So we may assume $B\in\Pi^{(k+k')}$.
Let $\pi\in\M_{k+k',k}$ be the projection sending $(x,y)\in V^k\times V^{k'}$ to $x\in V^k$.  
  Let $B'\in \Pi^{(k)}$. 
For $x\in B'$, we have  $x\in B_{\exists_{=t}}$ iff  $\#\{y\in S^{k'}: (x,y)\in B\}=t$.
 We also know 
that $\#\{y\in S^{k'}: (x,y)\in B\} =  \#\{z\in B: \pi(z)=x\}$ is constant when $x$ ranges over $B'$.
So either $B'\subseteq B_{\exists_{=t}}$ or $B'\cap B_{\exists_{=t}}=\emptyset$.
Therefore $B_{\exists_{=t}}$ is a disjoint union of blocks in $\Pi^{(k)}$, i.e., $B_{\exists_{=t}}\in\B(\Pi^{(k)})$.
This proves (2) for the case $Q=\exists_{=t}$.
The case $Q=\exists$ follows since $B_\exists=\bigcup_{t\in\N^+} B_{\exists_{=t}}$.
The case $Q=\forall$ also follows since $B_\forall$ is the complement of $\overline{B}_{\exists}$ in $S^k$, where $\overline{B}$ denotes the complement of $B$ in $S^{k+k'}$.

To prove (3), suppose $B$ is the disjoint union of  $B_1,\dots,B_s\in \Pi^{(k)}$. Then $\tau(B)\cap S^{k'}=\bigcup_{i=1}^s (\tau(B_i)\cap S^{k'})$. 
So we may assume $B\in\Pi^{(k)}$.
If $\tau(B)\cap S^{k'}=\emptyset$, then trivially $\tau(B)\cap S^{k'}\in \B(\Pi^{(k')})$.
So assume $\tau(B)\cap S^{k'}\neq \emptyset$.
Choose $B'\in \Pi^{(k')}$ such that $\tau(B)\cap B'\neq \emptyset$. Then $\tau(B)=B'\in \B(\Pi^{(k')})$ by Property~(P1) of linear $m$-schemes.

To prove (4), suppose $B$ is the disjoint union of  $B_1,\dots,B_s\in \Pi^{(k')}$. Then $\tau^{-1}(B)\cap S^{k}=\bigcup_{i=1}^s (\tau^{-1}(B_i)\cap S^{k})$. 
So we may assume $B\in\Pi^{(k')}$. By  Property~(P1), for every $B'\in \Pi^{(k)}$ intersecting $\tau^{-1}(B)\cap S^{k}$ (i.e. $\tau(B')\cap B\neq\emptyset$), we have $\tau(B')=B$, which implies $B'\subseteq \tau^{-1}(B)\cap S^{k}$. So $\tau^{-1}(B)\cap S^{k}$ is a disjoint union of blocks in $\Pi^{(k)}$, i.e., $\tau^{-1}(B)\cap S^{k}\in\B(\Pi^{(k)})$.
\end{proof}

\begin{proof}[Proof of Lemma~\ref{lem_fixing}]
As $\Pi_{x_1,\dots,x_t}=(\Pi_{x_1,\dots,x_{t-1}})_{x_t}$, it suffices to prove the claim for $t=1$. So assume $t=1$ and $x=x_1$.
For $k\in [m-1]$ and $B\in \Pi_x^{(k)}$, let $\widetilde{B}$ be the block in $\Pi^{(k+1)}$ satisfying $B=\{y\in S^k: (x,y)\in \widetilde{B}\}$.
For $k,k'\in [m-1]$ and $\tau\in\M_{k,k'}$, denote by $\widetilde{\tau}: V^{k+1}\to V^{k'+1}$ the map sending $(u,v)\in V\times V^k$ to $(u,\tau(v))$, which is in $\M_{k+1,k'+1}$.


For $k\in [m-1]$, identify $V^k$ with $\pi_{k+1,1}^{-1}(x)\subseteq V^{k+1}$ via  $y\mapsto (x,y)$.
Then  $S^k$ is identified with the set $S^{k+1}\cap \pi_{k+1,1}^{-1}(x)$.
Then $\Pi_x^{(k)}$ becomes a partition of  $S^{k+1}\cap \pi_{k+1,1}^{-1}(x)$, which  is precisely $\Pi^{(k+1)}|_{S^{k+1}\cap \pi_{k+1,1}^{-1}(x)}$. 
A block $B\in \Pi_x^{(k)}$ is then identified with $\widetilde{B}\cap \pi_{k+1,1}^{-1}(x)$.
For $k,k'\in [m-1]$, a  map $\tau\in\M_{k,k'}$ is identified with 
\[
\widetilde{\tau}|_{\pi_{k+1,1}^{-1}(x)}: \pi_{k+1,1}^{-1}(x)\to \pi_{k'+1,1}^{-1}(x).
\]

Consider $k, k'\in [m-1]$, $B\in \Pi^{(k+1)}$, $B'\in \Pi^{(k'+1)}$, and $\tau\in \M_{k,k'}$.
As $\Pi$ is a linear $m$-scheme, we have 
\begin{enumerate}
\item $\widetilde{\tau}(B)$ and $B'$ are either the same or disjoint from each other, and
\item  $\#\{z\in B: \widetilde{\tau}(z)=y\}$ is constant when $y$ ranges over $B'$.
\end{enumerate}
As $\widetilde{\tau}$ fixes the first  coordinate, it sends $\pi^{-1}_{k+1,1}(t)$ to $\pi^{-1}_{k'+1,1}(t)$ for $t\in V$.
Therefore
\begin{enumerate}
\item  $\widetilde{\tau}(B\cap\pi^{-1}_{k+1,1}(x))$ and $B'\cap \pi^{-1}_{k'+1,1}(x)$ are either the same or disjoint from each other, and 
\item for $y\in B'\cap \pi^{-1}_{k'+1,1}(x)$, $\#\{z\in B\cap\pi^{-1}_{k+1,1}(x): \widetilde{\tau}(z)=y\}=\#\{z\in B: \widetilde{\tau}(z)=y\}$, which is constant when  $y$ ranges over $ B'\cap \pi^{-1}_{k'+1,1}(x)$.
\end{enumerate}
This shows that $\Pi_x$ is a linear $m$-scheme by the above identifications. 

Assume $\Pi_x$ is not strongly antisymmetric. 
Then there exists a sequence of  bijections $\tau_i: B_{i-1}\to B_i$, $i=1,\dots,s$, whose composition is a nontrivial permutation of $B_0=B_s$, and for $i\in [s]$, either $\tau_i$ or $\tau_i^{-1}$ is in $\M_{\Pi_x}$. Suppose $B_i\in \Pi_x^{(k_i)}$  for $i=0,\dots,s$, where $k_i\in [m-1]$.
We have identified $B_i$ with $\widetilde{B}_i\cap \pi_{k_i}^{-1}(x)$ for unique $\widetilde{B}_i\in \Pi^{(k_i+1)}$.
Each map $\tau_i$ then extends to a bijection $\widetilde{\tau}_i: \widetilde{B}_{i-1}\to \widetilde{B}_i$ such that either $\widetilde{\tau}_i$ or $\widetilde{\tau}_i^{-1}$ is in $\M_\Pi$. The composition of these maps $\widetilde{\tau}_i$ or their inverses then gives a nontrivial permutation of $\widetilde{B}_0=\widetilde{B}_s$.
So $\Pi$ is not strongly antisymmetric.
\end{proof}

 \begin{proof}[Proof of Lemma~\ref{lemma_bounddim}]
 (1):  Let $k=\dim \sg{S}_{\F}$. Assume to the contrary that $m\geq k$.
 Let $T$ be the set of $(x_1,\dots,x_k)\in S^k$ such that the coordinates $x_i$ are linearly independent over $\F$.
 By Lemma~\ref{lem_linrel}, we have $T\in \B(\Pi^{(k)})$.
 
 Assume $\mathbf{x}=(x_1,\dots,x_k), \mathbf{y}=(x_1,\dots,x_k)\in T$ are in the same block $B\in \Pi^{(k)}$. As the coordinates of $\mathbf{x}$ as well as those of $\mathbf{y}$ form a basis of $\sg{S}_{\F}$, there exists an  invertible linear map $\tau\in  \M_{k,k}$ sending $\mathbf{x}$ to $\mathbf{y}$.
 So $\mathbf{y}\in \tau(B)\cap B$.
Then $\tau(B)=B$ by Property (P1). 
As $\Pi$ is strongly antisymmetric, we must have $\mathbf{x}=\mathbf{y}$. This shows that $T$ is a disjoint union of singletons in $\Pi^{(k)}$, i.e., $\Pi^{(k)}|_{T}=\infty_T$.
Also note  $\pi_{k,1}\in \M_{k,1}$ maps $T$ surjectively to $S$. It follows from Lemma~\ref{lem_closedness}\,(3) that $\Pi^{(1)}=\infty_S$, contradicting the assumption that $B$ is not a singleton. So $m<k=\dim \sg{S}_{\F}$.

(2):  The proof here is the same as the proof for $m$-schemes in \cite{IKS09}, which essentially goes back to \cite{Evd94}. 
We first prove the following claim.

\begin{claim}\label{claim_halve}
Suppose  $\Pi$  is a strongly antisymmetric linear $m$-scheme on $S\subseteq V$, where $m\geq 2$, and $B\in \Pi^{(1)}$ is not a singleton.
Let $x,y$ be distinct elements in $B$. Let $B'$ be the block in $\Pi_x^{(1)}$ containing $y$.
Then $1<|B'|\leq |B|/2$. 
\end{claim}

\begin{proof}
Let $\tau: V^2\to V^2$ be the permutation $(a,b)\mapsto (b,a)$.
Then $\tau\in\M_{2,2}$. 
Let $B''$ be the block in $\Pi^{(2)}$ containing $(x,y)$.
Note $\tau(B'')\in \Pi^{(2)}$ by Property (P1).
As $x,y\in B$, the projections $\pi_{2,1}$ and $\pi_{2,2}$ both map $B''$ and $\tau(B'')$ surjectively to $B$ by Property (P1).
So  $B'',\tau(B'')\subseteq B\times B$.
As $\tau(x,y)=(y,x)\neq (x,y)$, we have $\tau(B'')\neq B''$ by  strong antisymmetry of $\Pi$.
So $|B''|=|\tau(B'')|\leq |B|^2/2$.
Then by Property (P2), we have $|B'|=|B''|/|B|\leq |B|/2$.
Finally, assume to the contrary $|B'|=1$. Then $\pi_{2,1}|_{B''}: B''\to B$ and $\pi_{2,2}|_{B''}:B''\to B$ are bijective by Property (P2).
The map $\pi_{2,2}|_{B_1}\circ (\pi_{2,1}|_{B_1})^{-1}$ sends $x$ to $y$, and hence is a nontrivial permutation of $B$. But this contradicts strong antisymmetry of $\Pi$. So $|B'|>1$.
\end{proof}

 Lemma~\ref{lemma_bounddim}\,(2) now follows  from Claim~\ref{claim_halve}, Lemma~\ref{lem_fixing}, and induction on $m$.
\end{proof}

\section{Omitted Proofs in Section~\ref{sec_main}}\label{sec_pf2}

This section contains the proofs omitted in  Section~\ref{sec_main} except that of Lemma~\ref{lem_densecase}.
We prove Lemma~\ref{lem_densecase}  in Appendix~\ref{sec_dense}.

\begin{proof}[Proof of Lemma~\ref{lem_fastshrink2}]
Let $k\in [m-2]$, $x_1,\dots,x_k\in B$, $B' \in \B(\Pi_{x_1,\dots,x_k}^{(1)})$, and $C>0$ be as in Lemma~\ref{lem_fastshrink}.
In addition, let $t:= (\rho^2 \log\log\log\log n)^{1/3}$.
As $\min\{|B'|, |B|/|B'|\}\geq 2^{C k t }$, we have $|B|\geq 2^{2C k t}$.
We will find  $k'\in [k]$, $y_1,\dots,y_{k'}\in B$, and $B''\in \Pi_{y_1,\dots,y_{k'}}^{(1)}$
such that $B''\subseteq B$, $|B''|>1$ and $|B|/|B''|=2^{\Omega(k t)}=2^{\Omega(k (\rho^2 \log\log\log\log n)^{1/3})}$.
The condition $B''\neq B$ follows automatically from $|B''|>1$:
By Lemma~\ref{lem_linrel}, we have $\{y_1\}\in \Pi_{y_1,\dots,y_{k'}}^{(1)}$   and hence $y_1\not\in B''$. So $B''\neq B$. 

If $\Pi_{x_1,\dots,x_k}^{(1)}|_{B'}\neq\infty_{B'}$, we can find $B''\in\Pi_{x_1,\dots,x_k}^{(1)}$ contained in $B'$ satisfying $|B''|>1$.
In this case, let $k'=k$ and $y_i=x_i$ for $i\in [k]$, and we are done.
So assume $\Pi_{x_1,\dots,x_k}^{(1)}|_{B'}=\infty_{B'}$.

For $0\leq i\leq k$, let $T_i$ be the set of $x\in B'$ satisfying $\{x\}\in \Pi_{x_1,\dots,x_i}^{(1)}$.
We have $T_0=\emptyset$, $T_k=B'$, and $T_{i-1}\subseteq T_i$ for $i\in [k]$ by Lemma~\ref{lem_refine}.
So there exists $i\in [k]$ such that $|T_i \setminus T_{i-1}|\geq |B'|/k\geq 2^{C k t }/k=2^{\Omega(kt)}$.
Fix such $i$ and let $\Pi'=\Pi_{x_1,\dots,x_{i-1}}$. By Lemma~\ref{lem_fixing}, $\Pi'$ is a strongly antisymmetric linear $(m-i+1)$-scheme, where $m-i+1\geq m-k+1\geq 3$.

First assume $i=1$. 
 By Claim~\ref{claim_halve} in the proof of Lemma~\ref{lemma_bounddim},  $\{x_1\}$ is the only singleton in $\Pi_{x_1}^{(1)}$  contained in $B$.
So $T_1\setminus T_0=T_1=\{x_1\}$. As $i=1$, we have $|B'|\leq k|T_1\setminus T_0|= k$.
As $|B'|\geq 2^{C k t }$, this implies $t \leq \log k/(Ck)=O(1)$.
In this case, just let $k'=1$, $y_1=x_1$ and choose $B''$ to be any block in $\Pi^{(1)}_{x_1}$ contained in $B$ other than $\{x_1\}$.
By Claim~\ref{claim_halve}, we have $|B|/|B''|\geq 2=2^{\Omega(t)}$. Thus the lemma holds.

Now assume $i>1$.
Let $B_1$ be the block in $\Pi'^{(1)}$ containing $x_i$.
Consider arbitrary $x\in T_i\setminus T_{i-1}$.
Let $B_x$ be the block in $\Pi'^{(1)}$ containing $x$, and let $B'_x$ be the block in $\Pi'^{(2)}$ containing $(x_i,x)$.
As $x\not\in T_{i-1}$, we have $|B_x|>1$.
And as $x\in T_i$, we have $\{x\}\in \Pi_{x_1,\dots,x_i}^{(1)}=\Pi'^{(1)}_{x_i}$.
So $\{y\in S: (x_i,y)\in B'_x\}=\{x\}$.
It follows that $|B'_x|=|B_1|$ by Property (P2).
On the other hand, the projection $\pi_{2,2}$ maps $B'_x$ surjectively to $B_x$ by Property (P1).
So $|B'_x|\geq |B_x|$. 

We now consider the two cases $|B_x|<|B'_x|$ and $|B_x|=|B'_x|$ separately. Assume $|B'_x|>|B_x|$.
Then $\Pi'^{(1)}_x$ contains a block $B''_x:=\{y\in B: (x,y)\in B'_x\}\subseteq B$ of cardinality $|B'_x|/|B_x|>1$ by Property (P2).
If $|B_x|<|B|^{1/2}$, we have $|B|/|B_x|\geq |B|^{1/2}=2^{\Omega(k t)}$.
In this case, we may choose $B''=B_x$, $k'=i-1$ and $y_j=x_j$ for $j\in [i-1]$.
On the other hand, if $|B_x|\geq |B|^{1/2}$, we have $|B|/|B''_x|\geq |B'_x|/|B''_x|=|B_x|\geq |B|^{1/2}=2^{\Omega(k t)}$.
In this case, we may choose $B''=B''_x$, $k'=i$, $y_i=x$, and $y_j=x_j$ for $j\in [i-1]$.

So we may assume $|B_x|=|B'_x|=|B_1|$. Then  $\{x\}\in\Pi'^{(1)}_{x_i}$. In fact, as $x\in T_i\setminus T_{i-1}$ is arbitrary, we may assume $|B_y|=|B'_y|=|B_1|$ for all $y\in T_i\setminus T_{i-1}$, where $B_y$ is the block in $\Pi'^{(1)}$ containing $y$ and $B'_y$ is the block in $\Pi'^{(2)}$ containing $(x_i,y)$.
Consider one such $y$ different from $x$ and note $\{y\}\in   \Pi'^{(1)}_{x_i,x}$ since $y\in T_i$.
Let $B^*$ be the block in $\Pi'^{(3)}$ containing $(x_i,x,y)$. Then $|B^*|=|B_1|$ since $\{x\}\in\Pi'^{(1)}_{x_i}$ and $\{y\}\in \Pi'^{(1)}_{x_i,x}$.
So $|B^*|=|B_x|=|B_y|$. It follows that the maps $\pi_{3,2}|_{B^*}:B^*\to B_x$ and $\pi_{3,3}|_{B^*}:B^*\to B_y$ are bijective. 
As $\pi_{3,3}|_{B^*}\circ (\pi_{3,2}|_{B^*})^{-1}\in\widetilde{\M}_{\Pi'}$ sends $x$ to $y$, we have $|B_x|=|B_y|$ and $B_x\neq B_y$ by strong antisymmetry of $\Pi'$.
As this holds for any distinct $x,y\in T_i\setminus T_{i-1}$, 
we have $|B|/|B_x|\geq |T_i\setminus T_{i-1}|=2^{\Omega(k t)}$ for $x\in T_i\setminus T_{i-1}$.
Then we may choose $B''=B_x$ for some $x\in T_i\setminus T_{i-1}$, and let $k'=i-1$ and $y_j=x_j$ for $j\in [i-1]$.
\end{proof}

\begin{proof}[Proof of Lemma~\ref{lem_bij}]
 By Lemma~\ref{lem_linrel} and Lemma~\ref{lem_closedness}, for $t\in\N$, the set of $x\in A$  for which there exist precisely $t$ elements $y$ in $A$ satisfying $\sigma_k(y)=\sigma_k(x)$ is in $\B(\Pi^{(k)})$. As $A\in \Pi^{(k)}$,   $\#\{y\in A: \sigma_k(y)=\sigma_k(x)\}$ is  constant independent of $x\in A$. So for $x\in A'=\sigma_k(A)$, $\#\{y\in A: \sigma_k(y)=x\}$ is independent of $x$ and equals $|A|/|A'|$.  
 
 Let $x,x'\in A$ such that $\sigma_k(x)=\sigma_k(x')$. We want to prove $x=x'$. Let $T$ be the set of $y\in A$ satisfying $\sigma_k(y)=\sigma_k(x)$. Then $|T|=|A|/|A'|$.
We also have $T\in\B(\Pi_x^{(k)})$ by Lemma~\ref{lem_linrel}. 
 Let $B$ be the block in $\Pi_x^{(k)}$ containing $x'$. 
As $x'\in T$, we have $B\subseteq T$ and hence $|B|\leq |T|=|A|/|A'|$. 
By Lemma~\ref{lem_fixing}, $\Pi_x$ is a strongly antisymmetric linear $(m-k)$-scheme.
 As $m-k>\log |A|/|A'|\geq \log |B|$, we have $B=\{x'\}$ by Lemma~\ref{lemma_bounddim}\,(2). 

Let $B'$ be the block in $\Pi^{(2k)}$ containing $(x,x')$. Denote by $\pi_1: V^{2k}\to V^k$ (resp. $\pi_2:V^{2k}\to V^k$) the projection to the first (resp. last) $k$ coordinates. Note $\pi_1,\pi_2\in \M_{2k,k}$.
 We have $\pi_1(B')=\pi_2(B')=A$ and $|B'|=|A||B|=|A|$. So $\pi_1|_{B'}: B'\to A$ and $\pi_2|_{B'}: B'\to A$ are bijections. Then $\pi_2|_{B'}\circ (\pi_1|_{B'})^{-1}$ is in $\widetilde{\M}_{\Pi}$ and sends $x$ to $x'$. By strong antisymmetry of $\Pi$, we have $x=x'$, as desired.
\end{proof}

\begin{proof}[Proof of Lemma~\ref{lem_schemepower}]
Consider $i\in [m']$ and $B\in \Pi^{(ki)}$ such that  $B\subseteq A^{i}$.
We prove that $\sigma_k^{(i)}|_{B}: B\to \sigma_k^{(i)}(B)$ is bijective.
 Consider $\mathbf{x}=(x_1,\dots,x_i), \mathbf{y}=(y_1,\dots,y_i)\in B$ such that $\sigma_k^{(i)}(\mathbf{x})=\sigma_k^{(i)}(\mathbf{y})$, i.e.,
 $\sigma_k(x_j)=\sigma_k(y_j)$ for $j\in [i]$. As $B\subseteq A^{i}$, we have $x_j,y_j\in A$ for $j\in [i]$. As $\sigma_k|_A: A\to A'$ is bijective, we have $x_j=y_j$ for $j\in [i]$. So $\mathbf{x}=\mathbf{y}$, i.e., $\sigma_k^{(i)}|_{B}: B\to \sigma_k^{(i)}(B)$ is bijective.

It remains to prove that $\Pi'$ is a well defined strongly antisymmetric linear $m'$-scheme on $A'$.
Let $i\in [m']$. We first prove that $\Pi'^{(i)}$ is a well defined partition of $A'^i$. Note the union of the sets in $\Pi'^{(i)}$ is $\sigma_k^{(i)}(A^i)=\sigma_k(A)^i=A'^i$.
To prove that $\Pi'^{(i)}$ is a partition, consider $B,B'\in \Pi^{(ki)}$ satisfying $\sigma_k^{(i)}(B)\cap  \sigma_k^{(i)}(B')\neq\emptyset$. We want to show $\sigma_k^{(i)}(B)= \sigma_k^{(i)}(B')$.
To see this, let $T$ be the set of $x\in B$ for which there exists $y\in B'$ satisfying $\sigma_k^{(i)}(x)= \sigma_k^{(i)}(y)$.
Then $T\in \B(\Pi^{(ki)})$ by  Lemma~\ref{lem_closedness}.
As  $ \sigma_k^{(i)}(B)\cap  \sigma_k^{(i)}(B')\neq\emptyset$, we have $T\neq\emptyset$.
As $T\subseteq B$ and $B\in \Pi^{(ki)}$, we have $T=B$.
This implies $\sigma_k^{(i)}(B)\subseteq \sigma_k^{(i)}(B')$.
Similarly, $\sigma_k^{(i)}(B')\subseteq \sigma_k^{(i)}(B)$. So $\sigma_k^{(i)}(B)=\sigma_k^{(i)}(B')$, as desired.

Consider $B\in \Pi'^{(i)}$, $B'\in \Pi'^{(i')}$ and $\tau\in\M_{i,i'}$ such that $\tau(B)\cap B'\neq\emptyset$.  Choose $\widetilde{B}\in \Pi^{(ki)}$ and $\widetilde{B}'\in \Pi^{(ki')}$ such that $\widetilde{B}\subseteq A^i$,   $\widetilde{B}'\subseteq A^{i'}$, $\sigma_k^{(i)}(\widetilde{B})=B$ and $\sigma_k^{(i')}(\widetilde{B}')=B'$.
By  Property (P2) of $\Pi$, when $x$ ranges over $\widetilde{B}$, the number of $y\in\widetilde{B}'$
satisfying $(\tau\circ \sigma_k^{(i)})(x)=\sigma_k^{(i')}(y)$ is   independent of $x$. And this number is
positive iff $(\tau\circ \sigma_k^{(i)})(x)\in \sigma_k^{(i')}(\widetilde{B}')=B'$.
As $\tau(B)\cap B'\neq\emptyset$ and $\tau(B)=(\tau\circ \sigma_k^{(i)})(\widetilde{B})$, we know this number is indeed positive.
So $\tau(B)\subseteq B'$. 
 Again by  Property (P2) of $\Pi$, when $y$ ranges over $\widetilde{B}'$, the number of $x\in\widetilde{B}$
satisfying $(\tau\circ \sigma_k^{(i)})(x)=\sigma_k^{(i')}(y)$ is   independent of $y$. So the map $\tau\circ \sigma_k^{(i)}|_{\widetilde{B}}: \widetilde{B}\to B'$ is a surjective $d$-to-1 map for some $d\in\N^+$.
As $\sigma_k^{(i)}|_{\widetilde{B}}: \widetilde{B}\to B$ is bijective, we know $\tau|_B: B\to B'$ is also a surjective $d$-to-1 map.
This proves Properties (P1) and (P2) of $\Pi'$. So $\Pi'$ is a linear $m'$-scheme.

Now further assume $\tau|_B: B\to B'$ is bijective. We claim there exists a bijection $\tau': \widetilde{B}\to\widetilde{B}'$ in $\widetilde{\M}_\Pi$ making the following diagram commute.
 
\begin{tikzcd}[column sep=large]
\widetilde{B} \arrow[r, "\tau'"] \arrow[d, "\sigma_k^{(i)}"']
&\widetilde{B}' \arrow[d, "\sigma_k^{(i')}"] \\
B  \arrow[r,  "\tau"]
& B'\nospacepunct{.}
\end{tikzcd}
 
 To see this, define $\Delta:=\{(x,y)\in \widetilde{B}\times \widetilde{B}': (\tau\circ\sigma_k^{(i)})(x)=\sigma_k^{(i')}(y)\}$.
 Then $\Delta\in \B(\Pi^{(ki+ki')})$.
As $\sigma_k^{(i)}|_{\widetilde{B}}:\widetilde{B}\to B$, $\sigma_k^{(i')}|_{\widetilde{B}'}:\widetilde{B}'\to B'$ and $\tau|_B: B\to B'$ are bijective, for every $x\in \widetilde{B}$, there exists unique $y\in \widetilde{B}'$ satisfying $(x,y)\in \Delta$. So $|\Delta|=|\widetilde{B}|=|\widetilde{B}'|$.
Let $\pi_1$ (resp. $\pi_2$) be the projection from $ V^{ki}\times V^{ki'}$ to $V^{ki}$ (resp. $V^{ki'}$). Then $\pi_1$ (resp. $\pi_2$) maps $\Delta$ bijectively to $\widetilde{B}$ (resp. $\widetilde{B}'$). By the definition of $\Delta$, the map $\tau':=\pi_2|_\Delta\circ (\pi_1|_\Delta)^{-1}\in \widetilde{\M}_\Pi$ makes the above diagram commute. This proves the claim.

By the claim just proved, every nontrivial permutation in $\widetilde{\M}_{\Pi'}$ lifts to a nontrivial permutation in $\widetilde{\M}_\Pi$. Therefore, as $\Pi$ is strongly antisymmetric, so is $\Pi'$.
\end{proof}

\begin{proof}[Proof of Claim~\ref{claim_final}]
Consider an arbitrary block $U\in \Pi'^{(1)}_{x_1,\dots,x_r}$ satisfying $U\subseteq A''$. 
We will find $U'\in \Pi_y^{(k)}$ such that $U'\subseteq A$ and $\sigma_k(U')=U$.
As $\sigma_k|_{A}: A\to A'$ is bijective, this implies $|U'|=|U|$.
The claim follows by choosing $T$ to be the (disjoint) union of these sets $U'$ where $U$ ranges over the blocks  in $\Pi'^{(1)}_{x_1,\dots,x_r}$ that are subsets of $A''$.

As $U\in \Pi'^{(1)}_{x_1,\dots,x_r}$, there exists $Z\in \Pi'^{(r+1)}$ such that $U=\{z\in A': (x_1,\dots,x_r,z)\in  Z\}$. 
Note $Z\subseteq A'^{r+1}$.
Fix $z_0\in U$ and choose $z_0'\in A$ such that $\sigma_k(z_0')=z_0$.
Choose $\widetilde{Z}\in  \Pi^{(k(r+1))}$ 
to be the block containing $(y_1,\dots,y_r,z_0')$.
Then $\widetilde{Z}\subseteq A^{r+1}$.

Note $\sigma_k^{(r+1)}(y_1,\dots,y_r,z_0')=(x_1,\dots,x_r,z_0)\in Z$. 
So $\sigma_k^{(r+1)}(\widetilde{Z})\cap Z\neq\emptyset$.
As $\sigma_k^{(r+1)}(\widetilde{Z})\in \Pi'^{(r+1)}$ by the definition of $\Pi'$, 
we have $\sigma_k^{(r+1)}(\widetilde{Z})=Z$.
Let $U'=\{z\in A: (y_1,\dots,y_r,z)\in \widetilde{Z}\}\subseteq A$. Then $U'\in \Pi_y^{(k)}$.

It remains to prove   $\sigma_k(U')=U$. Consider $z\in\sigma_k(U')\subseteq A'$ and choose $z'\in U'$ such that $z=\sigma_k(z')$.
As $z'\in U'$, we have $(y_1,\dots,y_r,z')\in\widetilde{Z}$.
So $\sigma_k^{(r+1)}(y_1,\dots,y_r,z')=(x_1,\dots,x_r,z)\in Z$. 
Therefore $z\in U$. This proves $\sigma_k(U')\subseteq U$.

Now consider $z\in U$. Then $(x_1,\dots,x_r,z)\in Z$. As $\sigma_k^{(r+1)}(\widetilde{Z})=Z$, there exists $(y'_1,\dots,y'_r,z')\in \widetilde{Z}$
such that $\sigma_k^{(r+1)}(y'_1,\dots,y'_r,z')=(x_1,\dots,x_r,z)$, i.e., $\sigma_k(z')=z$ and $\sigma_k(y'_i)=x_i=\sigma_k(y_i)$ for $i\in [r]$.
As $\sigma_k|_A: A\to A'$ is bijective, we have $y'_i=y_i$ for $i\in [r]$. So $(y_1,\dots,y_r,z')\in\widetilde{Z}$. It follows that $z'\in U'$. So $z=\sigma_k(z')\subseteq \sigma_k(U')$. This proves $\sigma_k(U')\supseteq U$. Therefore $\sigma_k(U')=U$.
 \end{proof}

\section{Proof of Lemma~\ref{lem_densecase}}\label{sec_dense}

We prove Lemma~\ref{lem_densecase} in this section, which addresses the case that $B\in\Pi^{(1)}$ is  dense in $\sg{B}$.

\subsection{Additional Notations and Preliminaries.}

We first introduce  additional notations and preliminaries that are used in the proof of Lemma~\ref{lem_densecase}.

\subparagraph*{Fourier analysis on finite abelian groups.} 
For a finite abelian group $A$, write $\widehat{A}$ for the dual group $\mathrm{Hom}(A, \mathbb{C}^\times)$. Elements in $\widehat{A}$ are called \emph{characters} of $A$. 
The complex conjugate $\overline{\chi}$ of a character $\chi$ is again a character.
For a function $f: A\to \mathbb{C}$ and a character $\chi\in\widehat{A}$, define the \emph{$\chi$-th Fourier coefficient} $\widehat{f}(\chi):=\E_{a\in A}[f(a) \overline{\chi(a)}]$. Then $f=\sum_{\chi\in \widehat{A}} \widehat{f}(\chi) \chi$. We also have   \emph{Parseval's identity}: 
\[
\E_{a\in A}\left[|f(a)|^2\right]=\sum_{\chi\in \widehat{A}} |\widehat{f}(\chi)|^2.
\]

\subparagraph*{Additive combinatorics.} For  $A\subseteq V$, write $\F A$ for the cone $\{c a: c\in\F, a\in A\}\subseteq V$.
Let $A^\pm:=A\cup(-A)\cup \{0\}\subseteq \F A$. Define $h(A)$ to be smallest  $k\in\N^+$ such that $k A^\pm = \sg{A}$.
The problem of bounding $h(A)$ in terms of $\mu(A)$ has been studied in \cite{HR96, KL03}. In particular, the following bound was obtained in \cite{HR96}.
\begin{theorem}[{\cite[Lemma~3]{HR96}}]\label{thm_smallgen}
$h(A)\leq\max\left\{2, \lfloor\frac{3}{2\mu(A)}\rfloor\right\}$.
\end{theorem}

For $A\subseteq V$, the \emph{additive energy} $E(A)$ of $A$ is defined by:
\begin{align*}
E(A):&=\#\{(a_1,a_2,a_3,a_4)\in A^4: a_1+a_2=a_3+a_4\}\\
&=\#\{(a_1,a_2,a_3,a_4)\in A^4: a_1-a_3=a_4-a_2\}.
\end{align*}
If $|A+A|/|A|$ is small, then $E(A)$ is large, as the map $A\times A\to A+A$ sending $(a,b)$ to $a+b$ has a lot of ``collisions''.
The converse of this statement is false. Nevertheless, the famous Balog-Szemer\'edi-Gowers  Theorem \cite{BS94, Gow98} states that if $E(A)$ is large enough, then there always exists a dense subset $A'\subseteq A$ such that $|A'+A'|/|A'|$ is small.
We need the following variant of this theorem tailored to  linear $m$-schemes:

\begin{theorem}[Balog-Szemer\'edi-Gowers Theorem for linear $m$-schemes]\label{thm_bsg}
Let $m\geq 4$.
Let $\Pi$ be a   linear $m$-scheme and  $B \in\Pi^{(1)}$  such that   $E(B)\geq \gamma |B|^3$, where $\gamma>0$. Then there exists  $B'\subseteq B$ such that $|B'|\geq \gamma|B|/3$ and $|B'-B'|< 2^{17}\gamma^{-9}|B|$.
Moreover,   $B'$ can be chosen such that $B'\in \B(\Pi^{(1)}_x)$ for some $x\in B$.
\end{theorem}

Theorem~\ref{thm_bsg} basically follows from the usual form of the Balog-Szemer\'edi-Gowers Theorem except the last claim that we may assume  $B'\in \B(\Pi^{(1)}_x)$ for some $x\in B$. This claim too can be verified by following Gowers' proof in \cite{Gow98}. For the sake of completeness, we present a proof following \cite{Gow98}, where some steps are simplified using properties of linear $m$-schemes.

\begin{proof}[Proof of Theorem~\ref{thm_bsg}]
We will prove that   there exist  $x\in B$ and $B'\in\Pi_x^{(1)}$ such that  $B'\subseteq B$, $|B'|\geq \gamma|B|/3$, and for any $a_1,a_2\in B'$, the number of solutions of the equation
\[
a_1-a_2=(x_1-y_1)-(x_2-y_2)-(x_3-y_3)+(x_4-y_4)
\]
in the variables $x_i,y_i\in B$ ($i=1,2,3,4$) is greater than $2^{-17}\gamma^{9}|B|^7$. Note that this implies $|B'-B'|<|B|^8/(2^{-17}\gamma^{9}|B|^7)=2^{17}\gamma^{-9}|B|$, as claimed.

For $z\in B-B$, write $\nu^-(z)$ for the number of $(x,y)\in B\times B$ satisfying $x-y=z$.
Then 
\[
\sum_{z\in B-B} \nu^-(z)=|B|^2 \quad\text{and}\quad \sum_{z\in B-B} \nu^-(z)^2=E(B).
\]
Let $T:=\{z\in B-B: \nu^-(z)\geq \gamma|B|/2\}$.
For $x\in B$, let  $N(x):=\{y\in B: x-y\in T\}$.
For $y\in B$, let  $N'(y):=\{x\in B: y\in N(x)\}=\{x\in B: x-y\in T\}$.
Note 
\begin{equation}\label{eq_bsg1}
\sum_{z\in T} \nu^-(z)^2=\sum_{z\in B-B} \nu^-(z)^2-\sum_{z\in (B-B)\setminus T} \nu^-(z)^2\geq E(B)-|B|^2\cdot \gamma|B|/2\geq \gamma|B|^3/2.
\end{equation}
On the other hand,  
\begin{equation}\label{eq_bsg2}
\sum_{x\in B} |N(x)|=\sum_{x\in B} |N'(x)|=\sum_{z\in T} \nu^-(z)\geq  \sum_{z\in T} \nu^-(z)^2/|B|.
\end{equation}
By \eqref{eq_bsg1} and \eqref{eq_bsg2}, we have  $\sum_{x\in B} |N(x)|=\sum_{x\in B} |N'(x)|\geq \gamma|B|^2/2$. 
Note $N(x), N'(x)\in \B(\Pi_x^{(1)})$ for $x\in B$ by Lemma~\ref{lem_closedness}.
By Property (P2) of $\Pi$, $|N(x)|$ (resp. $|N'(x)|$) is constant when $x$ ranges over $B$. Therefore,
$|N(x)|=|N'(x)|\geq \gamma|B|/2 \quad\text{for all}~x\in B$.

Let $N$ be the cardinality of the sets $N(x)$ and $N'(x)$. Then $N\geq  \gamma|B|/2 $.
For $x\in B$, let $M(x)$ be the number of pairs $(y,z)\in N'(x)\times N'(x)$ satisfying $|N(y)\cap N(z)|\leq \gamma^2|B|/36$.
By Property (P2) of $\Pi$, $M(x)$ is independent of $x\in B$.
Choose arbitrary $x_0\in B$. 
Note that for $x,y\in B$, we have $x\in N(y)$ iff $y\in N'(x)$.
Therefore
\begin{equation}\label{eq_bsg3}
\begin{aligned}
M(x_0)&=\sum_{x\in B} M(x)/|B|
=\sum_{x\in B} \sum_{\substack{y,z\in N'(x)\\  |N(y)\cap N(z)|\leq \gamma^2|B|/36}}1/|B|\\
&=\sum_{\substack{y,z\in B\\ |N(y)\cap N(z)|\leq \gamma^2|B|/36}} \sum_{x\in N(y)\cap N(z)} 1/|B|\leq \gamma^2|B|^2/36\leq N^2/9.
\end{aligned}
\end{equation}
Consider the undirected graph $G$ on the vertex set $N'(x_0)$ such that $(y,z)\in N'(x_0)\times  N'(x_0)$ is an edge in $G$  iff $|N(y)\cap N(z)|\leq \gamma^2|B|/36$.
By \eqref{eq_bsg3}, the average degree of $G$ is at most $N/9$. 
By Markov's inequality, at most $(1/3)$-fraction of vertices $y\in N'(x_0)$ satisfy $\deg(y)\geq N/3$.
Let 
\[
B':=\{y\in N'(x_0): \deg(y)\leq N/3\}\subseteq B.
\]
Then $|B'|\geq 2N/3\geq \gamma|B|/3$. Also note $B'\in \B(\Pi_{x_0}^{(1)})$ by Lemma~\ref{lem_closedness}.

By definition, for all $y\in B'$, we have
\[
\#\{z\in N'(x_0): |N(y)\cap N(z)|\leq \gamma^2|B|/36\} \leq N/3.
\]
Consider arbitrary $a_1,a_2\in B'$. By the union bound, we have
\begin{equation}\label{eq_boundz}
\#\{z\in N'(x_0): |N(a_i)\cap N(z)|>\gamma^2|B|/36~\text{for}~i=1,2\}\geq N-2(N/3)\geq \gamma|B|/6.
\end{equation}
Consider any $z\in N'(x_0)$ satisfying $|N(a_i)\cap N(z)|>\gamma^2|B|/36$ for $i=1,2$. For any $w\in N(a_1)\cap N(z)$, we have $\nu^-(a_1-w),  \nu^-(z-w)\geq \gamma|B|/2$, yielding at least $(\gamma|B|/2)^2$ representations of $a_1-z$ of the form
\[
a_1-z=(x_1-y_1)-(x_2-y_2)
\]
such that $x_1,y_1,x_2,y_2\in B$, $x_1-y_1=a_1-w$, and $y_2-x_2=z-w$.
As the number of choices for $w\in N(a_1)\cap N(z)$ is at least $\gamma^2|B|/36$, there are at least $(\gamma|B|/2)^2\cdot \gamma^2|B|/36=\gamma^4|B|^3/144$ representations $a_1-z=(x_1-y_1)-(x_2-y_2)$ where $x_1,y_1,x_2,y_2\in B$. (Note different choices of $w$ yield different representations as $w=a_1-x_1+y_1$ can be recovered from a representation.)
Similarly, there are at least $\gamma^4|B|^3/144$ representations $a_2-z=(x_3-y_3)-(x_4-y_4)$ where $x_3,y_3,x_4,y_4\in B$.
Combining, we obtain at least $(\gamma^4|B|^3/144)^2$ representations of $a_1-a_2$ of the form $a_1-a_2=((x_1-y_1)-(x_2-y_2))-((x_3-y_3)-(x_4-y_4))$ using each $z$. 
By \eqref{eq_boundz}, there are at least $  \gamma|B|/6$ choices of $z$, yielding at least $\gamma|B|/6\cdot (\gamma^4|B|^3/144)^2> 2^{-17}\gamma^9|B|^7$ representations of $a_1-a_2$ of the form
\[
a_1-a_2=(x_1-y_1)-(x_2-y_2)-(x_3-y_3)+(x_4-y_4).
\]
where $x_i,y_i\in B$ for $i=1,2,3,4$. (Again, different choices of $z$ yield different representations as $z=a_1-(x_1-y_1)+(x_2-y_2)$ can be recovered from a representation.) 
This completes the proof.
\end{proof}

\subsection{Constructible Sets}

We define and discuss the notion of \emph{constructible sets} in this subsection, which will be useful later.

\begin{definition}[constructible set]
Let $\Pi$ be a linear $m$-scheme on $S$, and $k\in [m]$. Define 
\[
\mathcal{S}_{\Pi, k}:=\{\tau(B):  B\in \Pi^{(k)}, \tau\in\M_{k,1}\}.
\]
We say a subset $T$ of $V$ is \emph{$(\Pi,k)$-constructible} if $T\in \B(\mathcal{S}_{\Pi, k})$, i.e., $T$ is a union of sets $T_1,\dots,T_r$ where each $T_i$ is the image of a block $B_i\in \Pi^{(k)}$ under a map $\tau_i\in\M_{k,1}$. 
\end{definition}

\begin{lemma}\label{lem_constset}
We have 
\begin{enumerate}
\item For $T\in \B(\mathcal{S}_{\Pi, k})$, $T\cap S\in \B(\Pi^{(1)})$.
\item Suppose $T\in \B(\mathcal{S}_{\Pi, k})$ and  $T'\in \B(\mathcal{S}_{\Pi, k'})$ where $k+k'\leq m$. Then $T\cap T', T\setminus T'\in  \B(\mathcal{S}_{\Pi, k})$.
\end{enumerate}
\end{lemma}
\begin{proof}
The first claim holds by  Lemma~\ref{lem_closedness}\,(3).
For the second claim, it suffices to show that for $B\in \Pi^{(k)}$, $B'\in \Pi^{(k')}$, $\tau\in\M_{k,1}$ and $\tau'\in \M_{k',1}$, we have $\tau(B)\cap \tau'(B'), \tau(B)\setminus\tau'(B')\in\B(\mathcal{S}_{\Pi, k})$. To see this, let $\Delta=\{(x,y)\in B\times B': \tau(x)=\tau'(y)\}$. Then $\Delta\in \B(\Pi^{(k+k')})$ by Lemma~\ref{lem_linrel}.
Let $B''=\{x\in B: \exists\,y\in B'~\text{such that}~(x,y)\in \Delta\}$. Then $B'', B\setminus B''\in \B(\Pi^{(k)})$  by Lemma~\ref{lem_closedness}.
Note $\tau(B)\cap \tau'(B')=\tau(B'')$ and $\tau(B)\setminus \tau'(B')=\tau(B\setminus B'')$. So $\tau(B)\cap \tau'(B'), \tau(B)\setminus\tau'(B')\in\B(\mathcal{S}_{\Pi, k})$.
\end{proof}

We also need the following technical lemma.

\begin{lemma}\label{lem_csubspace}
Let $W$ and $W'$ be linear subspaces of $V$ such that $W\subseteq W'$. Let $d=\dim W'-\dim W$.
Suppose $W$ is $(\Pi,k)$-constructible  for some $k\in [m]$.
Also suppose $W'\subseteq W+t(\F S)$ for some $t\in \N$ satisfying $k+2dt\leq m$.
Then there exists $x\in S^{dt}$ such that $W'$ is $(\Pi_x, k+dt)$-constructible. 
\end{lemma}
\begin{proof}
Choose $x_1,\dots,x_d\in W'$ such that $W'= W+\sum_{i=1}^d \F x_i$. 
As $W'\subseteq W+t(\F S)$, we may assume $x_1,\dots,x_d\in t(\F S)$.
For $i\in [d]$, write $x_i=\sum_{j=1}^t c_{i,j} x_{i,j}$ for some $c_{i,j}\in\F$ and $x_{i,j}\in S$.
Let $x=(x_{i,j})_{i\in [d], j\in [t]}\in S^{dt}$.

Consider $B\in \Pi^{(k)}$ and $\tau\in\M_{k,1}$, so that $\tau(B)\in \mathcal{S}_{\Pi, k}$. 
Let $B'=B\times \{x_{1,1}\}\times\{x_{1,2}\}\times\dots\times\{x_{d,t}\}\in \B(\Pi_x^{(k+dt)})$.
For $s=(s_1,\dots,s_d)\in \F^{d}$, let $\tau_s: V^k\times V^{dt}\to V$ be the map sending $(y, (z_{i,j})_{i\in [d], j\in [t]})$ 
 to $\tau(y)+\sum_{i=1}^d  s_i  \cdot \left(\sum_{j=1}^tc_{i, j} z_{i,j}\right)$. We have $\tau_s\in \M_{k+dt,1}$ for $s\in\F^d$.
Then 
\[
\tau(B)+\sum_{i=1}^{d} \F x_i=\bigcup_{s\in \F^{d}} \tau_s(B')\in \B(\mathcal{S}_{\Pi_x, k+dt}).
\]
As $W$ is a union of sets of the form $\tau(B)$ with $B\in \Pi^{(k)}$ and $\tau\in\M_{k,1}$, 
 we have $W'= W+\sum_{i=1}^d \F x_i\in \B(\mathcal{S}_{\Pi_x, k+dt})$.
\end{proof}

\subsection{The Decomposition Theorem}
As mentioned before, one key idea in the proof is  reducing the density of $B$. We know that if $|B|>1$, then the set $B\setminus\{x\}$ contains two different blocks $B', B''\in \Pi_x^{(1)}$ of equal size (see the proof of Claim~\ref{claim_halve}). Suppose  $\sg{B'}=\sg{B''}=\sg{B}$. Then 
\[
\mu(B')= |B'|/|\sg{B'}| \leq  \frac{1}{2}|B|/|\sg{B}|=\mu(B)/2.
\]
In this case, replacing $\Pi$ by $\Pi_x$ and $B$ by $B'$ reduces the density by at least a factor of two. Continuing this process, we would obtain a block of density at most $\exp(-k)$ in $k$ steps.

Unfortunately, we do not know if $\sg{B'}=\sg{B''}=\sg{B}$ always holds. If $\sg{B'}\subsetneq \sg{B}$, then it might happen that $\mu(B')\gg \mu(B)$. Note that if this occurs, then $B$ is ``overrepresented'' in the  proper subspace $\sg{B'}$, i.e., $\mu_{\sg{B'}}(B)\geq \mu(B')\gg \mu(B)$.

To address this problem, we want to find a subspace $W\subseteq \sg{B}$ such that $B$ is ``psedorandom'' within $W$ in the sense that $\mu_{W\cap\sg{B'}}(B)\approx \mu_{W}(B)$. Then, by restricting to the subspace $W$, we can make sure that overrepresentation does not occur.

How do we find $W$? Consider the case $\mu_{\sg{B'}}(B)\gg \mu(B)$ and assume $\sg{B'}$ is a hyperplane of $\sg{B}$ for simplicity.
It is easy to see that   the characteristic function $1_B$ must have a non-negligible correlation with some nontrivial character $\chi\in \widehat{\sg{B}}$ vanishing on $\sg{B'}$,
i.e., $|\widehat{1}_B(\chi)|\geq \epsilon$ where $\epsilon>0$ is non-negligible. So, we could choose $W$ to be the intersection of $\ker(\chi)$ where $\chi$ ranges over the set of nontrivial characters satisfying $|\widehat{1}_B(\chi)|\geq \epsilon$. Then restricting to the subspace $W$  ``quotients out'' these characters.
 
However, there are two problems with this idea. The first one is that we need to make sure restricting to $W$ maintains the linear $m$-scheme structure.
This is solved by considering only the nontrivial characters whose kernels are constructible sets  instead of all characters. 

The second problem is that the set $W\cap B$ may simply be empty, in which case the statement $\mu_{W\cap\sg{B'}}(B)\approx \mu_{W}(B)$ is trivially true but  no longer useful. In this case, instead of restricting to the subspace $W$, we restrict to $W'=W+\F x$ for some $x\in B$.
As $x\in B$ may vary, we actually find a collection of subspaces $W_i$ of $\sg{B}$ instead of a single subspace.

The main result of this subsection is a decomposition theorem that allows us to find the subspaces $W_i$. To formally state this result, we need the following definitions.

\begin{definition}\label{defi_cspace}
Let $\Pi$ be a linear $m$-scheme on $S$. Let $B\in \Pi^{(1)}$ and $1\leq k\leq m/2$.
define $\mathcal{W}_{\Pi, k, B}$ to be the set of  subspaces  $W\subseteq \sg{B}$ such that (1)
 the codimension of $W$ in $\sg{B}$ is at most $k$, and (2) $W$ is $(\Pi_{x},k)$-constructible for some $x\in S^k$. 
\end{definition}

\begin{definition}\label{defi_specialch}
Let $\Pi$ be a linear $m$-scheme on $S$. Let $B\in \Pi^{(1)}$, $1\leq k\leq m/2$, and $0<\epsilon<1$. Define $\mathcal{X}_{\Pi, k, B,\epsilon}$ to be the set of  nontrivial  characters $\chi\in\widehat{\sg{B}}$
such that    $|\widehat{1_B}(\chi)|\geq \epsilon$ and   $\ker(\chi)\in \mathcal{W}_{\Pi, k, B}$.
\end{definition}

The next lemma states that if $B$ is not pseudorandom against a subspace $W\in \mathcal{W}_{\Pi, k, B}$, then there exists a nontrivial character $\chi$ such that $|\widehat{1}_B(\chi)|$ is non-negligible and $\ker(\chi)$ is constructible.

\begin{lemma}\label{lem_twocases}
Let $\Pi$ be a linear $m$-scheme on $S$. Let $B\in \Pi^{(1)}$, $1\leq k\leq m/2$, and $0<\epsilon<1$. Let $t=\lfloor\frac{3}{2\mu(B)}\rfloor+1$, $k'=k(t+1)$ and $\epsilon'=\epsilon/\ell^k$  (recall $\ell=\mathrm{char}(\F)=|\F|$). Suppose $m\geq 2k'$. Then one of the following is true:
\begin{enumerate}
\item $|\mu_W(B)-\mu(B)|\leq \epsilon$ for all $W\in \mathcal{W}_{\Pi, k, B}$.
\item $\mathcal{X}_{\Pi, k', B, \epsilon'}\neq\emptyset$.
\end{enumerate}
\end{lemma}
\begin{proof}
Suppose (1) does not hold. Choose $W\in \mathcal{W}_{\Pi, k, B}$ such that $|\mu_{W}(B)-\mu(B)|>\epsilon$.
Let $G\subseteq \widehat{\sg{B}}$ be the subgroup of characters $\chi$ vanishing on $W$. Then $|G|=|V|/|W|\leq \ell^k$.
Let $1_W: \sg{B}\to\C$ be the characteristic function of $W$.
Note  $\widehat{1_{W}}(\chi)=|W|/|\sg{B}|$ for $\chi\in G$ and $\widehat{1_{W}}(\chi)=0$ for $\chi\in \widehat{\sg{B}}\setminus G$.
So 
\begin{equation}\label{eq_ft}
1_{W}=|W|/|\sg{B}|\cdot \sum_{\chi\in G} \chi.
\end{equation}
Also  note for $\chi\in \widehat{\sg{B}}$, we have 
\begin{equation}\label{eq_expt}
\E_{a\in B}[\chi(a)]=\mu(B)^{-1}\E_{a\in \sg{B}}[1_B(a)\chi(a)]=\mu(B)^{-1}\widehat{1_B}(\overline{\chi}).
\end{equation}
Denote by $\chi_0$ the trivial character in $\widehat{\sg{B}}$.  Then we have
\begin{equation}\label{eq_fourier}
\begin{aligned}
\mu_W(B)&=\sum_{a\in B} 1_W(a)/|W| 
\stackrel{\eqref{eq_ft}}{=} \sum_{\chi\in G}\sum_{a\in B} \chi(a)/|\sg{B}|\\
&=\left(\sum_{\chi\in G}\E_{a\in B}[\chi(a)]\right)\cdot \mu(B)
\stackrel{\eqref{eq_expt}}{=}\mu(B)+\sum_{\chi\in G\setminus\{\chi_0\}} \widehat{1_B}(\overline{\chi}).
\end{aligned}
\end{equation}

As $|\mu_{W}(B)-\mu(B)|>\epsilon$, there exists $\chi^*\in G\setminus\{\chi_0\}$ such that $|\widehat{1_B}(\chi^*)|=|\widehat{1_B}(\overline{\chi^*})|\geq \epsilon/|G|\geq \epsilon'$ (the first equality holds as $1_B$ is real-valued).

As $\chi^*\in G$,  we have  $W\subseteq \ker(\chi^*)$. Note $\dim \ker(\chi^*)-\dim W\leq \dim \sg{B}-\dim W\leq k$.
Also note $\ker(\chi^*)\subseteq \sg{B}\subseteq t(\F B)$ by Theorem~\ref{thm_smallgen}.
As $W\in \mathcal{W}_{\Pi, k, B}$, we know $W$ is $(\Pi_x, k)$-constructible for some $x\in S^k$.
Then by Lemma~\ref{lem_csubspace}, $\ker(\chi^*)$ is $((\Pi_{x})_{x'}, k+kt)$-constructible for some $x'\in S^{kt}$.
It follows that $\ker(\chi^*)\in \mathcal{W}_{\Pi, k', B}$.
So $\chi^*\in \mathcal{X}_{\Pi, k', B, \epsilon'}$ and hence $\mathcal{X}_{\Pi, k', B, \epsilon'}\neq\emptyset$.
\end{proof}

Now we are ready to state and prove the \emph{decomposition theorem}.

\begin{theorem}[decomposition theorem]\label{thm_struct}
Let $\Pi$ be a linear $m$-scheme on $S$. Let $B\in \Pi^{(1)}$, $1\leq k'\leq m/4$, and $0<\epsilon'<1$. 
Let $t= \lfloor\frac{3}{2\mu(B)}\rfloor+1$.  Suppose $m\geq 2t+2$   and $\mathcal{X}_{\Pi, k', B, \epsilon'}\neq\emptyset$.
Let $H=\bigcap _{\chi\in \mathcal{X}_{\Pi, k', B, \epsilon'}} \ker(\chi)$, and let $\mathcal{C}$ be the set of subspaces $\{H+\F x: x\in B\}$. Then we have:
\begin{enumerate}
\item  $H\in \B(\mathcal{S}_{\Pi, t})$, i.e., $H$ is $(\Pi,t)$-constructible.
\item $B\cap H=\emptyset$.
\item $H$ is a  hyperplane of every $W\in \mathcal{C}$. 
\item $W\cap W'=H$ for distinct $W, W'\in\mathcal{C}$. In particular, $\{B\cap W: W\in\mathcal{C}\}$ is a partition of $B$ by (2).
\item The sets $B\cap W$ have equal size, where $W$ ranges over $\mathcal{C}$.
\item 
  $|\mathcal{C}|\leq \ell^{1/\epsilon'^2}$.
\item Let $W\in\mathcal{C}$ and $W'\in \mathcal{W}_{\Pi, k'', B}$ where $1\leq k''\leq m/2 $. Let $d$ be the codimension of $W\cap W'$ in $\sg{B}$. Suppose $k'\geq k''+dt+1$. Then either $|\mu_{W\cap W'}(B)-\mu_{W}(B)|\leq \ell^d\epsilon'$ or  $\mu_{W\cap W'}(B)=0$. The latter case  occurs only if $W\cap W'\subseteq H$.
\end{enumerate}
\end{theorem}

By Lemma~\ref{lem_twocases}, if $B$ is not pseudorandom against some subspace in $\mathcal{W}_{\Pi, k, B}$, i.e., $|\mu_W(B)-\mu(B)|> \epsilon$ for some $W\in \mathcal{W}_{\Pi, k, B}$, then $\mathcal{X}_{\Pi, k', B, \epsilon'}$ is nonempty.
Then we can find  a collection of subspaces $ \mathcal{C}=\{H+\F x: x\in B\}$, where  $H=\bigcap _{\chi\in \mathcal{X}_{\Pi, k', B, \epsilon'}} \ker(\chi)$.
Theorem~\ref{thm_struct}\,(7) then states that, by restricting to each  $W\in \mathcal{C}$,  $B$ becomes pseudorandom in the weaker sense that for small enough $k''$ and $W'\in \mathcal{W}_{\Pi, k'', B}$ satisfying $W\cap W'\not\subseteq H$, we have $\mu_{W\cap W'}(B)\approx\mu_{W}(B)$ (however, if $W\cap W'\subseteq H$, we would have $\mu_{W\cap W'}(B)=0$). 

Moreover, the set $H$ is constructible (Theorem~\ref{thm_struct}\,(1)) and is a common hyperplane of $W\in\mathcal{C}$ (Theorem~\ref{thm_struct}\,(3)). It is also the common intersection of $W\in\mathcal{C}$ (Theorem~\ref{thm_struct}\,(4)). So the subspaces $W\in\mathcal{C}$ form a ``sunflower'' where $H$ is the ``kernel''. The set $B$ is evenly distributed on the ``leaves'' of the sunflower (Theorem~\ref{thm_struct}\,(2), (4) and (5)). Finally, we have an upper bound for the number of leaves (Theorem~\ref{thm_struct}\,(6)).

\begin{proof}[Proof of Theorem~\ref{thm_struct}]

(1): For $\chi\in \mathcal{X}_{\Pi, k', B, \epsilon'}$, we will find a subspace $H_\chi\subseteq\sg{B}$ satisfying $H_\chi\in \B(\mathcal{S}_{\Pi, t})$ and $H\subseteq H_\chi\subseteq \ker(\chi)$.
Then $H=\bigcap_{\chi\in \mathcal{X}_{\Pi, k', B, \epsilon'}} H_\chi$ since
\[
H\subseteq \bigcap_{\chi\in \mathcal{X}_{\Pi, k', B, \epsilon'}} H_\chi\subseteq \bigcap_{\chi\in \mathcal{X}_{\Pi, k', B, \epsilon'}} \ker(\chi)=H.
\]
As   $\B(\mathcal{S}_{\Pi, t})$ is closed under intersection by Lemma~\ref{lem_constset}\,(2), we would have $H\in \B(\mathcal{S}_{\Pi, t})$.
So it remains to find $H_\chi$ for $\chi\in \mathcal{X}_{\Pi, k', B, \epsilon'}$.

Consider arbitrary $\chi\in \mathcal{X}_{\Pi, k', B, \epsilon'}$. 
As  $\ker(\chi)\in \mathcal{W}_{\Pi, k', B}$, there exist $x=(x_1,\dots,x_{k'})\in S^{k'}$, $s\in\N^+$, $B_1,\dots,B_s\in \Pi_{x}^{(k')}$ and $\tau_1,\dots,\tau_s\in\M_{k',1}$ such that $\ker(\chi)=\bigcup_{i=1}^s \tau_i(B_i)$. 
We may assume the sets $\tau_i(B_i)$ are disjoint by Lemma~\ref{lem_constset}\,(2). 
For $i\in [s]$, let $B'_i$ be the block in $\Pi^{(2k')}$ satisfying $B_i=\{y\in S^{k'}: (x,y)\in B'_i\}$, which uniquely exists since $B_i\in \Pi_{x}^{(k')}$.
Let $B^*$ be the block in $\Pi^{(k')}$ containing $x$.
For $i\in [s]$ and $z\in B^*$, let $B_{i,z}=\{y\in S^{k'}: (z,y)\in B'_i\}\in \Pi_z^{(k')}$, so that $B_{i,x}=B_i$.
Finally, choose $u\in B\setminus\ker(\chi)$, and let $B^{**}$ be the block in $\Pi^{(k'+1)}$ containing $(x,u)$.

We know $\ker(\chi)=\bigcup_{i=1}^s \tau_i(B_i)$.
The idea is showing that for each $z\in B^*$, there exists a character $\chi_z\in \mathcal{X}_{\Pi, k', B, \epsilon'}$ such that $\ker(\chi_z)=\bigcup_{i=1}^s \tau_i(B_{i,z})$, and $\chi_x=\chi$. Then we can choose $H_\chi=\bigcap_{z\in\B^*} \ker(\chi_z)$. To achieve this, we need the following claim.
\begin{claim}\label{claim_constant}
 We have:
\begin{enumerate}[(a)]
\item For distinct $i,j\in [s]$  and $z\in B^*$, $\tau_i(B_{i,z})\cap \tau_j(B_{j,z})=\emptyset$.
\item For $i\in [s]$, $|\tau_i(B_{i,z})|$ is independent of $z\in B^*$.
\item For $i,j\in [s]$, the number of $(a,b)\in B_{i,z}\times B_{j,z}$ satisfying $\tau_i(B_{i,z})-\tau_j(B_{j,z})\in \bigcup_{k=1}^s \tau_k(B_{k,z})$ is independent of $z\in B^*$.
\item For all $(z,w)\in B^{**}$ where $z\in S^k$ and $w\in S$, we have $w\not\in \bigcup_{i=1}^s\tau_i(B_{i,z})$.
\item For $i\in [s]$ and $c\in \F$, $|B\cap (\tau_i(B_{i,z})+c w)|$ is independent of $(z,w)\in\B^{**}$. 
\end{enumerate}
\end{claim}

\begin{proof}[Proof of Claim~\ref{claim_constant}]

 (a): Let $i,j\in [s]$ such that $i\neq j$. By Property (P2) of $\Pi$, the number of $(a,b)\in S^{k'}\times S^{k'}$ satisfying $(z,a)\in B_i'$, $(z,b)\in B_j'$ and $\tau_i(a)=\tau_j(b)$ is independent of $z\in B^*$.
For each $z$, this number equals zero iff $\tau_i(B_{i,z})\cap \tau_j(B_{j,z})=\emptyset$. As $\tau_i(B_i)\cap \tau_j(B_j)=\emptyset$, we have $\tau_i(B_{i,z})\cap \tau_j(B_{j,z})=\emptyset$ for $z\in B^*$.
 
 (b): Let $i\in [s]$. 
 For $d\in\N^+$, let $B_{i,d}'$ be the set  of $(z,y)\in B_i'$ satisfying 
 \[
 \#\{y'\in S^{k'}: (z,y')\in B_i' \text{ and } \tau_i(y')=\tau_i(y)\}=d.
 \]
 By Lemma~\ref{lem_closedness},  we have $B_{i,d}'\in \B(\Pi^{(2k')})$ for $d\in \N^+$.
 As $B_i'\in \Pi^{(2k')}$ and  $B_{i,d}'\subseteq B_i'$  for $d\in \N^+$, we have $B_i'=B_{i,d_0}'$ for some $d_0\in\N^+$.
 This means for all $z\in B^*$, the map $\tau_i|_{B_{i,z}}: B_{i,z}\to \tau_i(B_{i,z})$ is a $d_0$-to-1 map.
 By Property (P2) of $\Pi$, $|B_{i,z}|$ is independent of $z\in B^*$.
 So $|\tau_i(B_{i,z})|=|B_{i,z}|/d_0$ is also independent of $z\in B^*$. 
 
 (c): Let $i,j\in [s]$. 
 For $k\in [s]$, define 
 \[
 T_k=\{(z,a,b,c)\in (S^{k'})^4: (z,c)\in B_k' \text{ and } \tau_i(a)-\tau_j(b)=\tau_k(c)\}.
 \] 
 Then $T_k\in\B(\Pi^{(4k')})$ for $k\in [s]$.
 Let $T=\bigcup_{k\in [s]} T_k\in \B(\Pi^{(4k')})$.
 By Lemma~\ref{lem_closedness} and Property (P2) of $\Pi$, the number
 \[
 \#\left\{(a,b)\in S^{k'}\times S^{k'}: 
 \begin{array}{l}
  (z,a)\in B_i', (z,b)\in B_j', and\\
  \exists\,c\in S^{k'}\,  (z,a,b,c)\in T
 \end{array}
 \right\}
 \]
 is independent of $z\in B^*$.
This is precisely the number of $(a,b)\in B_{i,z}\times B_{j,z}$ satisfying $\tau_i(B_{i,z})-\tau_j(B_{j,z})\in \bigcup_{k=1}^s \tau_k(B_{k,z})$.

(d): 
By Property (P2) of $\Pi$, for $i\in [s]$, the number of $a\in S^{k'}$ satisfying $(z,a)\in B_i'$ and $w=\tau_i(a)$ is independent of $(z,w)\in B^{**}$.
This number equals zero for all $i\in [s]$ iff $w\not\in \bigcup_{i=1}^s\tau_i(B_{i,z})$. As $(x,u)\in B^{**}$ and $u\not\in \bigcup_{i=1}^s\tau_i(B_{i})$, we have $w\not\in \bigcup_{i=1}^s\tau_i(B_{i,z})$ for $(z,w)\in B^{**}$.

 (e): Let $i\in [s]$ and $c\in\F$. 
 By Lemma~\ref{lem_closedness} and Property (P2) of $\Pi$, the number
 \[
 \#\{a\in   B: \exists\, y\in S^{k'} \text{ satisfying } (z,y)\in B_i' \text{ and } a=\tau_i(y)+cw\}
 \]
 is independent of $(z,w)\in B^{**}$. This number is precisely $|B\cap (\tau_i(B_{i,z})+cw)|$.
\end{proof}

Let $W_z=\bigcup_{i=1}^s \tau_i(B_{i,z})$ for $z\in B^*$. Note $W_x=\ker(\chi)$. 
By Claim~\ref{claim_constant}\,(a) and (b), $|W_z|$ is independent of $z\in B^*$.
By Claim~\ref{claim_constant}\,(c), for $i,j\in [s]$, the number of $(a,b)\in B_{i,z}\times B_{j,z}$ satisfying $\tau_i(B_{i,z})-\tau_j(B_{j,z})\in W_z$ is independent of $z\in B^*$. This number attains the maximum possible value $|B_{i,z}|\times |B_{j,z}|$ for all $i,j\in [s]$ iff $W_z-W_z\subseteq W_z$, or equivalently, $W_z$ is an abelian subgroup of $\sg{B}$. As $W_{x}=\ker(\chi)$ is a hyperplane of $\sg{B}$ and $\F$ is a prime field, we know $W_z$ is a hyperplane of $\sg{B}$  for all $z\in B^*$.

Choose $u_z\in S$ for each $z\in B^*$ such that $(z,u_z)\in B^{**}$ and $u_x=u$ (such $u_z$ exists by Property (P2) of $\Pi$). 
For $z\in B^*$, we know $u_z\not\in W_z$ and hence $W_z+\F u_z=\sg{B}$ by Claim~\ref{claim_constant}\,(d). For $z\in B^*$, let $\chi_z$ be the unique character in $\widehat{\sg{B}}$ satisfying $W_z\subseteq\ker(\chi_z)$ and $\chi_z(u_z)=\chi(u)$. In particular, $\chi_x=\chi$.
Note
\begin{align*}
\widehat{1_B}(\chi_z)&=\E_{a\in \sg{B}}\left[1_B(a)\overline{\chi_z(a)}\right]=|\sg{B}|^{-1}\left(\sum_{c\in \F} |B\cap (W_z+cu_z)| \cdot \overline{\chi_z(c u_z)}\right)\\
&=|\sg{B}|^{-1}\left(\sum_{c\in \F} |B\cap (W_z+cu_z)| \cdot \overline{\chi(c u)}\right).
\end{align*}
By Claim~\ref{claim_constant}\,(b) and (e), $|B\cap (W_z+cu_z)|$ is independent of $z\in B^*$.
It follows that $|\widehat{1_B}(\chi_z)|=|\widehat{1_B}(\chi)|\geq \epsilon'$ for all $z\in B^*$.
Also note  $W_z=\bigcup_{i=1}^s \tau_i(B_{i,z})\in\mathcal{W}_{\Pi, k', B}$. 
By definition, we have $\chi_z\in  \mathcal{X}_{\Pi, k', B, \epsilon'}$ for  $z\in B^*$.

For $i\in [s]$, let 
\[
U_i=\left\{a\in S^{k'}:  \begin{array}{l}
 \forall\,z\in B^* ~\exists\, b\in S^{k'}~ \exists\,j\in [s]\\
   (z,b)\in B_j' \text{ and }\tau_i(a)=\tau_j(b)
 \end{array}
 \right\}.
\]
We have $U_i\in \B(\Pi^{(k')})$ by Lemma~\ref{lem_closedness}.
Let $H_\chi:=\bigcup_{i=1}^s \tau_i(U_i)$. Then $H_\chi\in \B(\mathcal{S}_{\Pi, k'})$.
On the other hand, note 
\[
H_\chi=\bigcap_{z\in B^*} W_z=\bigcap_{z\in B^*}\ker(\chi_z)\subseteq \sg{B}.
\]
By Theorem~\ref{thm_smallgen}, we have $H_\chi\subseteq \sg{B}\subseteq t(\F B)$. Note $t(\F B)\in \B(\mathcal{S}_{\Pi, t})$.
Then $H_\chi=H_\chi\cap t(\F B)\in \B(\mathcal{S}_{\Pi, t})$ by Lemma~\ref{lem_constset}\,(2).
As $\chi_z\in  \mathcal{X}_{\Pi, k', B, \epsilon'}$ for  $z\in B^*$, $\chi_x=\chi$, and $x\in B^*$, we have
\[
H=\bigcap_{\chi\in \mathcal{X}_{\Pi, k', B, \epsilon'}} \ker(\chi)\subseteq H_\chi\subseteq \ker(\chi),
\]
as desired.

(2): Assume to the contrary that $B\cap H\neq\emptyset$.  Note $B\cap H\in \B(\Pi^{(1)})$ by (1) and Lemma~\ref{lem_constset}\,(1). As $B\in \Pi^{(1)}$, we   have   $B\subseteq H$ and hence $\sg{B}\subseteq H$.
But this is impossible as $H=\bigcap _{\chi\in \mathcal{X}_{\Pi, k', B, \epsilon'}} \ker(\chi)$ is a proper subspace of $\sg{B}$.

(3): Every $W\in \mathcal{C}$ is of the form $H+\F x$ for some $x\in B$. So either  $H=W$ or $H$ is a hyperplane of $W$. The former case is impossible since $B\cap H=\emptyset$ by (2). 

(4): By (3), we have $\dim W=\dim W'=\dim H+1$. Note $H\subseteq W\cap W'\subseteq W$. As $H$ is a hyperplane of $W$,  either $W\cap W'=H$ or $W\cap W'=W$. The latter case is impossible as $W\neq W'$ and $\dim W=\dim W'$. 

(5): The set
$T:=\{(x,y)\in B\times B: y\in  H+\F x\}$
is  in $\B(\Pi^{(2)})$ by (1) and  Lemma~\ref{lem_closedness}.
By Property (P2), the cardinality of $T_x:=\{y\in B: (x,y)\in T\}$ is constant when $x$ ranges over $B$. But $T_x$ is precisely $B\cap W$ where  $W=H+\F x$. So $B\cap W$ have equal size when $W$ ranges over $\mathcal{C}$.

(6): 
Let $d$ be the codimension of $H$ in $\sg{B}$.
We have $\sum_{\chi\in  \mathcal{X}_{\Pi, k', B, \epsilon'}} |\widehat{1_B}(\chi)|^2\leq 1$ by  Parseval's identity.
As $|\widehat{1_B}(\chi)|\geq \epsilon'$ for all $\chi\in \mathcal{X}_{\Pi, k', B, \epsilon}$, we have $ |\mathcal{X}_{\Pi, k', B, \epsilon'}|\leq 1/\epsilon'^2$ and hence $d\leq 1/\epsilon'^2$. 
For every two distinct subspaces $W, W'\in\mathcal{C}$, $W/H$ and $W'/H$ are distinct one-dimensional subspaces of $\sg{B}/H$ by (3) and (4). 
The claim follows by noting that the number of one-dimensional subspaces of $\sg{B}/H$ equals $(\ell^d-1)/(\ell-1)\leq \ell^d\leq \ell^{1/\epsilon'^2}$.

(7):   
By (1), (3) and Lemma~\ref{lem_csubspace}, we have $W\in \B(\mathcal{S}_{\Pi_x, t+1})$ for some $x\in S$ .
We also have $W'\in \B(\mathcal{S}_{\Pi_y, k''})$ for some $y\in S^{k''}$ by definition.
Then $W\cap W'\in  \B(\mathcal{S}_{\Pi_{x,y}, t+1})$ by Lemma~\ref{lem_refine} and Lemma~\ref{lem_constset}.

Assume $|\mu_{W\cap W'}(B)-\mu_{W}(B)|>\ell^d\epsilon'$.
We want to prove  $W\cap W'\subseteq H$, which implies $\mu_{W\cap W'}(B)=0$ by (2).
Let $G_1$ (resp. $G_2$) be the subgroup of the characters $\chi\in\widehat{\sg{B}}$ vanishing on $W\cap W'$ (resp. $W$).
We have $G_2\subseteq G_1$ and $|G_1|=|\sg{B}|/|W\cap W'|=\ell^d$.
Note $|\mu_{W\cap W'}(B)-\mu_{W}(B)|= |\sum_{\chi\in G_1\setminus G_2} \widehat{1_B}(\chi)|$ (cf. Equation~\eqref{eq_fourier} in the proof of Lemma~\ref{lem_twocases}).
So there exists $\chi^*\in G_1\setminus G_2$ such that $|\widehat{1_B}(\chi^*)|\geq |G_1|^{-1}\ell^d\epsilon'=\epsilon'$.
As $\chi^*\in G_1$, we have $W\cap W'\subseteq\ker(\chi^*)$.
The codimension of $W\cap W'$ in $\ker(\chi^*)$ is $d-1$. So by Lemma~\ref{lem_csubspace} and Theorem~\ref{thm_smallgen},   $\ker(\chi^*)$ is $((\Pi_{x,y})_{x'}, (t+1)+(d-1)t)$-constructible for some $x'\in S^{(d-1)t}$.  
We know $k'\geq k''+1+(d-1)t$ and $k'\geq (t+1)+(d-1)t$.
It follows that $\ker(\chi^*)\in  \mathcal{W}_{\Pi, k', B}$.
So $\chi^*\in \mathcal{X}_{\Pi, k', B, \epsilon'}$.  

Let $G$  be the be the subgroup of the characters $\chi\in\widehat{\sg{B}}$ vanishing on $H\subseteq W$. As $H$ is a hyperplane of $W$, $G_2$ is a subgroup of $G$ of corank one. As $\chi^*\in  \mathcal{X}_{\Pi, k', B, \epsilon'}$ and $H=\bigcap _{\chi\in \mathcal{X}_{\Pi, k', B, \epsilon'}} \ker(\chi)$, we have $H\subseteq  \ker(\chi^*)$ and hence $\chi^*\in G$. As $\chi^*\not\in G_2$, we have $G=\langle G_2, \chi^*\rangle$, which is equivalent to $H=W\cap \ker(\chi^*)$. So $W\cap W'\subseteq W\cap \ker(\chi^*)=H$, as desired.
\end{proof}
 
 \subsection{Reducing the Density of $B$}
 
We adopt the following notation throughout this subsection.
 \begin{definition}
 For $B\in\Pi^{(2)}$ and $x\in S$, where $\Pi$ is a linear $m$-scheme on $S$ with $m\geq 2$, define  $B_x:=\{y\in S: (x,y)\in B\}\in \Pi_x^{(1)}$, called the \emph{$x$-fiber} of $B$.
 \end{definition}
 
Suppose $B\in\Pi^{(1)}$ and $B^*\in \Pi^{(2)}$ satisfy $B^{*}\subseteq (B\times B)\setminus \Delta_B$, where $\Delta_B:=\{(x,x):x\in B\}$.
Let $B^{**}=\{(y,x): (x,y)\in B'\}$. Then $B^{**}\in \Pi^{(2)}$, $|B^{**}|=|B^{*}|$, and $B^{**}\neq B^{*}$.
So for $x\in B$, we have two distinct blocks $B^{*}_x, B^{**}_x\in\Pi_x^{(1)}$ of equal size contained in $B\setminus\{x\}$.

We want to prove that replacing $B$ by $B^{*}_x$   reduces the density by at least a constant factor. To achieve this,  we need to  restrict to a subspace in which $B$ is pseudorandom. This leads to the following definition.

\begin{definition}\label{defi_twocases}
Let $\Pi$ be a linear $m$-scheme on $S$. Let $B\in \Pi^{(1)}$, $k\in\N^+$, and $0<\epsilon<1$. Let $t=\lfloor\frac{3}{2\mu(B)}\rfloor+1$, $k'=k(t+2)$ and $\epsilon'=\epsilon/\ell^k$. Suppose $m\geq 2k'$.
\begin{enumerate}
\item  If $\mathcal{X}_{\Pi, k', B, \epsilon'}=\emptyset$, define $H_{\Pi, k, B, \epsilon}:=\emptyset$ and $W_{\Pi, k, B, \epsilon}(x):=\sg{B}$ for $x\in B$.
\item If $\mathcal{X}_{\Pi, k', B, \epsilon'}\neq\emptyset$, define  $H_{\Pi, k, B,\epsilon}:=\bigcap _{\chi\in \mathcal{X}_{\Pi, k', B, \epsilon'}} \ker(\chi)$  and $W_{\Pi, k, B, \epsilon}(x):=H_{\Pi, k, B,\epsilon}+\F x$ for $x\in B$.
\end{enumerate}
\end{definition}

Then $B$ is ``pseudorandom'' within  $W_{\Pi, k, B, \epsilon}(x)$ for each $x\in B$ in the following sense.

\begin{lemma}\label{lem_pseudorandom}
Let $\Pi, m, B, k, k', \epsilon, \epsilon'$ be as in Definition~\ref{defi_twocases}. Suppose   $m\geq 4k'$. Let  $x\in B$ and $W=W_{\Pi, k, B, \epsilon}(x)$.
Let $W'\in  \mathcal{W}_{\Pi, k, B}$ such that $W\cap W'\not\subseteq H_{\Pi, k, B,\epsilon}$ and the codimension of $W\cap W'$ in $\sg{B}$ is bounded by $k$. 
Then  $|\mu_{W\cap W'}(B)-\mu_W(B)|\leq \epsilon$.
\end{lemma}

\begin{proof}
If $\mathcal{X}_{\Pi, k', B, \epsilon'}=\emptyset$, we have $W=\sg{B}$ by Definition~\ref{defi_twocases} and hence $W\cap W'=W'\in  \mathcal{W}_{\Pi, k, B}$. In this case, the lemma follows from Lemma~\ref{lem_twocases}.
On the other hand, if $\mathcal{X}_{\Pi, k', B, \epsilon'}\neq\emptyset$, the lemma follows from Theorem~\ref{thm_struct}\,(7). 
\end{proof}

We also need the following simple lemma.

\begin{lemma}\label{lem_simple}
Let $\Pi$ be a linear $m$-scheme and  let $B\in\B(\Pi^{(1)})$. Suppose $|B|\geq N$ for some $N>0$. Then there exists  $B'\in \B(\Pi^{(1)})$ such that $B'\subseteq B$ and either of the following holds:
\begin{enumerate}
\item $N/2\leq |B'|\leq N$.
\item   $B'\in\Pi^{(1)}$ and $|B'|\geq  N$.
\end{enumerate} 
\end{lemma}
\begin{proof}
Let $T=\{B'\in \Pi^{(1)}: B'\subseteq B\}$.
If there exists $B'\in T$ satisfying $|B'|\geq N/2$ then either (1) or (2) holds, depending on whether $|B'|\leq N$.
So assume $|B'|<N/2$ for all $B'\in T$.
Choose a minimal subset $T'\subseteq T$ such that $\sum_{B'\in T'} |B'|\geq N/2$, which is always possible as  $\sum_{B'\in  T} |B'| = |B|\geq N$.
Let  $B''=\bigcup_{B'\in T'} B'$. Then $B''\in \B(\Pi^{(1)})$, $B''\subseteq B$, and $|B''|\geq N/2$. 
By the minimality of $T'$ and the fact $|B'|<N/2$ for $B'\in T$, we  have $|B''|\leq N$. So (1) holds.
\end{proof}

Next, we prove the following lemma.

\begin{lemma}\label{lem_rednormal}
Let $\Pi$ be a strongly antisymmetric linear $m$-scheme on $S$.
Let $B\in \Pi^{(1)}$,   $K>1$, $t= \lfloor\frac{3K}{2\mu(B)}\rfloor+1$, $k\in\N^+$, $k'=k(t+2)$, $0<\epsilon<1$ and $\epsilon'=\epsilon/\ell^k$. 
Suppose $m\geq 4k'+2$, $k\geq 2t$, and $|B|>K$.
Let $x\in B$ and $W=W_{\Pi, k, B, \epsilon}(x)$, and suppose $\epsilon\leq \mu_W(B)/2$.
Then one of the following holds:
\begin{enumerate}
\item  There exist $y\in B$ and  $B' \in \B(\Pi_{x,y}^{(1)})$  such that $B'\subseteq B$ and  $\ell^{-(1/\epsilon'^2)}|B|/K\leq |B'|\leq |B|/K$.
\item  There exist $B'\in \Pi_x^{(1)}$ and $y\in B'$ such that $B'\subseteq W$ and for the subspace $W'=W_{\Pi_x, k, B', \epsilon}(y)$, we have $B'\subseteq B$, $|B'|\geq |B|/K$, and $\mu_{W'}(B')\leq (\mu_{W}(B)+\epsilon)/2$.
\item  There exists $B^*\in \Pi^{(2)}$ contained in $B\times B$ such that $B^*$ and $B^{**}:=\{(y,x): (x,y)\in B^*\}$ satisfy the following conditions:
\begin{enumerate}[(3a)]
\item $B^*_x, B^{**}_x\subseteq B\cap W$, $B^*_x\neq B^{**}_x$, and $|B^*_x|=|B^{**}_x|\geq |B|/K$.
\item $W_{\Pi_x, k, B^*_x, \epsilon}(y)=\sg{B^*_x}$ for   $y\in B^*_x$.
\item $W_{\Pi_x, k, B^{**}_x, \epsilon}(y)=\sg{B^{**}_x}$ for  $y\in B^{**}_x$, and $\dim\sg{B^{**}_x}=\dim\sg{B^{*}_x}$.
\item $B\cap \sg{B^{*}_x}\cap \sg{B^{**}_x}=\emptyset$. 
\item $x\in \sg{B^{*}_x}$.
\item $|(B\cap \sg{B^{*}_x})\setminus B^{*}_x|, |(B\cap \sg{B^{**}_x})\setminus B^{**}_x|< \ell^{-(1/\epsilon'^2)}|B|/K$.
\end{enumerate}
\end{enumerate}
\end{lemma}

Case~(1) basically implies Lemma~\ref{lem_densecase} (for large enough $K$). So we are done in this case.
We give some intutions for Case~(2) and (3). For simplicity, assume $W=\sg{B}$. 

As mentioned above, the set $B\setminus\{x\}$ contains two distinct blocks $B^*_x, B^{**}_x\in\Pi_x^{(1)}$ of equal size as subsets. Assume $B^*_x$ and $B^{**}_x$ are dense enough in $\sg{B}$. 
Let $H=\sg{B^*_x}\cap \sg{B^{**}_x}$.
Then it is not hard to show that either $H\in\{\sg{B^*_x}, \sg{B^{**}_x}\}$ or $B^*_x\cap H=B^{**}_x\cap H=\emptyset$.
Otherwise, we could intersect  either $B^*_x$ or $B^{**}_x$   with $H$ and obtain a  nonempty proper subset of $B^*_x$ or $B^{**}_x$ that is in $\B(\Pi_x^{(1)})$, contradicting $B^*_x, B^{**}_x\in\Pi_x^{(1)}$. 

Assume $H=\sg{B^*_x}=\sg{B^{**}_x}$. As $B$ is pseudorandom within $W$, we have
$\mu_H(B)\approx \mu_W(B)$. As $B^*_x, B^{**}_x\subseteq B$ and $B^*_x\cap B^{**}_x=\emptyset$, we have
\[
\min\{\mu(B^*_x), \mu(B^{**}_x)\} =\min\{\mu_H(B^*_x), \mu_H(B^{**}_x)\} \leq  \mu_H(B)/2 \approx \mu_W(B)/2.
\]
A similar argument shows that this holds more generally when $H\in\{\sg{B^*_x}, \sg{B^{**}_x}\}$.
Moreover, in order to apply Lemma~\ref{lem_rednormal} repeatedly, we will actually
find some $B'\in \Pi_x^{(1)}$ and a subspace $W'$ such that $\mu_{W'}(B')\lessapprox \mu_W(B)/2$ and $B'$ is pseudorandom within $W'$. This is captured by Case~(2).
 
 Unfortunately, this argument does not seem to work in the case $B^*_x\cap H=B^{**}_x\cap H=\emptyset$, and we do not know how to rule out this case. This exceptional case is described by Case~(3) above and will be addressed later via a  careful analysis (see Lemma~\ref{lem_ruleout} below). 
 
\begin{proof}[Proof of  Lemma~\ref{lem_rednormal}]

As the proof is long and technical, we divide it into several steps.

\textbf{Step 1: Finding $B'\in \Pi_x^{(1)}$ such that $B'\subseteq B\cap W$ and $|B'|\geq |B|/K$.}

If $\mathcal{X}_{\Pi, k', B, \epsilon'}=\emptyset$, we have $W=\sg{B}=t(\F B)\in \B(\mathcal{S}_{\Pi,t})$ by Theorem~\ref{thm_smallgen}.
And if $\mathcal{X}_{\Pi, k', B, \epsilon'}\neq\emptyset$, we have $H_{\Pi, k, B, \epsilon}\in \B(\mathcal{S}_{\Pi,t})$ by Theorem~\ref{thm_struct}\,(1) and hence $W=H_{\Pi, k, B, \epsilon}+\F x\in \B(\mathcal{S}_{\Pi_x,t+1})$. 
In either case, we have $W\in \B(\mathcal{S}_{\Pi_x,t+1})$.
So $B\cap W\in \B(\Pi_x^{(1)})$ by Lemma \ref{lem_constset}\,(1).

Also note $|B\cap W|\geq  \ell^{-1/\epsilon'^2}  |B|$. This is trivial if $\mathcal{X}_{\Pi, k', B, \epsilon'}=\emptyset$, in which case $W=\sg{B}$.
Otherwise it follows from Theorem~\ref{thm_struct}\,(4), (5) and (6).
 
If $|B\cap W|\leq |B|/K$ then Case~(1) holds. So assume $|B\cap W|\geq |B|/K$.
Applying Lemma~\ref{lem_simple} to $\Pi_x$, $B\cap W\in  \B(\Pi_x^{(1)})$ and $N=|B|/K$, we see that there exists $B'\in \B(\Pi_x^{(1)})$ such that $B'\subseteq B\cap W$ and either of the following holds:
\begin{enumerate}[(a)]
\item $|B|/(2K)\leq |B'|\leq |B|/K$.
\item $B'\in\Pi_x^{(1)}$ and $|B'|\geq |B|/K$.
\end{enumerate}
In the former case, Case~(1) holds and we are done. So assume  $B'\in\Pi_x^{(1)}$ and $|B'|\geq |B|/K$. 

\emph{In the following, we choose $B'$ such that $\dim \sg{B'}$ is minimized subject to the conditions $B'\in\Pi_x^{(1)}$, $B'\subseteq B\cap W$ and $|B'|\geq |B|/K$.}

\textbf{Step 2: Proving (3a).} 

Choose $B^*$ to be the block in $\Pi^{(2)}$ satisfying $B'=\{y\in S: (x,y)\in B^*\}$. 
Then $B^*_x=B'\subseteq B\cap W$.
Let $B^{**}=\{(y,x): (x,y)\in B^*\}\in \Pi^{(2)}$ as in (3).
As $x\in B$ and $B'\subseteq B$, we have $B^*\cap (B\times B)\neq \emptyset$ and hence $B^*, B^{**}\subseteq B\times B$. So $B^{**}_x\subseteq B$.
By Property (P2), we have $|B^*_x|=|B^*|/|B|$ and $|B^{**}_x|=|B^{**}|/|B|$. As $|B^*|=|B^{**}|$, we have $|B^*_x|=|B^{**}_x|=|B'|\geq |B|/K>1$.
As $|B^*_x|>1$, there exists $y\in B^*_x$ different from $x$. Then $(x,y)\in B^*$ and $(y,x)\in B^{**}$. The map $(a,b)\mapsto (b,a)$ sends $B^*$ to $B^{**}$ and does not fix $(x,y)$.
By strong antisymmetry of $\Pi$,  $B^*\neq B^{**}$. So $B^*_x\neq B^{**}_x$.

It remains to prove $B^{**}_x\subseteq W$. Let $y\in B^{*}_x$ and $z\in B^{**}_x$.
Then  $y\in B\cap W$ and $(y,x), (x,z)\in B^{**}$. 
 We want to prove $z\in W$.
As $x\in W$, we have $y\in H_{\Pi, k, B, \epsilon}+\F x$ by Theorem~\ref{thm_struct}\,(2) and (3). 
By Theorem~\ref{thm_struct}\,(1), $H_{\Pi, k, B, \epsilon}\in \B(\mathcal{S}_{\Pi,t})$.
So there exist $B_1,\dots,B_s\in \Pi^{(t)}$ and $\tau_1,\dots,\tau_s\in\M_{t,1}$ such that $H_{\Pi, k, B, \epsilon}=\bigcup_{i=1}^s \tau_i(B_i)$.
Then $(y,x)$ satisfies the relation $y-c x\in \bigcup_{i=1}^s \tau_i(B_i)$ for some $c\in\F^\times$.  As $(y,x)$ and $(x,z)$ are in the same block of $\Pi^{(2)}$, we have 
$x-c z\in \bigcup_{i=1}^s \tau_i(B_i)=H_{\Pi, k, B, \epsilon}$. So $z\in c^{-1}x+H_{\Pi, k, B, \epsilon}=W$, as desired.

\textbf{Step 3: Proving (3b).}

For   $y\in B^*_x$, the definition only guarantees  $W_{\Pi_x, k, B^*_x, \epsilon}(y)$ to be a subspace of $\sg{B^*_x}$.
However, as $B_x^*=B'$ is chosen such that $\dim \sg{B'}$ is minimized, we actually have $W_{\Pi_x, k, B^*_x, \epsilon}(y)=\sg{B^*_x}$.
The reason is that if  $W_{\Pi_x, k, B^*_x, \epsilon}(y)\subsetneq \sg{B^*_x}$, we could use $H_{\Pi_x, k, B^*_x, \epsilon}$ to find another block $B''\in\Pi_x^{(1)}$ satisfying the same conditions that are satisfied by $B'$ and additionally $\dim \sg{B''}<\dim \sg{B'}$ holds. But this   contradicts the minimality of $\dim \sg{B'}$.

To formalize this argument, we first prove the following claim.
\begin{claim}\label{claim_smaller}
Suppose $W'$ is a $(\Pi_x, k)$-constructible subspace of $W$ such that  $B\cap W'\neq\emptyset$ and the codimension of $W'$ in $\sg{B}$ is at most $k$.
Then either Case~(1) of Lemma~\ref{lem_rednormal} holds or  there exists $B''\in \Pi_x^{(1)}$ satisfying $B''\subseteq B\cap W'$ and $|B''|\geq |B|/K$.  Moreover, in the latter case, $\dim \sg{B^*_x}\leq \dim \sg{B''}\leq \dim W'$.  
\end{claim}

\begin{claimproof}[Proof of Claim~\ref{claim_smaller}]
As $W'$ is a $(\Pi_x, k)$-constructible and its codimension in $\sg{B}$ is at most $k$, we have $W'\in \mathcal{W}_{\Pi, k, B}$ by definition.
As $B\cap W'\neq\emptyset$, we have $W'\not\subseteq  H_{\Pi, k, B, \epsilon}$ by Theorem~\ref{thm_struct}\,(2).
By Lemma~\ref{lem_pseudorandom},
\[
\mu_{W'}(B) =\mu_{W\cap W'}(B)\geq \mu_W(B)-\epsilon\geq \mu_W(B)/2.
\]
Therefore 
\begin{align*}
|B\cap W'|&\geq \mu_W(B)/2 \cdot |W'|\geq  \mu_W(B)|W|/(2  \ell^{k}) = |B\cap W|/(2   \ell^{k}) \\
&\geq |B|/(2K \ell^k)
\geq \ell^{-(1/\epsilon'^2)}|B|/K.
\end{align*}
As $W'$ is a $(\Pi_x, k)$-constructible, we have $B\cap W'\in \B(\Pi_{x}^{(1)})$ by Lemma \ref{lem_constset}\,(1).
If $|B\cap W'|\leq |B|/K$ then Case~(1) holds. So assume  $|B\cap W'|\geq |B|/K$.
Applying Lemma~\ref{lem_simple} to $\Pi_x$, $B\cap W'$ and $N=|B|/K$, we see that there exists $B''\in \B(\Pi_x^{(1)})$ such that $B''\subseteq B\cap W'\subseteq B\cap W$ and either of the following holds:
\begin{enumerate}[(a)]
\item $|B|/(2K)\leq |B''|\leq |B|/K$.
\item $B''\in\Pi_x^{(1)}$ and $|B''|\geq |B|/K$.
\end{enumerate}
In the former case, Case~(1) holds. So assume $B''\in\Pi_x^{(1)}$ and $|B''|\geq |B|/K$.
We then have $\dim \sg{B^*_x}=\dim \sg{B'}\leq \dim \sg{B''}\leq \dim W'$ by the minimality of $\dim \sg{B'}$ subject to the conditions $B'\in\Pi_x^{(1)}$, $B'\subseteq B\cap W$ and $|B'|\geq |B|/K$.
\end{claimproof}

Now we prove (3b). If $\mathcal{X}_{\Pi_x, k',  B^*_x, \epsilon'}=\emptyset$, then (3b) holds by Definition~\ref{defi_twocases}. So assume $\mathcal{X}_{\Pi_x, k',  B^*_x, \epsilon'}\neq \emptyset$.

We want to prove $W_{\Pi_x, k, B^*_x, \epsilon}(y)=\sg{B^*_x}$  for $y\in \sg{B^*_x}$. Let $H=H_{\Pi_x, k, B^*_x, \epsilon}\subseteq \sg{B^*_x}$.  As $H$ is a hyperplane of $W_{\Pi_x, k, B^*_x, \epsilon}(y)$ for $y\in \sg{B^*_x}$, it suffices to prove $\dim H\geq \dim \sg{B^*_x}-1$.

Let $W'= H+\F x\subseteq W$. 
We check that $W'$ satisfies the conditions in  Claim~\ref{claim_smaller}:
\begin{itemize}
\item As $x\in B\cap W'$, we have $B\cap W'\neq\emptyset$. 
\item By Theorem~\ref{thm_struct}\,(1),   $H\in \B(\mathcal{S}_{\Pi_x,t}) $. As $\{x\}\in \Pi_x^{(1)}$, we have  $W'=H+\F x\in  \B(\mathcal{S}_{\Pi_x,t+1})\subseteq  \B(\mathcal{S}_{\Pi_x,k})$. So $W'$ is $(\Pi_x, k)$-constructible.
\item Let $y\in B^*_x$ and $W''=W_{\Pi_x, k, B^*_x, \epsilon}(y)=H+\F y$.  Then $W''\in  \B(\mathcal{S}_{\Pi_{x,y},t+1})$ and hence $B^*_x\cap W''\in\Pi_{x,y}^{(1)}$ by  Lemma \ref{lem_constset}\,(1).
By Theorem~\ref{thm_struct}\,(4), (5) and (6), we have 
\[
|B^*_x\cap W''|\geq  \ell^{-1/\epsilon'^2}  |B^*_x|\geq \ell^{-1/\epsilon'^2}|B|/K.
\]
If $|B^*_x\cap W''|\leq |B|/K$ then Case~(1) holds. So assume $|B^*_x\cap W''|\geq |B|/K$. Then the codimension of $W''$ in $\sg{B}$ is at most $\log_\ell (K/\mu(B))\leq t\leq k-1$.
As $\dim W'=\dim(H+\F x)\geq \dim (H+\F y)-1=\dim W''-1$, the codimension of $W'$ in $\sg{B}$ is at most $k$.
 \end{itemize}

By Claim~\ref{claim_smaller}, either Case~(1) holds or  $\dim W'\geq \dim\sg{B^*_x}$.
Assume the latter case occurs. 
As $W'=H+\F x$, we have $\dim H\geq \dim\sg{B^*_x}-1$, as desired.

\textbf{Step 4: Proving (3c).}

Let $y\in B^{**}_x$ and $W'= W_{\Pi_x, k, B^{**}_x, \epsilon}(y)=H'+\F y$. We want to prove  $\dim  \sg{B^{*}_x}=\dim \sg{B^{**}_x}$ and $W'=\sg{B^{**}_x}$ (if neither Case~(1) nor (2) holds).

We first show that either Case~(2) holds or $\dim W' \leq \dim \sg{B^*_x}$.
By Lemma~\ref{lem_pseudorandom} and (3b), we have
\begin{equation}\label{eq_density}
\mu(B^*_x)\leq \mu_{\sg{B^*_x}}(B)=\mu_{W\cap \sg{B^*_x}}(B)\leq \mu_W(B)+\epsilon.
\end{equation}
Note $|B^{**}_x|=|B^{*}_x|$. If $\dim W'> \dim \sg{B^*_x}$, then
\[
\mu_{W'}(B^{**}_x)\leq \frac{|B^*_x|}{|W'|}\leq \frac{|B^*_x|}{2| \sg{B^*_x}|}= \mu(B^*_x)/2\stackrel{\eqref{eq_density}}{\leq} (\mu_W(B)+\epsilon)/2 
\]
and hence Case~(2) holds. So   assume $\dim W' \leq \dim \sg{B^*_x}$.

Consider the case $\mathcal{X}_{\Pi_x, k',  B^{**}_x, \epsilon'}=\emptyset$. Then $W'=\sg{B^{**}_x}$   by Definition~\ref{defi_twocases}. 
 So $\dim \sg{B^{**}_x}\leq \dim \sg{B^{*}_x}$. Exchanging $B^{*}_x$ and $B^{**}_x$ in the above proof shows $\dim \sg{B^{*}_x}\leq \dim \sg{B^{**}_x}$ (or Case~(2) holds). So $\dim \sg{B^{*}_x}=\dim \sg{B^{**}_x}$ and (3c) holds.
 
 So assume $\mathcal{X}_{\Pi_x, k',  B^{**}_x, \epsilon'}\neq \emptyset$.  Let  $H'=H_{\Pi_x, k, B^{**}_x, \epsilon}$.
In Step~3, we have shown that $H_{\Pi_x, k, B^{*}_x, \epsilon}+\F x$ satisfies the conditions in Claim~\ref{claim_smaller}. The same proof with $B^*_x$ replaced by $B^{**}_x$ also shows $W'=H'+\F x$ satisfies the conditions in Claim~\ref{claim_smaller}.  Applying Claim~\ref{claim_smaller} to $W'$, we see either Case~(1) holds or  $\dim W'\geq \dim\sg{B^*_x}$.
Assume the latter. Then  $\dim W'=\dim\sg{B^*_x}$ as we already know  $\dim W' \leq \dim \sg{B^*_x}$.

It remains to prove $W'=\sg{B_x^{**}}$ (or Case~(1) or (2) holds). Assume $W'$ is a proper subspace of $\sg{B_x^{**}}$.
Then $|B^{**}_x\cap W'|\leq |B^{**}_x|/2 $ by Theorem~\ref{thm_struct}\,(4) and (5).
Therefore 
\[
\mu_{W'}(B^{**}_x)\leq \frac{|B^{**}_x|}{2|W'|}=\frac{|B^*_x|}{2| \sg{B^*_x}|}= \mu(B^*_x)/2\stackrel{\eqref{eq_density}}{\leq} (\mu_W(B)+\epsilon)/2 
\]
and hence Case~(2) holds.

\textbf{Step 5: Proving (3d).}

Assume neither Case~(1) nor (2) holds. Also assume $B\cap \sg{B^{*}_x}\cap \sg{B^{**}_x}\neq \emptyset$. We will derive a contradiction. The intuition is that since $B\cap \sg{B^{*}_x}\cap \sg{B^{**}_x}\neq \emptyset$, we could find another block $B''\subseteq B\cap \sg{B^{*}_x}\cap \sg{B^{**}_x}$ and use it to contradict the minimality of $\dim \sg{B'}$.

Let $W'=\sg{B^{*}_x}\cap \sg{B^{**}_x}\subseteq W$. We check that $W'$ satisfies the conditions in Claim~\ref{claim_smaller}:
\begin{itemize}
\item By assumption, we have $B\cap W'=B\cap \sg{B^{*}_x}\cap \sg{B^{**}_x}\neq\emptyset$.
\item As $|B^{*}_x|=|B^{**}_x|\geq |B|/K$, both $\mu(B^{*}_x)$ and $\mu(B^{**}_x)$ are at least $\mu_{\sg{B}}(B^{*}_x)=\mu_{\sg{B}}(B^{**}_x)\geq \mu(B)/K$. By Theorem~\ref{thm_smallgen}, $\sg{B^{*}_x}$ and $\sg{B^{**}_x}$ are $(\Pi_x,t)$-constructible and hence $(\Pi_x,k)$-constructible.
By Lemma~\ref{lem_constset}\,(2), $W'=\sg{B^{*}_x}\cap \sg{B^{**}_x}$ is also $(\Pi_x,k)$-constructible.
\item As $\mu_{\sg{B}}(B^{*}_x)=\mu_{\sg{B}}(B^{**}_x)\geq \mu(B)/K$, both  $\sg{B^{*}_x}$ and $\sg{B^{**}_x}$ have codimension at most  $\log_\ell (K/\mu(B))\leq t$ in $\sg{B}$. So $W'=\sg{B^{*}_x}\cap \sg{B^{**}_x}$ has codimension at most $2t\leq k$ in $\sg{B}$.
\end{itemize}
By Claim~\ref{claim_smaller}, we have  $\dim W'\geq \dim \sg{B^*_x}$.  Then  $W'=\sg{B^{*}_x}=\sg{B^{**}_x}$ by (3c).
As $W'$ is  $(\Pi_x, k)$-constructible and its codimension in $\sg{B}$ is at most $k$, we have $W'\in \mathcal{W}_{\Pi, k, B}$ by definition.
As $B^*_x, B^{**}_x\subseteq B$, $B^*_x\cap B^{**}_x=\emptyset$, and $|B^*_x|=|B^{**}_x|$, we have
 \[
\mu_{W'}(B^*_x)\leq \mu_{W'}(B)/2=\mu_{W\cap W'}(B)/2\leq (\mu_W(B)+\epsilon)/2.
 \]
where the last inequality holds by Lemma~\ref{lem_pseudorandom}.
 By (3b), $W'=W_{\Pi_x, k, B^*_x, \epsilon}(y)$ for $y\in B^*_x$. So Case~(2) holds, contradicting our assumption.

\textbf{Step 6: Addressing the case $\mathcal{X}_{\Pi, k', B, \epsilon'}=\emptyset$.}

Let $W'=\sg{B^{*}_x}\cap \sg{B^{**}_x}$ as in Step 5.
Suppose $\mathcal{X}_{\Pi, k', B, \epsilon'}=\emptyset$. Then $H_{\Pi, k, B, \epsilon}=\emptyset$ and hence $W'\not\subseteq  H_{\Pi, k, B, \epsilon}$.
By Lemma~\ref{lem_pseudorandom}, we have 
\[
\mu_{W'}(B)=\mu_{W\cap W'}(B)\geq \mu_{W}(B)-\epsilon>0.
\]
So $B\cap \sg{B^{*}_x}\cap \sg{B^{**}_x}\neq \emptyset$. By Step 5,  either Case~(1) or (2) holds.

From now on, we assume $\mathcal{X}_{\Pi, k', B, \epsilon'}\neq \emptyset$, so that $H_{\Pi, k, B, \epsilon}$ is a hyperplane of $W$.

\textbf{Step 7: Proving (3e).}

We want to prove $x\in \sg{B^{*}_x}$. While we do not know if this holds in general, we will show that it can be achieved by replacing $B^*_x=B'$ with another block and $\dim \sg{B'}$ remains minimized.

Let $H=H_{\Pi, k, B, \epsilon}$, which is a hyperplane of $W$ by Step 6.
Let $W'=(\sg{B^*_x}\cap H)+\F x\subseteq W$. 
We check that $W'$ satisfies the conditions in Claim~\ref{claim_smaller}:
\begin{itemize}
\item As $x\in B\cap W'$, we have $B\cap W'\neq\emptyset$.
\item By Theorem~\ref{thm_struct}\,(1),   $H\in \B(\mathcal{S}_{\Pi,t})\subseteq \B(\mathcal{S}_{\Pi_x,t})$. We also know $\sg{B^*_x}\in \B(\mathcal{S}_{\Pi_x, t})$ (see Step~5).
So $\sg{B^*_x}\cap H\in \B(\mathcal{S}_{\Pi_{x},t})$ by Lemma~\ref{lem_constset}\,(2). As $\{x\}\in\Pi_x^{(1)}$, we have $W'=(\sg{B^*_x}\cap H)+\F x\in \B(\mathcal{S}_{\Pi_{x},t+1})\subseteq  \B(\mathcal{S}_{\Pi_{x},k})$, i.e., $W'$ is $(\Pi_x, k)$-constructible. 
\item We already know the codimension of $\sg{B^*_x}$ in $\sg{B}$ is at most $k$ (see Step 5). As $\dim W'=\dim(\sg{B^*_x}\cap H)+1=\dim \sg{B^*_x}$, the codimension of $W'$ in $\sg{B}$ is at most $k$.
\end{itemize}

By Claim~\ref{claim_smaller},   either (1) holds  or there exists $B''\in \Pi_x^{(1)}$ satisfying $B''\subseteq B\cap W'\subseteq B\cap W$, $|B''|\geq |B|/K$, and $\dim \sg{B^*_x}=\dim \sg{B'}\leq \dim \sg{B''}\leq \dim W'$.
Assume the latter case occurs.
We know $\dim W'=  \dim \sg{B^*_x}$. So $ \dim \sg{B''}=\dim \sg{B'}$ and $\sg{B''}=W'\ni x$. Replacing $B'$ by $B''$  preserves the conditions $B'\in\Pi_x^{(1)}$, $B'\subseteq B\cap W$ and $|B'|\geq |B|/K$, and $\dim \sg{B'}$ remains minimized subject to these conditions. So, by replacing $B'$ with $B''$, we may assume $x\in \sg{B'}=\sg{B^*_x}$, and all the results proved above still hold.

 \textbf{Step 8: Proving (3f).}
 
 We prove that $|(B\cap \sg{B^{**}_x})\setminus B^{**}_x|< \ell^{-(1/\epsilon'^2)}|B|/K$ (or either Case~(1) or (2) holds). The proof for the claim $|(B\cap \sg{B^{*}_x})\setminus B^{*}_x|< \ell^{-(1/\epsilon'^2)}|B|/K$ is the same.
 
 Let $T=(B\cap \sg{B^{**}_x})\setminus B^{**}_x$. We have $T\in \B(\Pi_x^{(1)})$ by Lemma~\ref{lem_constset}\,(1).
 If $\ell^{-(1/\epsilon'^2)}|B|/K\leq |T|\leq |B|/K$ then (1) holds. 
 
 Now assume $|T|\geq |B|/K$. Applying Lemma~\ref{lem_simple} to $\Pi_x$, $T$ and $N=|B|/K$, we see that there exists $T'\in \B(\Pi_x^{(1)})$ such that $T'\subseteq T\subseteq \sg{B^{**}_x}$ and either of the following holds:
\begin{enumerate}[(a)]
\item $|B|/(2K)\leq |T'|\leq |B|/K$.
\item $T'\in\Pi_x^{(1)}$ and $|T'|\geq |B|/K$.
\end{enumerate}
In the former case, Case~(1) holds and we are done. So assume  $T'\in\Pi_x^{(1)}$ and $|T'|\geq |B|/K$. 

Let $H=H_{\Pi_x, k, T', \epsilon}$. Define a subspace $W'$ of $W$ as follows: If  $\mathcal{X}_{\Pi_x, k', T', \epsilon'}=\emptyset$, let $W'=\sg{T'}$.
Otherwise let $W'=H+\F x$. Note that in either case, we have $\dim W'=\dim W_{\Pi_x, k, T', \epsilon}(y)$ for $y\in T'$.
Using the fact $|T'|\geq |B|/K$, it can be shown that either Case~(1) holds or $W'$ satisfies the conditions in  Claim~\ref{claim_smaller}.\footnote{We omit the proof. If $\mathcal{X}_{\Pi_x, k', T', \epsilon'}=\emptyset$, we have $W'=\sg{T'}$ and the proof is similar to that in Step~5. Otherwise, we have $W'=H+\F x$ and the proof is similar to that in Step~3.}  Assume the latter case occurs. 

By Claim~\ref{claim_smaller}, we have $\dim W'\geq \dim \sg{B^{*}_x}=\dim \sg{B^{**}_x}$. Therefore, $\dim W_{\Pi_x, k, T', \epsilon}(y)\geq \dim \sg{B^{**}_x}$ for $y\in T'$. As $T'\subseteq  \sg{B^{**}_x}$, we have $W_{\Pi_x, k, T', \epsilon}(y)=\sg{B^{**}_x}$ for $y\in T'$.

As  $T', B^{**}_x\subseteq  B\cap \sg{B^{**}_x}$ and $T'\cap B^{**}_x=\emptyset$, we have
\[
\min\{\mu_{\sg{B^{**}_x}}(T'), \mu(B^{**}_x)\}  \leq \mu_{\sg{B^{**}_x}}(B)/2 =  \mu_{W\cap \sg{B^{**}_x}}(B)/2 \leq (\mu_W(B)+\epsilon)/2
\]
where the last inequality holds by Lemma~\ref{lem_pseudorandom} and (3c). So Case~(2) holds.

Therefore, we may assume $|T|< \ell^{-(1/\epsilon'^2)}|B|/K$, since otherwise Case~(1) or (2) holds.
\end{proof}

To address Case (3) of Lemma~\ref{lem_rednormal}, we need the following lemma.

\begin{lemma}\label{lem_constantsize}
Let $\Pi$ be a  linear $m$-scheme on $S$, where $m\geq 2$.
Let $B\in \Pi^{(1)}$ and $B', B''\in \Pi^{(2)}$ such that $B', B''\subseteq B\times B$. Let $x\in B$, $\mu=|B'_x|/|\sg{B}|$ and $t=\lfloor\frac{3}{2\mu}\rfloor+1$.
Suppose $m\geq t+1$.
For $y\in B$, we have (1) $|B''_y|=|B''_x|$, (2) either $B''_y\subseteq \sg{B'_y}$ or $B''_y\cap \sg{B'_y}=\emptyset$, and (3) $B''_y\subseteq \sg{B'_y}$ iff $B''_x\subseteq \sg{B'_x}$.
\end{lemma}

\begin{proof}
The first claim holds by Property (P2) of $\Pi$.
Also note $|B'_y|=|B'_x|$ and hence $\mu(B'_y)\geq \mu$ for  $y\in B$ by Property (P2) of $\Pi$.

Consider $y\in B$.
By Theorem~\ref{thm_smallgen}  and the fact $\mu(B'_y)\geq \mu$, we have   $\sg{B'_y}\in \B(\mathcal{S}_{\Pi_y,t})$. By Lemma~\ref{lem_constset}\,(1),  we have $B''_y\cap \sg{B'_y}\in \B(\Pi_y^{(1)})$.
As $B''_y\in\Pi_y^{(1)}$, $B''_y\cap \sg{B'_y}$ equals either $B''_y$ or the empty set. This proves (2).

Let $T$ be the set of $y\in B$ satisfying $B''_y\subseteq \sg{B'_y}$, or equivalently,  $B''_y\cap \sg{B'_y}\neq \emptyset$.
Then $T=\bigcup_{\tau\in\M_{t,1}} T_\tau$, where
\[
T_\tau=\left\{y\in B: \begin{array}{l}\exists\,  z=(z_1,\dots,z_t)\in B^t  \text{ such that }\\  (y,z_1)\in B', \dots, (y, z_t)\in B', (y,\tau(z))\in B''
\end{array}
\right\}.
\]
So $T\in \B(\Pi^{(1)})$ by Lemma~\ref{lem_closedness}. As $B\in \Pi^{(1)}$,   either $T=B$ or $T=\emptyset$. This proves (3).
\end{proof}

Next, we prove that Case (3) of Lemma~\ref{lem_rednormal} actually implies Case (1):

\begin{lemma}\label{lem_ruleout}
Under the notations of Lemma~\ref{lem_rednormal}, further assume $k\geq (t+1)\log_\ell (K/\mu(B))$.
If Case~(3) of Lemma~\ref{lem_rednormal} holds, then Case~(1) also holds.
\end{lemma}

\begin{proof}
Assume Case~(3) of Lemma~\ref{lem_rednormal} holds.  
By Theorem~\ref{thm_smallgen} and (3a),   $\sg{B^*_x}=t(\F B^*_x)$. As $x\in \sg{B^*_x}$ by (3e), we may choose $x_1,\dots,x_t\in B^*_x$ such that $x$ is in the linear span of $x_1,\dots,x_t$ over $\F$.
As $|B^*_x|\geq |B|/K>1$ and $\{x\}, B_x^*\in \Pi_x^{(1)}$, we have $x\not\in B^*_x$ and hence $x\not\in\{x_1,\dots,x_t\}$.

As $x_1,\dots,x_t\in B^*_x$, we have   $x\in B^{**}_{x_i}$ for $i\in [t]$.
Define  $L:= \sg{B^*_x}\cap \bigcap_{i=1}^t \sg{B^{**}_{x_i}}\subseteq W$.
 Then $x\in   L$. 
  For $i\in [t]$, define $T_i:=(B\cap \sg{B^{**}_{x_i}})\setminus B^{**}_{x_i}$.
By (3f) and Lemma~\ref{lem_constantsize}, we have $|T_i|=|(B\cap \sg{B^{**}_x})\setminus B^{**}_x|<N$  for $i\in [t]$, where $N:=\ell^{-(1/\epsilon'^2)}|B|/K$.

Let $U$ be the union of all the blocks $B'\in \Pi_x^{(1)}$ satisfying $B'\subseteq B$ and $|B'|\leq N$.
Then $U\in \B(\Pi_x^{(1)})$.
If $|U|\geq 2N$, then as $U$ is a union of blocks of cardinality at most $N$,  there exists a subset $U'\subseteq U$ in $\B(\Pi_x^{(1)})$ such that $N\leq |U'|\leq 2N\leq |B|/K$.
Then  Case~(1) of Lemma~\ref{lem_rednormal} holds.
So we may assume $|U|\leq 2N$.

Note $\sg{B^{*}_{x}}\in \B(\mathcal{S}_{\Pi_x,t})$ and $\sg{B^{**}_{x_i}}\in \B(\mathcal{S}_{\Pi_{x_i},t})$ for $i\in [t]$.
 By Lemma~\ref{lem_constset}\,(2), we have $L\in \B(\mathcal{S}_{\Pi_{x, x_1,\dots,x_t},t})$.
As $|B^{*}_{x}|\geq |B|/K$ and $|B^{**}_{x_i}|=|B^{**}_{x}|\geq |B|/K$ for $i\in [t]$,
the codimension of $L$ in $\sg{B}$ is at most $(t+1)\log_\ell (K/\mu(B))\leq k$.
By definition, we have $L\in \mathcal{W}_{\Pi, k, B}$.
 
As $x\in L$, we have $L\not\subseteq H_{\Pi, k, B,\epsilon}$. So $|\mu_{L}(B)-\mu_W(B)|\leq \epsilon$ by Lemma~\ref{lem_pseudorandom}. Therefore
\begin{align*}
|B\cap L|&\geq (\mu_W(B)-\epsilon)\cdot |L|\geq \mu_W(B)/2\cdot |L|\geq  \mu_W(B)/2\cdot \ell^{-k}\cdot |W|\\
&=|B\cap W|/(2\ell^k) \geq |B|/(2\ell^k K) \geq (t+2) N >|U|+\sum_{i=1}^t |T_i|.
\end{align*}
Therefore, there exists $y\in  B\cap L$ such that  $y\not\in U$ and $y\not\in T_i$ for $i\in [t]$. Fix such $y$.

For $i\in [t]$, we have $y\in (B\cap\sg{B^{**}_{x_i}}) \setminus T_i=B^{**}_{x_i}$, i.e., $(x_i, y)\in B^{**}$, or equivalently, $(y,x_i)\in B^*$. So $x_i\in B^*_y\subseteq \sg{B^*_y}$ for $i\in [t]$.
As $x$ is in the linear span of $x_1,\dots,x_t$, we have $x\in \sg{B^*_y}$.

Let $B'$ (resp. $B''$) be the block in $\Pi^{(2)}$ containing $(x,y)$ (resp. $(y,x)$). 
Then $|B'|=|B''|$.
Note that $B'_x$ is the block in $\Pi_x^{(1)}$ containing $y$.
As $y\not\in U$, we have $|B'_x|>N$.
Then $|B''_y|>N$ since $|B''_y|=|B''|/|B|=|B'|/|B|=|B'_x|$ by Property (P2) of $\Pi$.

Note that $B''_y$ is the block  in $\Pi_y^{(1)}$ containing $x$. As  $x\in \sg{B^*_y}$, we have $B''_y\subseteq B\cap \sg{B^*_y}$ by  Lemma~\ref{lem_constantsize}.
By (3f) and Lemma~\ref{lem_constantsize},  we have $|(B\cap \sg{B^{*}_y})\setminus B^{*}_y|=|(B\cap \sg{B^{*}_x})\setminus B^{*}_x|<N$. As $|B''_y|>N$, we must have $B''_y= B^{*}_y$.
So $x\in  B^{*}_y$, or equivalently, $y\in B^{**}_x$.

As $y\in L\subseteq \sg{B^*_x}$, we have $y\in \sg{B^*_x}\cap \sg{B^{**}_x}$. So $B\cap \sg{B^*_x}\cap \sg{B^{**}_x}\neq\emptyset$, contradicting (3d).
\end{proof}

Combining Lemma~\ref{lem_rednormal} with Lemma~\ref{lem_ruleout}, we obtain the following lemma, which will be used to prove  Lemma~\ref{lem_densecase} in the next Subsection.

\begin{lemma}\label{lem_rednormal2} 
Let $\Pi$ be a strongly antisymmetric linear $m$-scheme on $S$.
Let $B\in \Pi^{(1)}$,   $K>1$, $t= \lfloor\frac{3K}{2\mu(B)}\rfloor+1$, $k\in\N^+$, $k'=k(t+2)$, $0<\epsilon<1$ and $\epsilon'=\epsilon/\ell^k$. 
Suppose $m\geq 4k'+2$, $k\geq 2t$, $k\geq (t+1)\log_\ell (K/\mu(B))$, and $|B|>K$.
Let $x\in B$ and $W=W_{\Pi, k, B, \epsilon}(x)$, and suppose $\epsilon\leq \mu_W(B)/2$.
Then one of the following holds:
\begin{enumerate}
\item  There exist $y\in B$ and  $B' \in \B(\Pi_{x,y}^{(1)})$  such that $B'\subseteq B$ and  $\ell^{-(1/\epsilon'^2)}|B|/K\leq |B'|\leq |B|/K$.
\item  There exist $B'\in \Pi_x^{(1)}$ and $y\in B'$ such that $B'\subseteq W$ and for the subspace $W'=W_{\Pi_x, k, B', \epsilon}(y)$, we have $B'\subseteq B$, $|B'|\geq |B|/K$, and $\mu_{W'}(B')\leq (\mu_{W}(B)+\epsilon)/2$.
\end{enumerate}
\end{lemma}

 \subsection{Finishing the Proof of Lemma~\ref{lem_densecase}}
 
Lemma~\ref{lem_densecase} can be easily derived from the following lemma.

\begin{lemma}\label{lem_rep}
Let $\Pi$ be a strongly antisymmetric linear $m$-scheme on $S$.
Let $B\in \Pi^{(1)}$ and  $K\geq 4$.
Let $k=6K^4$, $r=\lceil \log K/\log\log K\rceil+1$, $\epsilon=(2/3)^{r-1}/3$ and $\epsilon'=\epsilon/\ell^k$.
Suppose $|B|\geq |\sg{B}|/K$, $|B|>K^2$ and $m\geq 120K^7+3r$.
Then there exist $x_1,\dots,x_{3r+1}\in B$ and $B'\in \B(\Pi_{x_1,\dots,x_{3r+1}}^{(1)})$  such that $B'\subseteq B$ and  
\[
\ell^{-(1/\epsilon'^2)}|B|/K^3\leq |B'|\leq \max\{K^{-1}, (2\gamma)^{r}\}\cdot |B|.
\]
where $\gamma:=4\left((r-1)/\log\ell\right)^{-1/20}$.
\end{lemma}

We prove Lemma~\ref{lem_rep} in two steps. First, we repeatedly apply Lemma~\ref{lem_rednormal2} $r$ times for some $r=\omega(1)$. If Case~(1) of Lemma~\ref{lem_rednormal2} ever occurs then we are done. So assume Case~(2) holds each time. This enables us to find a block $B'\in \Pi_{x_1,\dots,x_r}^{(1)}$ contained in $B$ and a subspace $W$ such that $\mu_W(B')\leq \exp(-r)$ and $B'$ is pseudorandom within $W$.

Let $B''=B'\cap W$. The condition  $\mu_W(B')\leq \exp(-r)$ implies $|B''+B''|\gg |B''|$ by the Freiman--Ruzsa Theorem.
In the second step of the proof, we use this condition to reduce the cardinality of $B''$ $r$ times  such that each time the cardinality is reduced by a superconstant factor. The ideas here are similar to those in the proof of Lemma~\ref{lem_shrinkweak}. However,  Lemma~\ref{lem_shrinkweak} itself is not sufficient here\footnote{Lemma~\ref{lem_shrinkweak} shows that, if $|B+B|\gg |B|$, we could find $B'\subseteq B$ with a mild lower bound for $\min\{|B'|, |B|/|B'|\}$.
However, in Lemma~\ref{lem_rep}, we need both a lower bound and a (good) upper bound for $|B|/|B'|$, and the latter does not follow from Lemma~\ref{lem_shrinkweak}. This upper bound is used in the analysis for the case that $|B|^2/|B+B|$ is small (see the proof of Lemma~\ref{lem_fastshrink} in Subsection~\ref{subsec_case3}).} and we need to augment it with the  Balog-Szemer\'edi-Gowers Theorem for linear $m$-schemes (Theorem~\ref{thm_bsg}).
 
\begin{proof}[Proof of Lemma~\ref{lem_rep}]
As mentioned above, the proof consists of two steps:

\textbf{Step 1: Finding a sparse block.}

We claim that for $0\leq i\leq r$, there exist $x_1,\dots,x_{i}\in B$ and $T(i)\subseteq B$ such that either of the following holds:
\begin{enumerate}[(a)]
\item  $T(i)\in \B(\Pi_{x_1,\dots,x_{i},x}^{(1)})$ for some $x\in B$, and $\ell^{-(1/\epsilon'^2)}|B|/K^2  \leq |T(i)|\leq |B|/K$.
\item  $T(i)\in  \Pi_{x_1,\dots,x_{i}}^{(1)}$,  $|T(i)|\geq |B|/K$, and there exists $x\in T(i)$ such that $\mu_W(T(i))\leq (2/3)^i$, where $W=W_{\Pi_{x_1,\dots,x_{i}}, T(i), k, \epsilon}(x)$.
\end{enumerate} 
We prove the claim by induction on $i$. For $i=0$, (b) of the claim holds by choosing $T(0)=B$.
Now assume $i>0$ and the claim holds for $i-1$.
Consider $x_1,\dots,x_{i-1}\in B$ and $T(i-1)\subseteq B$ satisfying the claim for $i-1$.

First assume (a) holds for $i-1$, i.e.,  $T(i-1)\in \B(\Pi_{x_1,\dots,x_{i-1},x}^{(1)})$ for some $x\in B$, and
\[
\ell^{-(1/\epsilon'^2)}|B|/K^2  \leq |T(i-1)|\leq |B|/K.
\]
Let $x_i=x$ and $x'\in B$. By Lemma~\ref{lem_refine}, $T(i-1)\in \B(\Pi_{x_1,\dots,x_{i},x'}^{(1)})$.
Choose $T(i)=T(i-1)$. Then (a) also holds for $i$.

So assume (b) holds for $i-1$, i.e.,  $T(i-1)\in  \Pi_{x_1,\dots,x_{i-1}}^{(1)}$, $|T(i-1)|\geq |B|/K>K$, and there exists $x\in T(i-1)$ such that $\mu_W(T(i-1))\leq (2/3)^{i-1}$, where $W=W_{\Pi_{x_1,\dots,x_{i-1}}, T(i-1), k, \epsilon}(x)$.
In particular, $\mu(T(i-1))\geq |T(i-1)|/|\sg{B}|\geq K^{-2}$.

Let $t=3K^3\geq \lfloor\frac{3K}{2\mu(T(i-1))}\rfloor+1$ and $k'=k(t+2)$. Then $m-(i-1)\geq 4k'+2$.
Applying Lemma~\ref{lem_rednormal2}  to $\Pi_{x_1,\dots,x_{i-1}}$ and $T(i-1)$, we see  that either of the following two cases  holds:
\begin{enumerate}
\item  There exist $y\in T(i-1)$ and  $B' \in \B(\Pi_{x_1,\dots,x_{i-1}, x,y}^{(1)})$  such that $B'\subseteq T(i-1)$ and  $\ell^{-(1/\epsilon'^2)}|T(i-1)|/K\leq |B'|\leq |T(i-1)|/K$.
\item  There exist $B'\in \Pi_{x_1,\dots,x_{i-1},x}^{(1)}$ and $y\in B'$ such that $B'\subseteq T(i-1)\cap W$, $|B'|\geq |T(i-1)|/K$, and $\mu_{W'}(B')\leq (\mu_{W}(T(i-1))+\epsilon)/2$, where  $W'=W_{\Pi_{x_1,\dots,x_{i-1},x}, k, B', \epsilon}(y)$.
\end{enumerate}
In Case (1), let $x_i=x$ and $T(i)=B'$, so that $T(i)\in  \B(\Pi_{x_1,\dots,x_{i}, y}^{(1)})$.
In this case, as $|B|/K\leq |T(i-1)| \leq |B|$, we have $\ell^{-(1/\epsilon'^2)}|B|/K^2  \leq |T(i)|\leq |B|/K$, i.e., (a) holds for $i$.

Now consider Case (2). Let $x_i=x$ and $T(i)=B'$, so that $T(i)\in  \Pi_{x_1,\dots,x_{i}}^{(1)}$.
We have $|T(i)|\geq |T(i-1)|/K\geq |B|/K^2$. If $|T(i)|\leq |B|/K$ then (a) holds for $i$. 
So assume  $|T(i)|\geq |B|/K$.  By Lemma~\ref{lem_pseudorandom}, we have
\[
\mu_{W'}(T(i))\leq (\mu_{W}(T(i-1))+\epsilon)/2\leq ((2/3)^{i-1}+\epsilon)/2\leq (2/3)^{i}.
\]
So (b) holds for $i$. This proves the claim.

If (a) holds for $i=r$ then the lemma holds. So we assume (b) holds for $i=r$, i.e., there exist $x_1,\dots,x_r\in B$, $T(r)\subseteq B$, and $x\in T(r)$ such that  $T(r)\in \Pi_{x_1,\dots,x_{r}}^{(1)}$, $|T(r)|\geq |B|/K$, and $\mu_W(T(r))\leq (2/3)^r$, where $W=W_{\Pi_{x_1,\dots,x_{r}}, T(r), k, \epsilon}(x)$.

In the following, let $\Pi':=\Pi_{x_1,\dots,x_{r}}$, $B':=T(r)$, and $W:=W_{\Pi_{x_1,\dots,x_{r}}, T(r), k, \epsilon}(x)$. 

 \textbf{Step 2: Reducing the cardinality of a block.}
 
 We claim that for $0\leq i\leq r$, there exist $y_1,\dots,y_{2i+1}\in B$ and $U(i)\subseteq B'\cap W$ such that $U(i)\in \B(\Pi'^{(1)}_{y_1,\dots,y_{2i+1}})$ and 
 \[
 \ell^{-(1/\epsilon'^2)}|B|/K^3 \leq |U(i)| \leq \max\{K^{-1}, (2\gamma)^{i}\}\cdot |B|.
 \]
 
We prove this claim by induction on $i$.
For $i=0$, let $y_1=x$ and $U(0)=B'\cap W$. By Theorem~\ref{thm_struct}\,(1), we have $H_{\Pi', B', k, \epsilon}\in \B(\mathcal{S}_{\Pi', t})$ and hence $W\in \B(\mathcal{S}_{\Pi'_{y_1}, t+1})$, where $t=3K^3$ is as chosen in Step 1.
By Lemma~\ref{lem_constset}\,(1), $B'\cap W\in \B(\Pi'^{(1)}_{y_1})$. By Theorem~\ref{thm_struct}\,(4), (5), (6) and the fact $|B'|\geq |B|/K$, we know $|U(0)|=|B'\cap W|\geq \ell^{-(1/\epsilon'^2)}|B|/K$. As $U(0)\subseteq B'\subseteq B$, we also have $|U(0)|\leq |B|$. So the claim holds for $i=0$.

Now assume $i>0$ and the claim holds for $i-1$. Consider $y_1,\dots,y_{2i-1}\in B$ and $U(i-1)\subseteq B'\cap W$ satisfying the claim for $i-1$.
If $|U(i-1)|\leq |B|/K$, then the claim holds for $i$ as well by choosing arbitrary $y_{2i},y_{2i+1}\in B$ and $U(i)=U(i-1)$. So assume $|U(i-1)|>|B|/K$. Then $|U(i-1)|\leq  (2\gamma)^{i-1}|B|$.

Next, we prove that the additive energy $E(U(i-1))$ of $U(i-1)$ is less than $\gamma|U(i-1)|^3$.
Assume to the contrary that $E(U(i-1))\geq |U(i-1)|^3$. By Theorem~\ref{thm_bsg}, there exist $y\in B$ and $U'\in \B(\Pi'^{(1)}_{y_1,\dots,y_{2i-1},y})$ such that $U'\subseteq U(i-1)$, $|U'|\geq \gamma|U(i-1)|/3$ and $|U'-U'|< 2^{17}\gamma^{-9}|U(i-1)|\leq 2^{19}\gamma^{-10} |U'|$.
It is well known that  $\frac{|A+A|}{|A|}\leq \left(\frac{|A-A|}{|A|}\right)^2$ for $A\subseteq V$ (see \cite[Exercise~6.5.15]{TV06}).
Therefore,
\[
|U'+U'|<(2^{19}\gamma^{-10})^2|U'|= 2^{38}\gamma^{-20}|U'|.
\]
By the Freiman--Ruzsa Theorem (Theorem~\ref{thm_fr}), we have 
\begin{equation}\label{eq_boundmu}
\mu(U')>\ell^{-2^{38}\gamma^{-20}}.
\end{equation}
As $|U'|\geq \gamma|U(i-1)|/3$ and $|U(i-1)|>|B|/K$, we have $|U'|\geq  \gamma|B|/(3K) \geq |B|/K^3$.
If $|U'|\leq |B|/K$, then the claim holds for $i$ by choosing $y_{2i}=y_{2i+1}=y$ and $U(i)=U'$. So assume $|U'|>|B|/K$.
Then the codimension of $\sg{U'}$ in $\sg{B}$ is at most $\log_\ell (K/\mu(B))\leq k$.
By Theorem~\ref{thm_smallgen}, $\sg{U'}=t(\F U')\in  \B(\mathcal{S}_{\Pi'_{y_1,\dots,y_{2i-1},y}, t})$.
By definition, $\sg{U'}\in \mathcal{W}_{\Pi, k, B}$.
So by Lemma~\ref{lem_pseudorandom}, we have
\[
\mu_{\sg{U'}}(B)\leq  \mu_W(B)+\epsilon\leq (2/3)^r+(2/3)^{r-1}/3=(2/3)^{r-1}\leq \ell^{-2^{38}\gamma^{-20}}
\]
where the last inequality holds since    $\gamma=4\left((r-1)/\log\ell\right)^{-1/20}$.
But this contradicts \eqref{eq_boundmu} since $\mu(U')\leq \mu_{\sg{U'}}(B)$. Therefore,
$E(U(i-1))<\gamma|U(i-1)|^3$.

For $z\in U(i-1)+U(i-1)$, denote by $\nu^+(z)$ the number of $(a,b)\in U(i-1)\times U(i-1)$ satisfying $a+b=z$.
Then 
\[
\sum_{z\in U(i-1)+U(i-1)} \nu^+(z)= |U(i-1)|^2 \quad\text{and}\quad
\sum_{z\in U(i-1)+U(i-1)} \nu^+(z)^2= E(U(i-1)).
\]
Note $|U(i-1)+U(i-1)|\leq |\sg{B}|\leq K|B| \leq K^2|U(i-1)|$.
Let 
\[
T=\{z\in U(i-1)+U(i-1): \nu^+(z)\geq |U(i-1)|/(2K^2)\}.
\]
Then 
\begin{align*}
\sum_{z\in T} \nu^+(z)
&=\sum_{z\in U(i-1)+U(i-1)} \nu^+(z)-\sum_{z\in (U(i-1)+U(i-1))\setminus T} \nu^+(z)\\
&\geq  |U(i-1)|^2-|U(i-1)+U(i-1)|\cdot |U(i-1)|/(2K^2)\\
&\geq |U(i-1)|^2/2.
\end{align*}
On the other hand, $\sum_{z\in T} \nu^+(z)^2\leq E(U(i-1))<\gamma|U(i-1)|^3$.
So there exists $z_0\in T$ such that 
\[
\nu^+(z_0)\leq \left(\sum_{z\in T} \nu^+(z)^2\right)/\left(\sum_{z\in T} \nu^+(z)\right)\leq 2\gamma|U(i-1)|\leq (2\gamma)^{i}|B|.
\]
As $z_0\in T$, we also have 
\[
\nu^+(z_0)\geq |U(i-1)|/(2K^2)\geq  \ell^{-(1/\epsilon'^2)}|B|/K^3.
\]
Choose $y_{2i}, y_{2i+1}\in U(i-1)$ such that $y_{2i}+y_{2i+1}=z_0$.
Let 
\[
U(i)=\{a\in U(i-1): \exists~ b\in U(i-1)\text{ such that }a+b=y_{2i}+y_{2i+1}\}.
\]
Then $U(i)\in \B(\Pi'^{(1)}_{y_1,\dots,y_{2i+1}})$ and $|U(i)|=\nu^+(z_0)$. So the claim holds for $i$. 

This proves the claim for all $0\leq i\leq r$. The lemma follows by choosing $i=r$. 
\end{proof}

Now we are ready to prove Lemma~\ref{lem_densecase}.

\begin{proof}[Proof of Lemma~\ref{lem_densecase}]
Let $K=\ell^{2(\log\log\log N)^{1/2}}$.
As $|B+B|/|B|\leq (\log\log\log N)^{1/2}$, we have $|\sg{B}|/|B|\leq K$ by the  Freiman--Ruzsa Theorem (Theorem~\ref{thm_fr}).
Pick the following parameters:
\begin{itemize}
\item $k=6K^4$,
\item $r=\lceil \log K/\log\log K\rceil+1$,  
\item $\gamma=4\left((r-1)/\log\ell\right)^{-1/20}$,   
\item $\epsilon=(2/3)^{r-1}/3$, and
\item $\epsilon'=\epsilon/\ell^{k}$.
\end{itemize}
Note   $\log \ell\leq (\log\log\log\log N)^{1/2}$ by assumption.

As we assume $N\geq c$ for a sufficiently large constant $c>0$, the conditions in Lemma~\ref{lem_rep} are satisfied.
By Lemma~\ref{lem_rep}, there exist $r'=\Theta(r)$, $x_1,\dots,x_{r'}\in B$, and $B'\in \B(\Pi_{x_1,\dots,x_{r'}}^{(1)})$ such that 
\[
|B|/|B'|\geq \min\{K, (2\gamma)^{-r}\}= 2^{\Theta(r \log\log\log\log N)}
\]
and 
\[
|B|/|B'|\leq K^3\ell^{1/\epsilon'^2}=K^3\ell^{\ell^{12K^4}/\epsilon^2}\leq 2^{2^{2^{(\log\log\log N)^{1/2+o(1)}}}}\leq 2^{\sqrt{\log N}},
\]
 as desired.
\end{proof}

\end{document}